\documentclass[10pt,reqno]{amsart}
\usepackage[foot]{amsaddr}

\usepackage[utf8]{inputenc}
\usepackage[english]{babel}

\usepackage[left=2cm,right=2cm,top=2cm,bottom=2cm]{geometry}
\usepackage{enumitem}
\setlist[itemize]{leftmargin=2em}

\usepackage{graphics}
\usepackage{float}
\usepackage{color}

\numberwithin{equation}{section}

\usepackage{amsmath}	
\usepackage{amssymb}	
\usepackage{mathtools}	

\usepackage[capitalise,noabbrev]{cleveref}
\crefname{equation}{Eq.\!\!}{Eq.\!\!}
\crefname{definition}{D.\!\!}{D.\!\!}
\crefname{lemma}{L.\!\!}{L.\!\!}
\crefname{proposition}{P.\!\!}{P.\!\!}
\crefname{theorem}{T.\!\!}{T.\!\!}
\crefname{corollary}{C.\!\!}{C.\!\!}
\crefname{example}{E.\!\!}{E.\!\!}
\crefname{remark}{R.\!\!}{R.\!\!}
\crefname{assumption}{A.\!\!}{A.\!\!}
\crefname{problem}{P.\!\!}{P.\!\!}

\newtheorem{definition}{Definition}[section]
\newtheorem{lemma}[definition]{Lemma}

\newtheorem{theorem}[definition]{Theorem}
\newtheorem{corollary}[definition]{Corollary}

\newtheorem{assumption}[definition]{Assumption}
\newtheorem{problem}[definition]{Problem}

\usepackage{ifthen}
\newcommand{\ifNotEmpty}[2]{\ifthenelse{\equal{#1}{}}{}{#2}}

\newcounter{repeatCharCounter}
\newcommand{\repeatChar}[2]{\setcounter{repeatCharCounter}{0}\whiledo{\value{repeatCharCounter}<#1}{#2\stepcounter{repeatCharCounter}}}

\renewcommand{\phi}{\varphi}
\renewcommand{\ell}{l}
\renewcommand{\epsilon}{\varepsilon}
\renewcommand{\*}{\hspace{-0.15em}\cdot\hspace{-0.15em} }	

\newcommand{\cleq}{\lesssim}
\newcommand{\cgeq}{\gtrsim}
\newcommand{\ceq}{\eqsim}

\renewcommand{\(}{\bigg(}
\renewcommand{\)}{\bigg)}
\renewcommand{\[}{\bigg[}
\renewcommand{\]}{\bigg]}

\newcommand{\set}[1]{\{#1\}}

\newcommand{\Set}[2]{\{#1\,|\,#2\}}

\newcommand{\N}{\mathbb{N}}
\newcommand{\Z}{\mathbb{Z}}

\newcommand{\R}{\mathbb{R}}
\newcommand{\C}{\mathbb{C}}

\newcommand{\complexI}{\mathbf{i}}

\newcommand{\Ball}[3][]{\mathrm{Ball}_{#1}(#2,#3)}

\newcommand{\Bubbles}[1]{\Omega_{#1}}

\newcommand{\closureN}[1]{\overline{#1}}

\newcommand{\inflateN}[3][]{{#2}^{#3}_{#1}}
\newcommand{\inflate}[3][]{(#2)^{#3}_{#1}}

\newcommand{\cardN}[1]{\# #1}

\newcommand{\meas}[2][]{|#2|_{#1}}
\newcommand{\measB}[2][]{\bigg|#2\bigg|_{#1}}

\newcommand{\diam}[2][]{\mathrm{diam}_{#1}(#2)}

\newcommand{\dist}[3][]{\mathrm{dist}_{#1}(#2,#3)}

\newcommand{\supp}[1]{\mathrm{supp}(#1)}

\newcommand{\restrictN}[2]{#1|_{#2}}
\newcommand{\restrict}[2]{(#1)|_{#2}}

\newcommand{\fDef}[3]{#1: #2 \longrightarrow #3}
\newcommand{\fDefB}[5][]{
\ifNotEmpty{#1}{#1:}
\left\{
\begin{array}{ccc}
#2 & \longrightarrow & #3 \\
#4 & \longmapsto & #5
\end{array}
\right.
}

\newcommand{\case}[4][]{
\left\{
\begin{array}{cl}
#2 & \text{if} \,\, #3 \\
#4 & \ifthenelse{\equal{#1}{}}{\text{else}}{\text{if} \,\, #1}
\end{array}
\right.}

\newcommand{\identity}{\mathrm{id}}
\newcommand{\charFunc}[1]{\mathbb{I}_{#1}}
\newcommand{\kronecker}[1]{\delta_{#1}}

\newcommand{\Landau}[1]{\mathcal{O}(#1)}

\let\oldprime\prime

\renewcommand{\prime}[2]{(#2)^{\repeatChar{#1}{\oldprime}}}

\newcommand{\diffN}[2]{#2^{(#1)}}
\newcommand{\diff}[2]{(#2)^{(#1)}}

\let\oldpartial\partial
\newcommand{\partialN}[3][]{\oldpartial_{#2}^{\ifNotEmpty{#1}{(#1)}} #3}
\renewcommand{\partial}[3][]{\oldpartial_{#2}^{\ifNotEmpty{#1}{(#1)}}(#3)}

\newcommand{\DN}[3][]{\mathrm{D}_{#1}^{#2} #3}
\newcommand{\D}[3][]{\mathrm{D}_{#1}^{#2}(#3)}

\newcommand{\LaplaceN}[2][]{\Delta_{#1} #2}

\newcommand{\I}[4][]{\int\displaylimits_{#2}^{#1} #3 \,\mathrm{d}#4}

\newcommand{\abs}[1]{|#1|}

\newcommand{\seminorm}[2][]{|#2|_{#1}}

\newcommand{\norm}[2][]{\|#2\|_{#1}}

\newcommand{\skalar}[3][]{\langle#2,#3\rangle_{#1}}

\newcommand{\bilinear}[3][]{#1(#2,#3)}

\renewcommand{\ker}[1]{\mathrm{ker}(#1)}

\newcommand{\ran}[1]{\mathrm{ran}(#1)}

\newcommand{\rank}[1]{\mathrm{rank}(#1)}

\newcommand{\dual}[1]{(#1)^{\oldprime}}

\newcommand{\mvemph}[1]{\boldsymbol{#1}}

\newcommand{\spanN}[1]{\mathrm{span}\,#1}

\newcommand{\dimN}[1]{\mathrm{dim}\,#1}
\renewcommand{\dim}[1]{\mathrm{dim}(#1)}

\newcommand{\Schwartz}[1][]{\mathcal{S}\ifNotEmpty{#1}{(#1)}}
\newcommand{\SchwartzO}[1]{\mathcal{S}^{#1}_0}

\newcommand{\Ck}[2]{C^{#1}(#2)}

\newcommand{\CkO}[2]{C^{#1}_0(#2)}

\newcommand{\lp}[2]{\ell^{#1}(#2)}
\newcommand{\Lp}[2]{L^{#1}(#2)}
\newcommand{\LpLoc}[2]{L^{#1}_{\mathrm{loc}}(#2)}

\newcommand{\Hk}[2]{H^{#1}(#2)}

\newcommand{\HkLoc}[2]{H^{#1}_{\mathrm{loc}}(#2)}

\newcommand{\Wkp}[3]{W^{#1,#2}(#3)}

\newcommand{\WkpLoc}[3]{W^{#1,#2}_{\mathrm{loc}}(#3)}

\newcommand{\Pp}[2]{\mathbb{P}^{#1}(#2)}

\newcommand{\VHarm}[1]{V_{\mathrm{harm}}(#1)}

\newcommand{\Fourier}[1]{\mathcal{F}(#1)}

\newcommand{\FourierInvN}[1]{\mathcal{F}^{-1} #1}
\newcommand{\FourierInv}[1]{\mathcal{F}^{-1}(#1)}

\newcommand{\FourierHatN}[1]{\widehat{#1}}

\newcommand{\FourierInvHatN}[1]{\check{#1}}	

\newcommand{\Elements}{\mathcal{T}}

\newcommand{\hMin}[1]{h_{\min\ifNotEmpty{#1}{,}#1}}
\newcommand{\hMax}[1]{h_{\max\ifNotEmpty{#1}{,}#1}}

\newcommand{\Tree}[2][]{\mathbb{T}_{#2}^{\ifNotEmpty{#1}{(}#1\ifNotEmpty{#1}{)}}}

\newcommand{\depth}[1]{\mathrm{depth}(#1)}

\newcommand{\BPart}{\mathbb{P}}
\newcommand{\BPartAdm}{\mathbb{P}_{\mathrm{adm}}}
\newcommand{\BPartSmall}{\mathbb{P}_{\mathrm{small}}}
\newcommand{\HMatrices}[2]{\mathcal{H}(#1,#2)}

\newcommand{\CCard}{\sigma_{\mathrm{card}}}     
\newcommand{\CAdm}{\sigma_{\mathrm{adm}}}       
\newcommand{\CSmall}{\sigma_{\mathrm{small}}}   
\newcommand{\CSparse}[1]{C_{\mathrm{sparse}}(#1)}   
\newcommand{\CExp}{\sigma_{\mathrm{exp}}}       
\newcommand{\CSSCO}{\sigma_{\mathrm{sco}}}		

\newcommand{\CutoffFc}[2]{\kappa_{#1}^{#2}}
\newcommand{\CutoffOp}[2]{K_{#1}^{#2}}
\newcommand{\CoarseningOp}[2]{Q_{#1}^{#2}}

\begin{document}

\title{$\mathcal{H}$-inverses for RBF interpolation}
\author[N. Angleitner, M. Faustmann, J.M. Melenk]{Niklas Angleitner, Markus Faustmann, Jens Markus Melenk}
\address{Technische Universit\"at Wien, Institute of Analysis and Scientific Computing (Inst. E 101), Wiedner Hauptstrasse 8-10, A-1040 Wien, Austria}
\email{niklas.angleitner@tuwien.ac.at}
\date{\today}

\begin{abstract}
We consider the interpolation problem for a class of radial basis functions (RBFs) that includes the classical polyharmonic splines (PHS). We show that the inverse of the system matrix for this interpolation problem can be approximated at an exponential rate in the block rank in the $\mathcal{H}$-matrix format if the block structure of the $\mathcal{H}$-matrix arises from a standard clustering algorithm.
\end{abstract}

\maketitle


\section{Introduction}

Radial basis functions (RBFs) have become an important tool in computational mathematics. Starting from the general question of interpolating scattered data, they found their way into statistics applications (see, e.g., \cite{Wahba90,GS94} or \cite{BL88} for application in machine learning) and, as a specific instance of meshfree methods, into the realm of numerical methods for partial differential equations \cite{MR2060191,Wendland_Scattered}.

The analysis of the approximation properties of a variety of RBFs is by now rather mature, \cite{MR1997878,MR2060191,Wendland_Scattered}. The classical interpolation problem with RBFs leads to linear systems with fully populated system matrices, which brings their efficient solution to the fore. A basic question in this connection is that of an efficient representation of the system matrix. Many RBFs such as polyharmonic splines, multiquadrics, and Gaussians are ``asymptotically smooth'' so that the system matrix can very efficiently be approximated by blockwise low rank matrices, \cite{MR1618780,MR2337575} and \cite{MR3707897,MR3779519}. This observation opened the door for log-linear complexity matrix-vector multiplication and a subsequent iterative solution. Nevertheless, good preconditioning strategies are necessary. Domain decomposition techniques \cite{MR1813294,MR3913657}, possibly combined with multilevel techniques \cite{MR2060191,MR4190812}, are an option. A recent alternative is the use of the arithmetic that comes with rank-structured matrices. Here, we consider specifically $\mathcal{H}$-matrices, \cite{Hackbusch_Hierarchical_matrices}. These are blockwise low-rank matrices endowed with a hierarchical structure that leads to an (approximate) arithmetic including addition, multiplication, inversion, and LU-factorization in logarithmic-linear complexity. This arithmetic can thus be used as a direct solver or be employed to create preconditioners. The arithmetic is only approximate but the error can be controlled by either a priori chosen parameters \cite{GH03} or adaptively \cite{BR03,Grasedyck05}. Yet, a basic question remains whether the target, i.e., the inverse of the system matrix or the LU-factorization, can be represented in the $\mathcal{H}$-matrix format. 
This question has been answered for finite element discretizations of elliptic PDEs in \cite{BebendorfHackbusch03}, recently improved in \cite{Faustmann_H_matrices_FEM} (to arbitrary accuracy) and \cite{Angleitner_H_matrices_FEM} (locally refined meshes), as well as for boundary element discretizations, \cite{FMP16,FMP17}, and the coupling of finite elements and boundary elements, \cite{FMP21}.

It is the purpose of the present paper to provide such an approximation result in \cref{Main_result} for RBFs that are fundamental solutions of certain partial differential operators with constant coefficients; in particular, the popular polyharmonic splines, introduced and analyzed in \cite{Duchon_Interpolation}, are covered by the present paper (see \cref{Fundamental_solution_Ex_2}).
Consequently, see \cite{Bebendorf07,Faustmann_H_matrices_FEM}, one also obtains the existence of $LU$-decompositions in the $\mathcal{H}$-matrix format, which gives mathematical underpinning to their  observed good performance in black-box preconditioning, \cite{MR3952680,MR4190812}.

The present paper is structured as follows: In \cref{Sec:Main_results} we formulate our model problem, the so called \emph{interpolation problem} together with the corresponding \emph{native space} $(V,\bilinear[a]{\cdot}{\cdot})$.  We then reformulate the problem as a linear system of equations (LSE) and introduce the \emph{interpolation matrix}. Moreover, some basic definitions concerning $\mathcal{H}$-matrices are given and finally, we state the main result, \cref{Main_result}. 
\cref{Sec:Proof} is devoted to the proof of our main result: First, in \cref{SSec:Native_space}, we investigate the native space $V$ and introduce the \emph{homogeneous native space} $V_0 \subseteq V$. In \cref{SSec:Interpol_problem_2}, we reformulate the original interpolation problem with an orthogonality condition in terms of $\bilinear[a]{\cdot}{\cdot}$ and the space $V_0$. In \cref{SSec:Rep_formula}, \cref{Rep_formula}, we derive a representation formula for the inverse system matrix to reduce the original ``matrix approximation problem'' to a ``function approximation problem''. Then, the main step in the proof is the construction of low dimensional spaces from which solutions to the interpolation problem can be approximated at an exponential rate, \cref{SSec:Coarse_ops}. 
Finally, \cref{Sec:Numerical_examples} provides some numerical examples.

Concerning notation: We write ``$a \cleq b$'' iff there exists a constant $C>0$ such that $a \leq C b$. The constant might depend on the space dimension $d$, the orders $k$ and $k_{\min}$ of the native space $V$, the coefficients $\sigma_l$ of the bilinear form $\bilinear[a]{\cdot}{\cdot}$, et cetera. However, it does \emph{not} depend on critical parameters such as the problem size $N$. We write $a \ceq b$, if both $a \cleq b$ and $a \cgeq b$ are satisfied. Matrices and vectors in linear systems of equations are expressed in boldface letters, e.g., $\mvemph{A} \in \R^{N \times N}$ and $\mvemph{f} \in \R^N$. The only exception to this rule is the set $\mvemph{C} \subseteq \R^N$ introduced in \cref{Coeff_set_C}, from which coefficient vectors $\mvemph{c} \in \mvemph{C}$ are drawn. For all $x \in \R^d$ and $r > 0$, we write $\Ball{x}{r} := \Set{y \in \R^d}{\norm[2]{y-x} < r}$ for the Euclidean ball of radius $r$, centered at $x$. The diameter of a set $\Omega \subseteq \R^d$ is denoted by $\diam[2]{\Omega} := \sup_{x,y \in \Omega} \norm[2]{x-y}$ and the distance between two sets $\Omega_1,\Omega_2 \subseteq \R^d$ is given by $\dist[2]{\Omega_1}{\Omega_2} := \inf_{x_1 \in \Omega_1, x_2 \in \Omega_2} \norm[2]{x_2-x_1}$. To treat derivatives in all space dimensions $d \geq 1$ alike, we adopt the usual multi-index notation $\DN{\alpha}{} := \partialN[\alpha_d]{d}{} \circ \dots \circ \partialN[\alpha_1]{1}{}$ and $\abs{\alpha} := \alpha_1+\dots+\alpha_d$, where $\alpha = (\alpha_1,\dots,\alpha_d) \in \N_0^d$.  We frequently drop the supplement ``for all $\alpha \in \N_0^d$'' and use shorthand phrases like ``for all $\abs{\alpha}=k$'' to describe the set $\Set{\alpha \in \N_0^d}{\abs{\alpha}=k}$.

As for specific function spaces, we work with the following definitions (all function spaces are meant to be real-valued): Let $d \geq 1$ and $\Omega \subseteq \R^d$ be an open set. For all $p \in \N_0$, we use the polynomial space $\Pp{p}{\Omega} := \spanN{\Set{x^{\alpha}}{\abs{\alpha} \leq p}}$, and we also write $\Pp{-1}{\Omega} := \set{0}$. We write $\Lp{p}{\Omega}$, $p \in [1,\infty]$, for the classical Lebesgue spaces. Moreover,  
a measurable function $\fDef{f}{\Omega}{\R}$ belongs to the space $\LpLoc{p}{\Omega}$, iff it satisfies $\restrictN{f}{\omega} \in \Lp{p}{\omega}$ for all bounded open sets $\omega \subseteq \R^d$ satisfying $\closureN{\omega} \subseteq \Omega$. For all $k \in \N_0$ and $p \in [1,\infty]$, the Sobolev space $\Wkp{k}{p}{\Omega}$ consists of all $k$-times weakly differentiable functions $f \in \LpLoc{1}{\Omega}$, whose weak derivatives $(\DN{\alpha}{f})_{\abs{\alpha} \leq k}$ lie in $\Lp{p}{\Omega}$. Similarly, a function $f \in \LpLoc{1}{\Omega}$ belongs to the space $\WkpLoc{k}{p}{\Omega}$, if it is $k$-times weakly differentiable and if there holds $\DN{\alpha}{f} \in \LpLoc{p}{\Omega}$ for all $\abs{\alpha} \leq k$. In the case $p=2$, we write $\Hk{k}{\Omega} := \Wkp{k}{2}{\Omega}$ and $\HkLoc{k}{\Omega} := \WkpLoc{k}{2}{\Omega}$.

\section{Main results} \label{Sec:Main_results}
\subsection{The interpolation problem} \label{SSec:Interpol_problem_1}

We start our presentation with the definition of the \emph{native space $V$}, which forms the basis of our functional-analytic framework:
\begin{definition} \label{Native_space}
Let $d,k \in \N$ with $k>d/2$. Furthermore, let $k_{\min} \in \set{0,\dots,k}$ and $\sigma_{k_{\min}},\dots,\sigma_k \geq 0$ be given coefficients with $\sigma_{k_{\min}},\sigma_k > 0$. We define the \emph{native space $V$} and the polynomial space $P$:
\begin{equation*}
\begin{array}{rcl}
V &:=& \Set{v \in \LpLoc{1}{\R^d}}{\forall l \in \set{k_{\min},\dots,k}: \forall \abs{\alpha}=l: \DN{\alpha}{v} \in \Lp{2}{\R^d}}, \\
P &:=& \Pp{k_{\min}-1}{\R^d}.
\end{array}
\end{equation*}
In other words, a function $v \in \LpLoc{1}{\R^d}$ belongs to $V$, if, for all $l \in \set{k_{\min},\dots,k}$ and all $\abs{\alpha}=l$, the $\alpha$-th weak derivative $\DN{\alpha}{v} \in \LpLoc{1}{\R^d}$ exists and lies in $\Lp{2}{\R^d}$. For all $u,v \in V$, we define
\begin{equation*}
\bilinear[a]{u}{v} := \sum_{l=k_{\min}}^{k} \sigma_l \sum_{\abs{\alpha}=l} \frac{l!}{\alpha!} \skalar[\Lp{2}{\R^d}]{\DN{\alpha}{u}}{\DN{\alpha}{v}}, \quad \quad \quad \seminorm[a]{v} := \sqrt{\bilinear[a]{v}{v}}.
\end{equation*}

Furthermore, for every open subset $\Omega \subseteq \R^d$, we set
\begin{equation*}
\seminorm[a,\Omega]{v} := \( \sum_{l=k_{\min}}^{k} \sigma_l \sum_{\abs{\alpha}=l} \frac{l!}{\alpha!} \norm[\Lp{2}{\Omega}]{\DN{\alpha}{u}}^2 \)^{1/2}.
\end{equation*}

\end{definition}

Note that the assumption $k>d/2$ guarantees that the Sobolev embedding $\Hk{k}{\Omega} \subseteq \Ck{0}{\closureN{\Omega}}$ is continuous for all bounded Lipschitz domains $\Omega \subseteq \R^d$ as well as $\Omega = \R^d$ (e.g., \cite[Chapter 7]{Gilbarg_Trudinger}). In particular, $\norm[\Ck{0}{\closureN{\Omega}}]{v} \leq C(d,k,\Omega) \norm[\Hk{k}{\Omega}]{v}$, for all $v \in \Hk{k}{\Omega}$. We will make use of this fact on several occasions throughout this work.

For the basic properties of $V$, $P$, $\bilinear[a]{\cdot}{\cdot}$ and $\seminorm[a]{\cdot}$, we refer the reader to \cref{Native_space_Props} in the proof section below. In brief, $\bilinear[a]{\cdot}{\cdot}$ defines a symmetric positive semi-definite bilinear form on $V$ and $\seminorm[a]{\cdot}$ is a seminorm with kernel $P$. There holds $V \subseteq \Ck{0}{\R^d}$, so that every $v \in V$ has well-defined point-values $v(x)$, $x \in \R^d$. In the extreme case $k_{\min} = k$, there holds $V = \mathrm{BL}^k(\R^d)$, which is the well-known \emph{homogeneous Sobolev space} or \emph{Beppo-Levi space} (e.g., \cite{Deny_Lions_Beppo_Levi}, \cite{Specovius_Density}). In the other extreme case $k_{\min} = 0$, however, there holds $V = \Hk{k}{\R^d}$, the standard Sobolev space. The norms are equivalent and the polynomial space becomes trivial, $P = \set{0}$. In particular, $(V,\bilinear[a]{\cdot}{\cdot})$ is then a proper Hilbert space.

\begin{definition} \label{Interpol_points}
Let $N_{\min} := \dimN{P} = \binom{d+k_{\min}-1}{d}$. For each $N \in \N$ with $N \geq N_{\min}$, let $\set{x_1,\dots,x_N} \subseteq \R^d$ be a set of pairwise distinct \emph{interpolation points}. We denote their \emph{separation distance} by
\begin{equation*}
\hMin{N} := \hMin{} := \frac{1}{2} \min_{n \neq m \in \set{1,\dots,N}} \norm[2]{x_n-x_m}.
\end{equation*}

We make the following assumptions about the family $(\set{x_1,\dots,x_N})_{N \geq N_{\min}}$:
\begin{enumerate}
\item \emph{Boundedness:} There exists a constant $C>0$, such that
\begin{equation*}
\forall N \geq N_{\min}: \forall n \in \set{1,\dots,N}: \quad \quad \norm{x_n} \leq C.
\end{equation*}

\item \emph{Unisolvent subset:} There exists a set of points $\Set{\xi_{\alpha}}{\abs{\alpha}<k_{\min}} \subseteq \R^d$, unisolvent for the space $P$, such that
\begin{equation*}
\forall N \geq N_{\min}: \quad \quad \Set{\xi_{\alpha}}{\abs{\alpha}<k_{\min}} \subseteq \set{x_1,\dots,x_N}.
\end{equation*}

\item \emph{Balance $N \leftrightarrow \hMin{}$:} There exist constants $C,\CCard \geq 1$, such that
\begin{equation*}
\forall N \geq N_{\min}: \quad \quad 1 \leq C N^{\CCard} \hMin{N}^d.
\end{equation*}
\end{enumerate}

\end{definition}

During the proof of \cref{Space_VBDL}, the assumption of boundedness will allow us to apply some Poincar\'e-type inequality on a fixed, bounded subset of $\R^d$, which is independent of the problem size $N$. Furthermore, as a minor technical detail, it guarantees that $\hMin{N} \cleq 1$ for all $N \geq N_{\min}$.

The unisolvency assumption, on the other hand, allows us to make the following implication: If a polynomial $p \in P$ satisfies $p(\xi_{\alpha})=0$ for all $\abs{\alpha}<k_{\min}$, then already $p=0$. This argument will be used on numerous occasions. Note that \emph{all} sets of interpolation points $\set{x_1,\dots,x_N}$, $N \geq N_{\min}$, must contain \emph{the same} set of unisolvent points $\Set{\xi_{\alpha}}{\abs{\alpha} < k_{\min}}$. This assumption is necessary to ensure that the aforementioned Poincar\'e constant in \cref{Space_VBDL} is independent of the problem size $N$. Clearly, if the family $(\set{x_1,\dots,x_N})_{N \geq N_{\min}}$ is constructed by an algorithm that successively adds more points to some initial point set, never deleting or modifying existing ones, then this assumption is satisfied.

The asserted balance between the problem size $N$ and the separation distance $\hMin{N}$ is fulfilled for a wide variety of families of interpolation points. As an example, consider the case where a bounded domain $\Omega \subseteq \R^d$ is given and where $\set{x_1,\dots,x_N} \subseteq \Omega$ for all $N \geq N_{\min}$. Denote by $\Elements_N := \set{T_1,\dots,T_N}$ the corresponding Voronoi decomposition of $\Omega$, i.e., $T_n = \Set{x \in \Omega}{\norm{x-x_n} = \min_{m \in \set{1,\dots,N}} \norm{x-x_m}}$, and by $\hMax{N} := \max_{n \in \set{1,\dots,N}} \diam[2]{T_n}$ the maximal cell diameter. Now suppose that
\begin{equation*}
\hMax{N}^{\CCard} \leq C\hMin{N},
\end{equation*}
which is satisfied, e.g., for \emph{uniform} and \emph{algebraically graded} families of interpolation points. Then,
\begin{equation*}
1 \cleq \meas{\Omega}^{\CCard} = \(\sum_{n=1}^{N} \meas{T_n}\)^{\CCard} \leq \(\sum_{n=1}^{N} \diam[2]{T_n}^d\)^{\CCard} \leq (N \hMax{N}^d)^{\CCard} \leq C^d N^{\CCard} \hMin{N}^d.
\end{equation*}

Although \cref{Interpol_points} is phrased in terms of a \emph{family} of interpolation points, most of the upcoming results deal with a single set $\set{x_1,\dots,x_N} \subseteq \R^d$ of interpolation points. For the most part, one can therefore think of $N$ as being ``fixed'' throughout this work.

\begin{definition} \label{Eval_op}
The \emph{evaluation operator} $\fDef{E_N}{\Ck{0}{\R^d}}{\R^N}$ is defined by $E_N v := (v(x_n))_{n=1}^{N}$ for all $v \in \Ck{0}{\R^d}$.
\end{definition}

Recall that $V \subseteq \Ck{0}{\R^d}$, so that $E_N v$ is well-defined for all $v \in V$.

Now that the native space $V$ and the interpolation points $x_1,\dots,x_N$ are fixed, let us pose the \emph{interpolation problem}:
\begin{problem} \label{Interpol_problem_1}
Let $f \in V$. Find a function $u \in V$ that satisfies the following interpolation and minimization properties:
\begin{equation*}
E_N u = E_N f, \quad \quad \quad \seminorm[a]{u} \leq \inf_{\substack{\tilde{u} \in V: \\ E_N \tilde{u} = E_N f}} \seminorm[a]{\tilde{u}}.
\end{equation*}
\end{problem}

Clearly, the interpolation condition ``$E_N u = E_N f$'' can also be written as ``$\forall n \in \set{1,\dots,N}: u(x_n) = f(x_n)$''. In other words, we are looking for a minimizer $u$ of the seminorm $\seminorm[a]{\cdot}$ over the set of interpolants of $f$. While the interpolation conditions fix the values of $u$ at the interpolation points $x_1,\dots,x_N$, the minimization property determines the behavior of $u$ on the rest of $\R^d$.

In \cref{SSec:Interpol_problem_2} below we will see that, for every given $f \in V$, this interpolation problem has a unique solution $u \in V$. It turns out that the mapping $f \mapsto u$ is linear and that there holds the a priori bound $\seminorm[a]{u} \leq \seminorm[a]{f}$, i.e., the problem is well-posed.

\subsection{The fundamental solution} \label{SSec:Fundamental_solution}

At first glance, for given $f \in V$, the solution $u \in V$ of \cref{Interpol_problem_1} looks like an infinite-dimensional object. However, since the set of interpolants $\tilde{u} \in V$ of $f$ is so large, the minimization property contains a lot of information about $u$. In fact, we shall shortly see that $u$ can be written in the form $u = \sum_{n=1}^{N} \mvemph{c}_n \phi_n + \sum_{\abs{\alpha}<k_{\min}} \mvemph{d}_{\alpha} \pi_{\alpha}$, where $\mvemph{c}_n, \mvemph{d}_{\alpha} \in \R$ are certain coefficients and where $\Set{\pi_{\alpha}}{\abs{\alpha}<k_{\min}} \subseteq P$ is a polynomial basis. The functions $\phi_n \in \Ck{0}{\R^d}$, on the other hand, have the particularly simple structure of translates, i.e., $\phi_n = \phi(\cdot-x_n)$. Here, $x_n \in \R^d$ is the $n$-th interpolation point and $\phi \in \Ck{0}{\R^d}$ is a specific function that is intimately linked to the native space $(V,\bilinear[a]{\cdot}{\cdot})$ from \cref{Native_space}. More precisely, $\phi$ is a fundamental solution of the differential operator that is associated with the bilinear form $\bilinear[a]{\cdot}{\cdot}$. The above representation will then allow us to rephrase \cref{Interpol_problem_1} as a linear system of equations (LSE) for the coefficients $\mvemph{c}_n, \mvemph{d}_{\alpha} \in \R$.

Before we treat the general setting $k_{\min} \geq 0$, let us have a look at the much simpler case $k_{\min} = 0$ first. Recall that then $V = \Hk{k}{\R^d}$ with equivalent norms. 
\begin{lemma} \label{Fundamental_solution_Ex_0}
If $k_{\min} = 0$, then there exists a unique function $\phi \in V$, such that the following equality holds true:
\begin{equation*}
\forall x_0 \in \R^d: \forall v \in V: \quad \quad \bilinear[a]{\phi(\cdot-x_0)}{v} = v(x_0).
\end{equation*}
\end{lemma}

\begin{proof}
The continuous Sobolev embedding $\Hk{k}{\R^d} \subseteq \Ck{0}{\R^d}$ guarantees that the linear form $v \mapsto v(0)$ is continuous. Then, according to the Riesz-Fr\'echet Representation Theorem, there exists a unique function $\phi \in V$, such that $\bilinear[a]{\phi}{v} = v(0)$ for all $v \in V$. Since the coefficients $\sigma_l l!/\alpha!$ of $\bilinear[a]{\cdot}{\cdot}$ are spatially constant, a simple integral transformation yields the desired formula for general $x_0 \in \R^d$.
\end{proof}

Now, given data $f \in V$, we make the ansatz $u := \sum_{n=1}^{N} \mvemph{c}_n \phi(\cdot-x_n) \in V$ for the solution of \cref{Interpol_problem_1}. The coefficients $\mvemph{c}_n \in \R$ can be chosen such that the interpolation conditions $E_N u = E_N f$ are satisfied (cf. \cref{LSE_cd_to_u}). On the other hand, for all $v \in V$ with $E_N v = \mvemph{0}$, the defining equation of $\phi$ tells us that $\bilinear[a]{u}{v} = \sum_{n=1}^{N} \mvemph{c}_n v(x_n) = 0$. In \cref{Interpol_problems_equivalent} further below, it will be argued in more detail that this orthogonality implies the required minimization property. We conclude that $u = \sum_{n=1}^{N} \mvemph{c}_n \phi(\cdot-x_n)$ is indeed the unique solution of \cref{Interpol_problem_1}.

Now let us return to the general case $k_{\min} \geq 0$. Since $\bilinear[a]{\cdot}{\cdot}$ need not be \emph{strictly} positive definite, the existence of a ``basis function'' $\phi$ is not so straightforward any more. Therefore, we have to take a different approach.
\begin{definition} \label{Diff_op}
Denote by $\sigma_{k_{\min}},\dots,\sigma_k \geq 0$ the coefficients of the bilinear form $\bilinear[a]{\cdot}{\cdot}$ from \cref{Native_space}. We define a differential operator $\fDef{\DN{2k}{}}{\CkO{\infty}{\R^d}}{\CkO{\infty}{\R^d}}$ via
\begin{equation*}
\DN{2k}{v} := \sum_{l=k_{\min}}^{k} \sigma_l (-\LaplaceN{})^l v.
\end{equation*}
\end{definition}

The precise relationship between $\bilinear[a]{\cdot}{\cdot}$ and $\DN{2k}{}$ is described in the subsequent lemma. Anticipating \cref{Native_space_Props}, we claim that there holds $\CkO{\infty}{\R^d} \subseteq V$, so we can plug any given $v \in \CkO{\infty}{\R^d}$ into $\bilinear[a]{\cdot}{\cdot}$. Furthermore, we will prove that every $u \in V$ lies in $\HkLoc{k}{\R^d}$, so that the integrals in $\bilinear[a]{u}{v}$ are amenable to successive partial integrations over bounded subsets of $\R^d$.
\begin{lemma} \label{Bil_form_Diff_op}
For all $u \in V$ and $v \in \CkO{\infty}{\R^d}$, there holds the identity
\begin{equation*}
\bilinear[a]{u}{v} = \I{\R^d}{u(\DN{2k}{v})}{x}.
\end{equation*}

\end{lemma}

\begin{proof}
The relation follows readily from successive partial integrations and the identity $\sum_{\abs{\alpha}=l} (l!/\alpha!) \DN{2\alpha}{v} = \LaplaceN{}^l v$. The compact support of $v \in \CkO{\infty}{\R^d}$ and the fact that $\DN{\alpha}{u} \in \LpLoc{2}{\R^d}$ for all $\abs{\alpha} \leq k$, guarantee that all integrals involved in the computation are well-defined.
%
%
\end{proof}

\begin{definition} \label{Coeff_set_C}
We define the set of coefficient vectors $\mvemph{C} := \Set{\mvemph{c} \in \R^N}{\forall p \in P: \skalar[2]{\mvemph{c}}{E_N p} = 0}$.
\end{definition}

A careful analysis of the function $\phi$ from the case $k_{\min}=0$ above leads us to the following set of assumptions:
\begin{assumption} \label{Fundamental_solution}
We assume that the coefficients $\sigma_{k_{\min}},\dots,\sigma_k \geq 0$ from \cref{Native_space} are such that there exists a function $\fDef{\phi}{\R^d}{\R}$ with the following properties:
\begin{enumerate}
\item \emph{Regularity:} There holds $\phi \in \Ck{0}{\R^d}$.
\item \emph{Fundamental solution:} $\phi$ is a fundamental solution of $\DN{2k}{}$, i.e.,
\begin{equation*}
\forall x_0 \in \R^d: \forall v \in \CkO{\infty}{\R^d}: \quad \quad \I{\R^d}{\phi(x-x_0)(\DN{2k}{v})(x)}{x} = v(x_0).
\end{equation*}
\item \emph{Conformity:} For all $\mvemph{c} \in \mvemph{C}$, there holds $\sum_{n=1}^{N} \mvemph{c}_n \phi(\cdot-x_n) \in V$.
\end{enumerate}
\end{assumption}

Note that we \emph{did not} assume that $\phi$ lies in the native space $V$. Instead, we only require certain linear combinations of its translates to do so. Combining \cref{Bil_form_Diff_op} with $(2)$ and $(3)$, we regain the important orthogonality $\bilinear[a]{u}{v} = \sum_{n=1}^{N} \mvemph{c}_n v(x_n) = 0$ for all $u$ of the form $u = \sum_{n=1}^{N} \mvemph{c}_n \phi(\cdot-x_n)$, $\mvemph{c} \in \mvemph{C}$, and for all $v \in \CkO{\infty}{\R^d}$ with $E_N v = \mvemph{0}$. Using a density argument, the orthogonality can then be extended to the space of functions $v \in V$ with $E_N v = \mvemph{0}$. A detailed proof of these facts will be given later in \cref{Coord_mappings_Phi_Pi_Props}.

In contrast to the case $k_{\min} = 0$, a function $\phi$ satisfying \cref{Fundamental_solution} need not be unique if $k_{\min}>0$. In fact, if $\phi$ is a valid choice and if $p \in P$, then $\phi + p$ is valid as well.

According to the much celebrated Malgrange-Ehrenpreis Theorem, \cite{Malgrange}, \cite{Ehrenpreis}, every non-trivial differential operator with constant coefficients has a fundamental solution in the distributional sense. Since $\DN{2k}{} = \sum_{l=k_{\min}}^{k} \sigma_l (-\LaplaceN{})^l$ falls into this category, this theorem provides us with a distribution $\phi \in \dual{\CkO{\infty}{\R^d}}$ that satisfies assumption $(2)$ in the distributional sense. However, to the best of our knowledge, there is no guarantee that this $\phi$ satisfies $(1)$ and $(3)$, if no further assumptions on the coefficients $\sigma_l$ are made.

\cref{Fundamental_solution_Ex_0} is somewhat unsatisfactory, since it does not provide an explicit formula for the basis function $\phi$. However, for a specific choice of the coefficients $\sigma_0,\dots,\sigma_k$, we have the following result, the proof of which we postpone to \cref{SSec:Fundamental_solutions_proofs}.
\begin{lemma} \label{Fundamental_solution_Ex_1}
Let $b \in (0,\infty)$. Consider the case where $k_{\min} := 0$ and $\sigma_l := \binom{k}{l} b^{2(k-l)} > 0$ for all $l \in \set{0,\dots,k}$. The differential operator from \cref{Diff_op} takes the form
\begin{equation*}
\DN{2k}{v} = \sum_{l=0}^{k} \binom{k}{l} b^{2(k-l)} (-\LaplaceN{})^l v = (b^2-\LaplaceN{})^k v.
\end{equation*}

Define the function
\begin{equation*}
\forall x \in \R^d: \quad \quad \phi(x) := \frac{(4\pi)^{-d/2}}{\Gamma(k)} \I[\infty]{0}{t^{k-d/2-1} e^{-b^2 t} e^{-\norm{x}^2/(4t)}}{t}.
\end{equation*}

Then, $\phi$ satisfies \cref{Fundamental_solution}.
\end{lemma}

We mention that $\phi$ can also be written in the form
\begin{equation*}
\phi(x) = \frac{(2\pi)^{-d/2}}{2^{k-1} \Gamma(k)} \(\frac{\norm{x}}{b}\)^{k-d/2} K_{k-d/2}(b\norm{x}),
\end{equation*}
which goes by the name of \emph{Mat\'ern function}, \emph{Sobolev spline} or \emph{Bessel potential} in the literature (e.g., \cite{Aronszajn_Bessel_potentials}). Here, $K_{\nu}(r) := \I[\infty]{0}{e^{-r\cosh(s)}\cosh(\nu s)}{s}$ is the well-known \emph{modified Bessel function of the second kind}. To see the identity, one writes $K_{\nu}(r) = \I[\infty]{-\infty}{e^{-r(e^s+e^{-s})/2} e^{\nu s}/2}{s}$, plugs in $\nu=k-d/2$ and $r=b\norm{x}$, and then substitutes $s=\ln(2bt/\norm{x})$. Furthermore, in the case where $d \in \set{1,3,5,\dots}$, there holds the representation
\begin{equation*}
\phi(x) = \frac{(4\pi)^{(1-d)/2}}{\Gamma(k) (2b)^{2L+1}} \sum_{l=0}^{L} \frac{(2L-l)!}{l!(L-l)!} (2b\norm{x})^l e^{-b\norm{x}}, \quad \quad \quad L := k-d/2-1/2 \in \N_0.
\end{equation*}
This follows easily from the known identity $K_{L+1/2}(r) = \pi^{1/2} e^{-r} \sum_{l=0}^{L} (L+l)!/(l!(L-l)!(2r)^{l+1/2})$, which can be found, e.g., in \cite[Page 925]{Gradshteyn_Table}.

We finish this section with another example for a function $\phi$ satisfying \cref{Fundamental_solution}. This time, we look at the other extreme case, $k_{\min} = k$. The proof will be delayed to \cref{SSec:Fundamental_solutions_proofs} again.
\begin{lemma} \label{Fundamental_solution_Ex_2}
Consider the case where $k_{\min} = k$ and $\sigma_{k_{\min}} = \sigma_k = 1$. The differential operator from \cref{Diff_op} reads
\begin{equation*}
\DN{2k}{v} = (-\LaplaceN{})^k v.
\end{equation*}

Define the \emph{thin-plate spline}
\begin{equation*}
\begin{array}{lrcl}
(d \in \set{1,3,5,\dots}) & \phi(x) &:=& C_1 \norm{x}^{2k-d}, \\
(d \in \set{2,4,6,\dots}) & \phi(x) &:=& C_2 \norm{x}^{2k-d} \ln\norm{x},
\end{array}
\end{equation*}
where $C_1 := \Gamma(d/2-k)/(4^k \pi^{d/2} (k-1)!)$ and $C_2 := (-1)^{k+(d-2)/2}/(2^{2k-1} \pi^{d/2} (k-1)!(k-d/2)!)$. Then, $\phi$ satisfies \cref{Fundamental_solution}.
\end{lemma}

\subsection{The LSE} \label{SSec:LSE}

We already suggested that the solution $u \in V$ of \cref{Interpol_problem_1} can be written in the form $u = \sum_{n=1}^{N} \mvemph{c}_n \phi_n + \sum_{\abs{\alpha}<k_{\min}} \mvemph{d}_{\alpha} \pi_{\alpha}$, where the coefficients $\mvemph{c}_n, \mvemph{d}_{\alpha} \in \R$ satisfy a certain LSE. The corresponding \emph{interpolation matrix} is the main object of interest of the present work.

\begin{definition} \label{Ansatz_functions}
Denote by $x_1,\dots,x_N \in \R^d$ the interpolation points from \cref{Interpol_points} and by $\phi \in \Ck{0}{\R^d}$ the fundamental solution from \cref{Fundamental_solution}. For every $n \in \set{1,\dots,N}$, we define the translate
\begin{equation*}
\phi_n := \phi(\cdot-x_n) \in \Ck{0}{\R^d}.
\end{equation*}
Furthermore, let $\Set{\pi_{\alpha}}{\abs{\alpha}<k_{\min}} \subseteq P$ be the Lagrange basis associated with the unisolvent point set $\Set{\xi_{\alpha}}{\abs{\alpha}<k_{\min}} \subseteq \R^d$ from \cref{Interpol_points}, i.e., $\pi_{\beta}(\xi_{\alpha}) = \kronecker{\alpha\beta}$ (Kronecker $\delta$) for all $\abs{\alpha},\abs{\beta}<k_{\min}$.
\end{definition}

\begin{definition} \label{LSE_Matrices}
We define the following matrices:
\begin{equation*}
\mvemph{A} := (\phi_n(x_m))_{m,n=1}^{N} \in \R^{N \times N}, \quad \quad \quad \mvemph{B} := (\pi_{\beta}(x_n))_{\abs{\beta}<k_{\min},n \in \set{1,\dots,N}} \in \R^{N_{\min} \times N}.
\end{equation*}

Furthermore, we define the \emph{interpolation matrix}
\begin{equation*}
\(\begin{matrix} \mvemph{A} & \mvemph{B}^T \\ \mvemph{B} & \mvemph{0} \end{matrix}\) \in \R^{(N+N_{\min})\times(N+N_{\min})}.
\end{equation*}
\end{definition}

As will be shown later in \cref{LSE_invertible}, the interpolation matrix is invertible. Now, let $f \in V$ and set $\mvemph{f} := E_N f \in \R^N$. Denote by $\mvemph{c} \in \R^N$ and $\mvemph{d} \in \R^{N_{\min}}$ the unique solution of the following LSE:
\begin{equation*}
\(\begin{matrix} \mvemph{A} & \mvemph{B}^T \\ \mvemph{B} & \mvemph{0} \end{matrix}\) \(\begin{matrix} \mvemph{c} \\ \mvemph{d} \end{matrix}\) = \(\begin{matrix} \mvemph{f} \\ \mvemph{0} \end{matrix}\).
\end{equation*}

Then, the unique solution $u \in V$ of \cref{Interpol_problem_1} can be written in the following form:
\begin{equation*}
u = \sum_{n=1}^{N} \mvemph{c}_n \phi_n + \sum_{\abs{\alpha}<k_{\min}} \mvemph{d}_{\alpha} \pi_{\alpha}.
\end{equation*}
Once again, we postpone the derivation of this identity to the proof section below, \cref{LSE_cd_to_u}. The first row of the LSE encodes the interpolation conditions $E_N u = E_N f$, whereas the second row guarantees that the coefficient vector $\mvemph{c} \in \R^N$ lies in the set $\mvemph{C}$ from \cref{Coeff_set_C}. Owing to \cref{Fundamental_solution}, this implies the conformity $\sum_{n=1}^{N} \mvemph{c}_n \phi_n \in V$.

\subsection{Hierarchical matrices} \label{SSec:Hierarchical_matrices}

An extensive discussion of hierarchical matrices can be found in the books \cite{Hackbusch_Hierarchical_matrices,BebendorfBuch,BoermBuch}. Here, we only provide a bare minimum of definitions that are need for the subsequent analysis.

\begin{definition} \label{Bubbles}
We define the \emph{bubbles}
\begin{equation*}
\forall n \in \set{1,\dots,N}: \quad \quad \Bubbles{n} := \Ball[2]{x_n}{\hMin{}} \subseteq \R^d.
\end{equation*}
Similarly, for every subset $I \subseteq \set{1,\dots,N}$, we set $\Bubbles{I} := \bigcup_{n \in I} \Bubbles{n} \subseteq \R^d$.
\end{definition}

Note that the definition of the separation distance $\hMin{}$ in \cref{Interpol_points} guarantees that the bubbles $\Bubbles{n}$ are pairwise disjoint, i.e., $\Bubbles{n} \cap \Bubbles{m} = \emptyset$ whenever $n \neq m$.

\begin{definition} \label{Box}
A subset $B \subseteq \R^d$ is called \emph{(axes-parallel) box}, if it has the form $B = \bigtimes_{i=1}^{d} (a_i,b_i)$ with $a_i \leq b_i$.
\end{definition}

\begin{definition} \label{Block_partition}
Let $\CSmall,\CAdm>0$. A tuple $(I,J)$ with $I,J \subseteq \set{1,\dots,N}$ is called \emph{small}, if there holds $\min\set{\cardN{I}, \cardN{J}} \leq \CSmall$. It is called \emph{admissible}, if there exist boxes $B_I,B_J \subseteq \R^d$ such that $\Bubbles{I} \subseteq B_I$, $\Bubbles{J} \subseteq B_J$ and
\begin{equation*}
\diam[2]{B_I} \leq \CAdm \dist[2]{B_I}{B_J}.
\end{equation*}

A set $\BPart$ of tuples $(I,J)$ with $I,J \subseteq \set{1,\dots,N}$ is called \emph{sparse hierarchical block partition}, if the following assumptions are satisfied:
\begin{enumerate}
\item The system $\Set{I \times J}{(I,J) \in \BPart}$ forms a partition of $\set{1,\dots,N} \times \set{1,\dots,N}$.

\item There holds $\BPart = \BPartSmall \cup \BPartAdm$, where every $(I,J) \in \BPartSmall$ is small and every $(I,J) \in \BPartAdm$ is admissible.

\item There exists a constant $C>0$, such that
\begin{equation*}
\forall \mvemph{M} \in \R^{N \times N}: \quad \quad \norm[2]{\mvemph{M}} \leq C \ln(N) \max_{(I,J) \in \BPart} \norm[2]{\restrictN{\mvemph{M}}{I \times J}}.
\end{equation*}

\end{enumerate}

\end{definition}

While items $(1)$ and $(2)$ are more or less standard in the literature about hierarchical matrices, item $(3)$ might seem odd to the informed reader. Usually, this inequality is proved, rather than assumed (e.g., \cite[Lemma 6.5.8]{Hackbusch_Hierarchical_matrices}). However, its proof typically requires a rigorous introduction of \emph{(hierarchical) cluster trees} and \emph{(hierarchical) block cluster trees}, two types of data structure that are described in great detail in \cite[Chapter 5]{Hackbusch_Hierarchical_matrices}. Here, we largely avoid this tedious task and only give a brief overview:

In essence, a \emph{(hierarchical) cluster tree} $\Tree{N}$ is a system of index \emph{clusters} $I \subseteq \set{1,\dots,N}$ that contains the full index set $\set{1,\dots,N}$ as its root. Furthermore, for every $I \in \Tree{N}$ with $\cardN{I} > \CSmall$, the cluster tree also contains two non-empty, disjoint clusters $I_1,I_2$ with $I = I_1 \cup I_2$. Starting at the tree root, these \emph{sons} are typically generated by a predefined \emph{clustering strategy} which takes into account the geometric positions of the interpolation points $x_1,\dots,x_N \in \R^d$. An example for a \emph{geometrically balanced clustering strategy}, which uses a hierarchy of axes-parallel boxes $B_I \subseteq \R^d$, can be found in \cite{Hackbusch_Geometric_clustering}. Finally, we denote by $\depth{\Tree{N}} \in \N$ the tree's \emph{depth}, and assume that the reader is familiar with this notion.

On the other hand, a \emph{(hierarchical) block cluster trees} $\Tree{N \times N}$ consists of tuples $(I,J)$, where $I$ and $J$ are clusters from a given cluster tree $\Tree{N}$. This time, the tuple $(\set{1,\dots,N}, \set{1,\dots,N})$ is the tree's root and, for every $(I,J) \in \Tree{N \times N}$ with $\diam[2]{B_I} > \CAdm \dist[2]{B_I}{B_J}$, all four pairs of sons $(I_1,J_1)$, $(I_I,J_2)$, $(I_2,J_1)$, $(I_2,J_2)$ lie in $\Tree{N \times N}$. Here, the sets $B_I,B_J \subseteq \R^d$ are the boxes that were used to build the cluster tree $\Tree{N}$. Finally, the well-known \emph{sparsity constant} $\CSparse{\Tree{N \times N}} \in \N$ counts the maximum number of ``partners'' $(I,J) \in \Tree{N \times N}$ that each cluster $I$ or $J$ can have (see, e.g., \cite[Section 6.3]{Hackbusch_Hierarchical_matrices}).

The geometrically balanced clustering strategy from \cite{Hackbusch_Geometric_clustering} guarantees the bounds $\depth{\Tree{N}} \cleq \ln(\hMin{}^{-1})$ and $\CSparse{\Tree{N \times N}} \cleq 1$. We plug in the relation $1 \cleq N^{\CCard} \hMin{}^d$ from \cref{Interpol_points} and get $\depth{\Tree{N}} \cleq \ln(N)$. Now, using \cite[Lemma 6.5.8]{Hackbusch_Hierarchical_matrices}, we obtain the following inequality:
\begin{equation*}
\forall \mvemph{M} \in \R^{N \times N}: \quad \quad \norm[2]{\mvemph{M}} \leq \CSparse{\Tree{N \times N}} \depth{\Tree{N}} \max_{(I,J) \in \BPart} \norm[2]{\restrictN{\mvemph{M}}{I \times J}} \cleq \ln(N) \max_{(I,J) \in \BPart} \norm[2]{\restrictN{\mvemph{M}}{I \times J}}.
\end{equation*}

These considerations justify the assumptions that we made in \cref{Block_partition}.

\begin{definition} \label{H_matrices}
Let $\BPart$ be a sparse hierarchical block partition and $r \in \N$ a given \emph{block rank bound}. We define the set of \emph{$\mathcal{H}$-matrices} by
\begin{equation*}
\HMatrices{\BPart}{r} := \Set{\mvemph{M} \in \R^{N \times N}}{\forall (I,J) \in \BPartAdm: \exists \mvemph{X} \in \R^{I \times r}, \mvemph{Y} \in \R^{J \times r}: \restrictN{\mvemph{M}}{I \times J} = \mvemph{X} \mvemph{Y}^T}.
\end{equation*}
\end{definition}

We mention that, according to \cite[Lemma 6.3.6]{Hackbusch_Hierarchical_matrices}, the memory requirements to store an $\mathcal{H}$-matrix $\mvemph{M} \in \HMatrices{\BPart}{r}$ can be bounded by $\CSparse{\Tree{N \times N}} (\CSmall + r) \depth{\Tree{N}} N$. Inserting the relations $\CSparse{\Tree{N \times N}} \cleq 1$ and $\depth{\Tree{N}} \cleq \ln(N)$ from before, we get an overall bound of $\Landau{r\ln(N)N}$ for the storage complexity.

\subsection{The main result}

We are finally in the position to formulate our main result.

\begin{theorem} \label{Main_result}
Let $\bilinear[a]{\cdot}{\cdot}$ be the bilinear form from \cref{Native_space} and consider a set of interpolation points $\set{x_1,\dots,x_N} \subseteq \R^d$ satisfying \cref{Interpol_points}. Denote by $\phi \in \Ck{0}{\R^d}$ the fundamental solution from \cref{Fundamental_solution}. Furthermore, let $(\begin{smallmatrix} \mvemph{A} & \mvemph{B}^T \\ \mvemph{B} & \mvemph{0} \end{smallmatrix})$ be the interpolation matrix from \cref{LSE_Matrices}. Finally, let $\BPart$ be a sparse hierarchical block partition as in \cref{Block_partition}. Write the inverse of the interpolation matrix in the block form $(\begin{smallmatrix} \mvemph{A} & \mvemph{B}^T \\ \mvemph{B} & \mvemph{0} \end{smallmatrix})^{-1} = (\begin{smallmatrix} \mvemph{S_{11}} & \mvemph{S_{12}} \\ \mvemph{S_{21}} & \mvemph{S_{22}} \end{smallmatrix})$ with matrices $\mvemph{S_{11}} \in \R^{N \times N}$, $\mvemph{S_{21}} \in \R^{N_{\min} \times N}$, $\mvemph{S_{12}} \in \R^{N \times N_{\min}}$ and $\mvemph{S_{22}} \in \R^{N_{\min} \times N_{\min}}$. Then, there exists a constant $\CExp>0$ such that the following holds true: For every block rank bound $r \in \N$, there exists an $\mathcal{H}$-matrix $\mvemph{M} \in \HMatrices{\BPart}{r}$ such that
\begin{equation*}
\norm[2]{\mvemph{S_{11}} - \mvemph{M}} \cleq \ln(N) N^{\CCard(3k-d)/d} \exp(-\CExp r^{1/(d+1)}).
\end{equation*}
\end{theorem}

\section{Proof of main results} \label{Sec:Proof}
\subsection{The native space $V$} \label{SSec:Native_space}

We start the proof of the main result with a discussion of the native space $V$ from \cref{Native_space}. Given a function $v \in V$ and a multi-index $\alpha \in \N_0^d$ with $\abs{\alpha} \in \set{k_{\min},\dots,k}$, we know that there exists an $\alpha$-th weak derivative $\DN{\alpha}{v} \in \LpLoc{1}{\R^d}$ that happens to lie in $\Lp{2}{\R^d}$. Naively, the definition of $V$ does not tell us anything about the existence and regularity of lower-order derivatives $\DN{\alpha}{v} \in \LpLoc{1}{\R^d}$, where $\abs{\alpha} < k_{\min}$. However, as shown below, it turns out that these lower-order derivatives do exist and lie in $\LpLoc{2}{\R^d}$.

Our proof of this fact uses a Poincar\'e-type inequality of the form $\norm[\Hk{k}{\Omega}]{\cdot} \cleq \seminorm[\Hk{k}{\Omega}]{\cdot} + \norm[\Lp{1}{\Omega}]{\cdot}$, where $\Omega \subseteq \R^d$ is a ball. The following, slightly more general result, will cover all the occurrences of Poincar\'e-type inequalities in the present work.
\begin{lemma} \label{Poincare_inequality}
Let $\Omega \subseteq \R^d$ be a bounded Lipschitz domain. Let $k \in \N_0$ and $Z$ be a normed vector space. Furthermore, let $\fDef{\iota_Z}{\Hk{k}{\Omega}}{Z}$ be a linear continuous operator satisfying the implication $(\iota_Z p = 0 \Rightarrow p=0)$ for every $p \in \Pp{k-1}{\Omega}$. Then, there holds the following Poincar\'e-type inequality:
\begin{equation*}
\forall v \in \Hk{k}{\Omega}: \quad \quad \norm[\Hk{k}{\Omega}]{v} \leq C(d,k,\Omega,Z,\iota_Z) (\seminorm[\Hk{k}{\Omega}]{v} + \norm[Z]{\iota_Z v}).
\end{equation*}
\end{lemma}

\begin{proof}
This can be proved by contradiction using arguments similar to \cite[Section 5.8.1.]{Evans_PDEs}.
%
\end{proof}


Now let us return to the question of existence of the lower-order derivatives $\DN{\alpha}{v}$, $\abs{\alpha}<k_{\min}$, if $v \in V$ is given. Apart from a Poincar\'e inequality, the subsequent proof sports a standard mollification argument with $\CkO{\infty}{\R^d}$-functions, as can be found in any good textbook treating Sobolev spaces (e.g., \cite{Gilbarg_Trudinger}, \cite{Evans_PDEs}). In brief, given a ``standard mollifier'' $\mu \in \CkO{\infty}{\R^d}$ with $\I{\R^d}{\mu(x)}{x}=1$, one defines a sequence $(\mu_n)_{n \in \N} \subseteq \CkO{\infty}{\R^d}$ via $\mu_n(x) := n^d \mu(nx)$. Then, for every function $v \in \LpLoc{1}{\R^d}$, one can show that $\mu_n*v \in \Ck{\infty}{\R^d}$ with $\D{\alpha}{\mu_n*v} = (\DN{\alpha}{\mu_n})*v$ for all $\alpha \in \N_0^d$. Additionally, if $v$ has an $\alpha$-th weak derivative $\DN{\alpha}{v} \in \LpLoc{1}{\R^d}$ for some $\alpha \in \N_0^d$, and if there holds $\restrict{\DN{\alpha}{v}}{\Omega} \in \Lp{p}{\Omega}$ for some $p \in [1,\infty)$ and some open subset $\Omega \subseteq \R^d$, then one can prove that $\restrict{\D{\alpha}{\mu_n*v}}{\Omega} \in \Lp{p}{\Omega}$ and that $\norm[\Lp{p}{\Omega}]{\DN{\alpha}{v}-\D{\alpha}{\mu_n*v}} \xrightarrow{n} 0$.

\begin{lemma} \label{Native_space_Props}
\leavevmode
\begin{enumerate}
\item The function $\bilinear[a]{\cdot}{\cdot}$ defines a symmetric, positive semi-definite bilinear form on $V$ and $\seminorm[a]{\cdot}$ defines a seminorm.
\item There holds $P \subseteq V$.
\item Let $u,v \in V$ such that $u \in P$ or $v \in P$. Then, $\bilinear[a]{u}{v} = 0$.
\item For all $v \in V$, there holds $\seminorm[a]{v} = 0$, if and only if $v \in P$.
\item For all $u,v \in V$, there holds the Cauchy-Schwarz inequality $\abs{\bilinear[a]{u}{v}} \leq \seminorm[a]{u}\seminorm[a]{v}$.
\item There hold the following inclusions (as sets):
\begin{equation*}
\CkO{\infty}{\R^d} \subseteq \Hk{k}{\R^d} \subseteq V \subseteq \HkLoc{k}{\R^d} \subseteq \Ck{0}{\R^d}.
\end{equation*}
In particular, every $v \in V$ is $k$-times weakly differentiable and there holds $\DN{\alpha}{v} \in \LpLoc{2}{\R^d}$ for all $\abs{\alpha} < k_{\min}$, as well as $\DN{\alpha}{v} \in \Lp{2}{\R^d}$ for all $\abs{\alpha} \in \set{k_{\min},\dots,k}$.

\item For all $v \in V$, there hold the bounds
\begin{equation*}
\seminorm[a]{v} \cleq \sum_{l=k_{\min}}^{k} \seminorm[\Hk{l}{\R^d}]{v} \cleq \seminorm[\Hk{k_{\min}}{\R^d}]{v} + \seminorm[\Hk{k}{\R^d}]{v} \cleq \seminorm[a]{v}.
\end{equation*}

\end{enumerate}
\end{lemma}

\begin{proof}
Items $(1)-(5)$ and the inclusions $\CkO{\infty}{\R^d} \subseteq \Hk{k}{\R^d} \subseteq V$ from $(6)$ are elementary. The inclusion $\HkLoc{k}{\R^d} \subseteq \Ck{0}{\R^d}$ follows from $k>d/2$ and a well-known Sobolev embedding. Item $(7)$ follows from the well-known characterization of Sobolev spaces by Fourier transforms (e.g., \cite[Section 5.8.4.]{Evans_PDEs}). Finally, let us sketch the proof of $V \subseteq \HkLoc{k}{\R^d}$: Given $v \in V$, we define the approximants $v_n := \mu_n*v \in \Ck{\infty}{\R^d}$, $n \in \N$, where $(\mu_n)_{n \in \N} \subseteq \CkO{\infty}{\R^d}$ is the sequence of mollifiers from above. We then consider a sequence of bounded Lipschitz domains $B_1 \subseteq B_2 \subseteq B_3 \subseteq \dots \subseteq \R^d$, say, $B_i := \Ball{0}{i}$. For fixed $i \in \N$, we can use \cref{Poincare_inequality}, the convergence $\norm[\Lp{1}{B_i}]{v-v_n} + \seminorm[\Hk{k}{B_i}]{v-v_n} \xrightarrow{n} 0$ and a Cauchy sequence argument to prove that $\restrictN{v}{B_i} \in \Hk{k}{B_i}$. One can then show that the global functions $\fDef{\DN{\alpha}{v}}{\R^d}{\R}$, $\abs{\alpha} \leq k$, defined via $\restrict{\DN{\alpha}{v}}{B_i} := \D{\alpha}{\restrictN{v}{B_i}}$, lie in $\LpLoc{2}{\R^d}$ and represent the $\alpha$-th weak derivative of $v$. In other words, $v$ must lie in $\HkLoc{k}{\R^d}$.
\end{proof}

The next lemma establishes the fact that $\CkO{\infty}{\R^d} \subseteq V$ is dense. We remind the reader of \cref{Coeff_set_C}, where the set $\mvemph{C} \subseteq \R^d$ was introduced.
\begin{lemma} \label{Native_space_Density}
The subset $\CkO{\infty}{\R^d} \subseteq V$ is dense in the following sense: For every $v \in V$, there exists a sequence $(v_n)_{n \in \N} \subseteq \CkO{\infty}{\R^d}$ such that $\seminorm[a]{v-v_n} \xrightarrow{n} 0$ as well as $\abs{\skalar[2]{\mvemph{c}}{E_N(v-v_n)}} \xrightarrow{n} 0$ for all $\mvemph{c} \in \mvemph{C}$.
\end{lemma}

\begin{proof}

%

A straightforward mollification argument shows that $V \cap \Ck{\infty}{\R^d} \subseteq V$ is dense in the stated sense. Therefore, it suffices to show that $\CkO{\infty}{\R^d}$ is dense in $V \cap \Ck{\infty}{\R^d}$. First, we prove the case where $d \geq 2$: Let $v \in V \cap \Ck{\infty}{\R^d}$ and $\kappa \in \CkO{\infty}{\R^d}$ be a cut-off function with $\supp{\kappa} \subseteq B_2$ and $\kappa \equiv 1$ on $B_1$, where $B_n := \Ball{0}{n} \subseteq \R^d$ denotes the ball with integer radius $n \in \N$, centered at the origin. Clearly, the scaled version $\kappa_n := \kappa(\cdot/n) \in \CkO{\infty}{\R^d}$ satisfies $\supp{\kappa_n} \subseteq B_{2n}$ and $\kappa_n \equiv 1$ on $B_n$. Furthermore, $\seminorm[\Wkp{l}{\infty}{\R^d}]{1-\kappa_n} \leq \seminorm[\Wkp{l}{\infty}{\R^d}]{1} + \seminorm[\Wkp{l}{\infty}{\R^d}]{\kappa_n} \cleq n^{-l}$ for all $l \in \N_0$. Next, denote by $\hat{A} := B_2\backslash B_1$ the ``reference annulus'' and let $\fDef{\hat{\Pi}}{\Hk{k_{\min}}{\hat{A}}}{\Pp{k_{\min}-1}{\hat{A}}}$ be the corresponding orthogonal projection. Let $\fDef{F_n}{\hat{A}}{B_{2n}\backslash B_n}$, $F_n(x) := nx$, and consider the polynomial $p_n := \hat{\Pi}(v \circ F_n) \circ F_n^{-1} \in P$ (implicitly extended to $\R^d$). Since we are in the case $d \geq 2$ at the moment, the annulus $\hat{A}$ is connected and we can apply the Poincar\'e-type inequality from \cref{Poincare_inequality} to the bounded Lipschitz domain $\Omega = \hat{A}$, the normed vector space $Z = \Hk{k_{\min}}{\hat{A}}$ and the continuous mapping $\fDef{\iota_Z = \hat{\Pi}}{\Hk{k_{\min}}{\hat{A}}}{\Hk{k_{\min}}{\hat{A}}}$. Then, using a standard scaling argument, we find that
\begin{equation*}
\sum_{j=0}^{k_{\min}} n^j \seminorm[\Hk{j}{B_{2n}\backslash B_n}]{v-p_n} \cleq n^{d/2} \norm[\Hk{k_{\min}}{\hat{A}}]{v \circ F_n - \hat{\Pi}(v \circ F_n)} \cleq n^{d/2} \seminorm[\Hk{k_{\min}}{\hat{A}}]{v \circ F_n} = n^{k_{\min}} \seminorm[\Hk{k_{\min}}{B_{2n}\backslash B_n}]{v}.
\end{equation*}

Now, for every $n \in \N$, we define the approximation $v_n := \kappa_n(v-p_n) \in \CkO{\infty}{\R^d}$. Exploiting that $\supp{\DN{\alpha}{\kappa_n}} \subseteq B_{2n}\backslash B_n$ for all $\alpha \neq 0$, and that $\supp{1-\kappa_n} \subseteq \R^d\backslash B_n$, we estimate 
\begin{eqnarray*}
\seminorm[a]{v-v_n} &\cleq& \sum_{l=k_{\min}}^{k} \seminorm[\Hk{l}{\R^d}]{v-v_n} \quad \stackrel{p_n \in P}{=} \quad \sum_{l=k_{\min}}^{k} \seminorm[\Hk{l}{\R^d}]{(1-\kappa_n)(v-p_n)} \\
&\stackrel{\text{Leibniz}}{\cleq}& \sum_{l=k_{\min}}^{k} \( \sum_{j=0}^{k_{\min}-1} \seminorm[\Wkp{l-j}{\infty}{\R^d}]{\kappa_n} \seminorm[\Hk{j}{B_{2n}\backslash B_n}]{v-p_n} + \sum_{j=k_{\min}}^{l} \seminorm[\Wkp{l-j}{\infty}{\R^d}]{1-\kappa_n} \seminorm[\Hk{j}{\R^d\backslash B_n}]{v} \) \\
&\cleq& \sum_{l=k_{\min}}^{k} \( n^{k_{\min}-l} \seminorm[\Hk{k_{\min}}{B_{2n}\backslash B_n}]{v} + \sum_{j=k_{\min}}^{l} n^{j-l} \seminorm[\Hk{j}{\R^d\backslash B_n}]{v} \) \quad \cleq \quad \sum_{l=k_{\min}}^{k} \seminorm[\Hk{l}{\R^d\backslash B_n}]{v} \xrightarrow{n} 0.
\end{eqnarray*}
The convergence of the last term follows from the fact that $\sum_{l=k_{\min}}^{k} \seminorm[\Hk{l}{\R^d}]{v} < \infty$, since $v \in V$. Finally, let $\mvemph{c} \in \mvemph{C}$. Clearly, there exists an index $n_0 \in \N$, such that $\set{x_1,\dots,x_N} \subseteq B_n$ for all $n \geq n_0$. Since $\kappa_n \equiv 1$ on $B_n$, we get the following identity:
\begin{equation*}
\forall n \geq n_0: \quad \quad \skalar[2]{\mvemph{c}}{E_N(v-v_n)} = \skalar[2]{\mvemph{c}}{E_N(v-\kappa_n(v-p_n))} = \skalar[2]{\mvemph{c}}{E_N p_n} = 0.
\end{equation*}

This settles the question of density in the case $d \geq 2$. Now we turn our attention to the case $d=1$: Since the reference annulus $\hat{A} = (-2,-1) \cup (1,2)$ is not connected, we don't have the Poincar\'e inequality at our disposal. Instead, we show that the inclusions $\CkO{\infty}{\R} \subseteq \Set{v \in \Ck{\infty}{\R}}{\diffN{k_{\min}}{v} \in \CkO{\infty}{\R}} \subseteq V \cap \Ck{\infty}{\R}$ are both dense. To see the latter one, let $v \in V \cap \Ck{\infty}{\R}$ and consider the Taylor-type approximants
\begin{equation*}
v_n(x) := \sum_{l=0}^{k_{\min}-1} \frac{\diffN{l}{v}(0)}{l!} x^l + \frac{1}{(k_{\min}-1)!} \I[x]{0}{(x-t)^{k_{\min}-1} \kappa_n(t) \diffN{k_{\min}}{v}(t)}{t},
\end{equation*}
where $\kappa_n \in \CkO{\infty}{\R}$ is the smooth cut-off function from before. Clearly, $v_n \in \Ck{\infty}{\R}$ and $\diffN{k_{\min}}{v_n} = \kappa_n \diffN{k_{\min}}{v} \in \CkO{\infty}{\R}$, meaning that $v_n$ lies in the supposedly dense subset. To bound the error $\seminorm[a]{v-v_n}$, we apply Leibniz' differentiation rule to the product $\kappa_n \diffN{k_{\min}}{v}$ and exploit the support properties of $\kappa_n$. Skipping the steps that are similar to the multivariate case, we obtain
\begin{equation*}
\seminorm[a]{v-v_n} \cleq \sum_{l=k_{\min}}^{k} \norm[\Lp{2}{\R}]{\diffN{l}{v} - \diffN{l}{v_n}} = \sum_{l=k_{\min}}^{k} \norm[\Lp{2}{\R}]{\diffN{l}{v} - \diff{l-k_{\min}}{\kappa_n \diffN{k_{\min}}{v}}} \cleq \dots \cleq \sum_{l=k_{\min}}^{k} \seminorm[\Hk{l}{\R \backslash B_n}]{v} \xrightarrow{n} 0.
\end{equation*}


On the other hand, since $\restrictN{\kappa_n}{B_n} \equiv 1$, Taylor's Theorem tells us that $v_n(x) = v(x)$ for all $x \in B_n$. Consequently, for sufficiently large $n \in \N$, there must hold $E_N v_n = (v_n(x_i))_{i=1}^{N} = (v(x_i))_{i=1}^{N} = E_N v$. This implies $\abs{\skalar[2]{\mvemph{c}}{E_N(v-v_n)}} \xrightarrow{n} 0$, even for all $\mvemph{c} \in \R^N$, and finishes the proof of the density of $\Set{v \in \Ck{\infty}{\R}}{\diffN{k_{\min}}{v} \in \CkO{\infty}{\R}}$ in $V \cap \Ck{\infty}{\R}$.

Finally, let $v \in \Ck{\infty}{\R}$ with $\diffN{k_{\min}}{v} \in \CkO{\infty}{\R}$. Let $R>0$ such that $\supp{\diffN{k_{\min}}{v}} \subseteq (-R,R)$. Clearly, there must hold $v \in \Pp{k_{\min}-1}{(-\infty,-R)}$ and $v \in \Pp{k_{\min}-1}{(R,\infty)}$. In particular, for all integer $n \geq R$ and all $j \in \N_0$, we can bound
\begin{equation*}
\norm[\Lp{2}{B_{2n} \backslash B_n}]{\diffN{j}{v}} \leq C(v) n^{k_{\min}-1-j+d/2} = C(v) n^{k_{\min}-j-1/2}.
\end{equation*}

We then define the approximants $v_n := \kappa_n v \in \CkO{\infty}{\R}$. Using Leibniz' product rule and the stability $\seminorm[\Wkp{j}{\infty}{\R^d}]{\kappa_n} \cleq n^{-j}$ once again, we estimate
\begin{equation*}
\seminorm[a]{v-v_n} \cleq \sum_{l=k_{\min}}^{k} \( \sum_{j=0}^{l-1} \norm[\Lp{2}{B_{2n}\backslash B_n}]{\diffN{l-j}{\kappa_n} \diffN{j}{v}} + \norm[\Lp{2}{\R \backslash B_n}]{(1-\kappa_n)\diffN{l}{v}} \) \cleq n^{-1/2} + \sum_{l=k_{\min}}^{k} \seminorm[\Hk{l}{\R \backslash B_n}]{v} \xrightarrow{n} 0.
\end{equation*}


Furthermore, since $\restrictN{v_n}{B_n} = \restrictN{v}{B_n}$, and $\set{x_1,\dots,x_N} \subseteq B_n$ for all sufficiently large $n \in \N$, it is clear that $\abs{\skalar[2]{\mvemph{c}}{E_N(v-v_n)}} \xrightarrow{n} 0$, where $\mvemph{c} \in \R^N$. This finishes the proof.
\end{proof}


Unless $k_{\min}=0$, it is now clear that $(V,\bilinear[a]{\cdot}{\cdot})$ is not an inner product space, let alone a Hilbert space. However, using the evaluation operator $\fDef{E_N}{\Ck{0}{\R^d}}{\R^N}$ from \cref{Eval_op}, it is not difficult to find a useful subspace of $V$, where $\bilinear[a]{\cdot}{\cdot}$ is strictly positive definite:
\begin{definition} \label{Hom_native_space}
We define the \emph{homogeneous native space} $V_0 := \Set{v \in V}{E_N v = \mvemph{0}} \subseteq V$.
\end{definition}

Recall that ``$E_N v = \mvemph{0}$'' is equivalent to ``$\forall n \in \set{1,\dots,N}: v(x_n) = 0$''. In particular, the space $V_0$ depends on the number $N$ of interpolation points $x_1,\dots,x_N$, as well as their positions.

\begin{lemma} \label{Hom_native_space_Props}
There holds $V_0 \cap P = \set{0}$. Furthermore, the tuple $(V_0,\bilinear[a]{\cdot}{\cdot})$ constitutes a Hilbert space.
\end{lemma}

\begin{proof}
The identity $V_0 \cap P = \set{0}$ follows from the unisolvency assumption in \cref{Interpol_points}. Since $P$ coincides with the kernel of $\seminorm[a]{\cdot}$ (cf. \cref{Native_space_Props}), the bilinear form $\bilinear[a]{\cdot}{\cdot}$ is positive definite on $V_0$. Finally, let us sketch the proof of completeness: Consider a Cauchy sequence $(v_n)_{n \in \N} \subseteq V_0$ with respect to $\seminorm[a]{\cdot}$ and pick a sequence of bounded Lipschitz domains $B_1 \subseteq B_2 \subseteq B_3 \subseteq \dots \subseteq \R^d$, e.g., $B_i := \Ball{0}{i}$. For fixed $i \in \N$, we use \cref{Poincare_inequality} and the identity $E_N v_n = \mvemph{0} = E_N v_m$ to estimate
\begin{equation*}
\norm[\Hk{k}{B_i}]{v_n-v_m} \leq C(d,k,B_i,(x_l)_l) (\seminorm[\Hk{k}{B_i}]{v_n-v_m} + \norm[\lp{2}{N}]{E_N(v_n-v_m)}) \leq C(d,k,B_i,(x_l)_l) \seminorm[a]{v_n-v_m}.
\end{equation*}
We obtain limits $v_{B_i} \in \Hk{k}{B_i}$, $i \in \N$, and piece them together to define global functions $\fDef{\DN{\alpha}{v}}{\R^d}{\R}$, $\abs{\alpha} \leq k$, by setting $\restrict{\DN{\alpha}{v}}{B_i} := \D{\alpha}{v_{B_i}}$. One can then prove that $v$ lies in $V_0$ and that $\seminorm[a]{v-v_n} \xrightarrow{n} 0$.
\end{proof}

\subsection{An equivalent interpolation problem} \label{SSec:Interpol_problem_2}

%

The homogeneous native space $V_0$ allows us to rewrite the minimization property in \cref{Interpol_problem_1} in terms of an orthogonality relation for $\bilinear[a]{\cdot}{\cdot}$.

\begin{problem} \label{Interpol_problem_2}
Let $f \in V$. Find a function $u \in V$ that satisfies the following interpolation and orthogonality conditions:
\begin{equation*}
E_N u = E_N f, \quad \quad \quad \quad \quad \forall v \in V_0: \quad \bilinear[a]{u}{v} = 0.
\end{equation*}
\end{problem}

Using standard variational techniques, it is not difficult to show that \cref{Interpol_problem_1} and \cref{Interpol_problem_2} are indeed equivalent:
\begin{lemma} \label{Interpol_problems_equivalent}
Let $f \in V$. A function $u \in V$ solves \cref{Interpol_problem_1}, if and only if it solves \cref{Interpol_problem_2}.
\end{lemma}


The orthogonality relation in \cref{Interpol_problem_2} is more suitable for our upcoming error analysis and we will not have to deal with the original minimization property in \cref{Interpol_problem_1} any further. The next lemma answers the question of solvability and uniqueness of solutions (effectively for both problems):
\begin{lemma} \label{Interpol_problem_2_Solvable}
For every $f \in V$, \cref{Interpol_problem_2} has a unique solution $u \in V$. The mapping $f \mapsto u$ is linear and there holds the a priori bound $\seminorm[a]{u} \leq \seminorm[a]{f}$.
\end{lemma}

\begin{proof}
According to the Riesz-Fr\'echet Representation Theorem there exists a unique element $u_0 \in V_0$ such that $\bilinear[a]{u_0}{v} = -\bilinear[a]{f}{v}$ for all $v \in V_0$. Then, $u := u_0 + f \in V$ is the sought solution.
%
%
%
\end{proof}

\subsection{Properties of the LSE} \label{SSec:LSE_Props}

In this section, we will show that the LSE from \cref{SSec:LSE} is indeed regular. Furthermore, we prove the asserted representation $u = \sum_{n=1}^{N} \mvemph{c}_n \phi_n + \sum_{\abs{\alpha}<k_{\min}} \mvemph{d}_{\alpha} \pi_{\alpha}$ of the solution $u \in V$ of \cref{Interpol_problem_1} (or, equivalently, \cref{Interpol_problem_2}).

\begin{definition} \label{Coord_mappings_Phi_Pi}
Denote by $\set{\phi_1,\dots,\phi_N} \subseteq \Ck{0}{\R^d}$ and $\Set{\pi_{\alpha}}{\abs{\alpha}<k_{\min}} \subseteq P$ the ansatz functions from \cref{Ansatz_functions}. We define the corresponding coordinate mappings
\begin{equation*}
\fDefB[\Phi]{\R^N}{\Ck{0}{\R^d}}{\mvemph{c}}{\sum_{n=1}^{N} \mvemph{c}_n\phi_n}, \quad \quad \quad \fDefB[\Pi]{\R^{N_{\min}}}{P}{\mvemph{d}}{\sum_{\abs{\alpha}<k_{\min}} \mvemph{d}_{\alpha}\pi_{\alpha}}.
\end{equation*}
\end{definition}

Recalling from \cref{Native_space_Props} that $P \subseteq V$, it is clear that the operator $\Pi$ always maps into the native space $V$, regardless of the input vector $\mvemph{d} \in \R^{N_{\min}}$. As for the operator $\Phi$, on the other hand, we remind the reader of \cref{Fundamental_solution}: The fundamental solution $\phi \in \Ck{0}{\R^d}$ is not necessarily an element of $V$, so we cannot guarantee $\Phi\mvemph{c} \in V$ for \emph{all} inputs $\mvemph{c} \in \R^N$. However, for the coefficient vectors $\mvemph{c}$ from the subset $\mvemph{C} \subseteq \R^N$ (cf. \cref{Coeff_set_C}), the conformity of $\Phi\mvemph{c}$ is provided by \cref{Fundamental_solution}.

\begin{lemma} \label{Coord_mappings_Phi_Pi_Props}
The operator $\fDef{\Pi}{\R^{N_{\min}}}{P}$ is bijective. For all $\mvemph{c} \in \mvemph{C}$, there holds $\Phi\mvemph{c} \in V$ with $\seminorm[a]{\Phi\mvemph{c}} \cgeq \hMin{}^{k-d/2} \norm[2]{\mvemph{c}}$. Finally, the operators $\Phi$ and $E_N$ (cf. \cref{Eval_op}) are transposed in the following sense:
\begin{equation*}
\forall \mvemph{c} \in \mvemph{C}: \forall v \in V: \quad \quad \bilinear[a]{\Phi\mvemph{c}}{v} = \skalar[2]{\mvemph{c}}{E_N v}.
\end{equation*}
\end{lemma}

\begin{proof}
The bijectivity of $\Pi$ follows from the fact that $\Set{\pi_{\alpha}}{\abs{\alpha}<k_{\min}} \subseteq P$ is a basis. Next, let us prove the transposition formula: For every $\mvemph{c} \in \mvemph{C}$, we know from \cref{Fundamental_solution} that $\Phi\mvemph{c} = \sum_{n=1}^{N} \mvemph{c}_n \phi_n \in V$. Now, for all $v \in \CkO{\infty}{\R^d}$, we compute
\begin{equation*}
\bilinear[a]{\Phi\mvemph{c}}{v} \stackrel{\cref{Bil_form_Diff_op}}{=} \I{\R^d}{(\Phi\mvemph{c})(\DN{2k}{v})}{x} = \sum_{n=1}^{N} \mvemph{c}_n \I{\R^d}{\phi(x-x_n)(\DN{2k}{v})(x)}{x} \stackrel{\cref{Fundamental_solution}}{=} \sum_{n=1}^{N} \mvemph{c}_n v(x_n) = \skalar[2]{\mvemph{c}}{E_N v}.
\end{equation*}
To extend this identity from $\CkO{\infty}{\R^d}$ to all of $V$, we use a simple density argument: Let $\mvemph{c} \in \mvemph{C}$ and $v \in V$. We know from \cref{Native_space_Density} that there exists a sequence $(v_n)_{n \in \N} \subseteq \CkO{\infty}{\R^d}$ such that $\seminorm[a]{v-v_n} \xrightarrow{n} 0$ and $\abs{\skalar[2]{\mvemph{c}}{E_N v}-\skalar[2]{\mvemph{c}}{E_N v_n}} \xrightarrow{n} 0$. Using the Cauchy-Schwarz inequality from \cref{Native_space_Props}, we find that $\abs{\bilinear[a]{\Phi\mvemph{c}}{v}-\bilinear[a]{\Phi\mvemph{c}}{v_n}} \leq \seminorm[a]{\Phi\mvemph{c}} \seminorm[a]{v-v_n} \xrightarrow{n} 0$ as well. Then, since $v_n \in \CkO{\infty}{\R^d}$, 
\begin{equation*}
\bilinear[a]{\Phi\mvemph{c}}{v} = \lim_{n \rightarrow \infty} \bilinear[a]{\Phi\mvemph{c}}{v_n} = \lim_{n \rightarrow \infty} \skalar[2]{\mvemph{c}}{E_N v_n} = \skalar[2]{\mvemph{c}}{E_N v}.
\end{equation*}

Finally, let $\mvemph{c} \in \mvemph{C}$ and denote by $\fDef{\Lambda}{\R^N}{V}$ the operator from \cref{Coord_mapping_Lambda} further below. Anticipating the results from \cref{Coord_mapping_Lambda_Props}, we know that $E_N \Lambda \mvemph{c} = \mvemph{c}$ and that $\seminorm[a]{\Lambda\mvemph{c}} \cleq \hMin{}^{d/2-k} \norm[2]{\mvemph{c}}$. Then, using the transposition formula from above and the Cauchy-Schwarz inequality from \cref{Native_space_Props}, we obtain the bound
\begin{equation*}
\norm[2]{\mvemph{c}}^2 = \skalar[2]{\mvemph{c}}{\mvemph{c}} = \skalar[2]{\mvemph{c}}{E_N \Lambda \mvemph{c}} = \bilinear[a]{\Phi\mvemph{c}}{\Lambda\mvemph{c}} \leq \seminorm[a]{\Phi\mvemph{c}} \seminorm[a]{\Lambda\mvemph{c}} \cleq \hMin{}^{d/2-k} \seminorm[a]{\Phi\mvemph{c}} \norm[2]{\mvemph{c}}.
\end{equation*}
We divide both sides by $\hMin{}^{d/2-k} \norm[2]{\mvemph{c}}$ and get the desired inequality.
\end{proof}

In the case $k_{\min}=0$, the transposition formula in \cref{Coord_mappings_Phi_Pi_Props} reduces to $\bilinear[a]{\phi_n}{v} = v(x_n)$ for all $n \in \set{1,\dots,N}$ and all $v \in V$. This identity is reminiscent of \cref{Fundamental_solution_Ex_0}, where the identity $\bilinear[a]{\phi}{v} = v(0)$, $v \in V$, was used to \emph{define} the basis function $\phi$. Furthermore, let it be said that the transposition formula must be handled with care, if $k_{\min}>0$. The left-hand side reads $\bilinear[a]{\Phi\mvemph{c}}{v} = \bilinear[a]{\sum_{n=1}^{N} \mvemph{c}_n \phi_n}{v}$, with the sum \emph{inside} of $\bilinear[a]{\cdot}{\cdot}$. This is a perfectly valid expression, because \cref{Fundamental_solution} guarantees that $\sum_{n=1}^{N} \mvemph{c}_n \phi_n$ lies in $V$, whenever $\mvemph{c} \in \mvemph{C}$. However, we \emph{cannot} pull the sum out of $\bilinear[a]{\cdot}{\cdot}$. In fact, the expression $\sum_{n=1}^{N} \mvemph{c}_n \bilinear[a]{\phi_n}{v}$ is not well-defined, since $\phi_n$ need not lie in $V$ and may not be plugged into $\bilinear[a]{\cdot}{\cdot}$ on its own.

Next, let us have a look at the properties of the matrices $\mvemph{A} \in \R^{N \times N}$, $\mvemph{B} \in \R^{N_{\min} \times N}$ and $(\begin{smallmatrix} \mvemph{A} & \mvemph{B}^T \\ \mvemph{B} & \mvemph{0} \end{smallmatrix}) \in \R^{(N+N_{\min})\times(N+N_{\min})}$ from \cref{LSE_Matrices}. Using the evaluation operator $\fDef{E_N}{\Ck{0}{\R^d}}{\R^N}$ from \cref{Eval_op} once again, we see that $\mvemph{A}$ can be written in the form $\mvemph{A} = (E_N \phi_1|\dots|E_N \phi_N)$ and that $\mvemph{B}^T$ can be written as $\mvemph{B}^T = (E_N \pi_1|\dots|E_N \pi_{N_{\min}})$, if a linear ordering of the polynomial basis is assumed. The next lemma uses these representations.

\begin{lemma} \label{LSE_invertible}
\leavevmode
\begin{enumerate}
\item For all $\mvemph{d} \in \R^{N_{\min}}$, there holds the identity $\mvemph{B}^T \mvemph{d} = E_N \Pi \mvemph{d}$. In particular, $\mvemph{B}^T$ is injective and $\mvemph{B}$ is surjective. The kernel of $\mvemph{B}$ is given by $\ker{\mvemph{B}} = \mvemph{C}$.

\item For all $\mvemph{c} \in \R^N$, there holds the identity $\mvemph{A}\mvemph{c} = E_N \Phi \mvemph{c}$. Furthermore,
\begin{equation*}
\forall \mvemph{c} \in \mvemph{C}: \quad \quad \skalar[2]{\mvemph{A}{\mvemph{c}}}{\mvemph{c}} \cgeq \hMin{}^{2k-d} \norm[2]{\mvemph{c}}^2.
\end{equation*}

\item The interpolation matrix $(\begin{smallmatrix} \mvemph{A} & \mvemph{B}^T \\ \mvemph{B} & \mvemph{0} \end{smallmatrix})$ is invertible.
\end{enumerate}

\end{lemma}

\begin{proof}
For all $\mvemph{d} \in \R^{N_{\min}}$, we have $\mvemph{B}^T \mvemph{d} = \sum_{\abs{\alpha}<k} \mvemph{d}_{\alpha} E_N \pi_{\alpha} = E_N(\sum_{\abs{\alpha}<k} \mvemph{d}_{\alpha} \pi_{\alpha}) = E_N \Pi \mvemph{d}$. In particular, using the bijectivity of $\Pi$ and \cref{Hom_native_space_Props}, we find that $\ker{\mvemph{B}^T} = \ker{E_N \circ \Pi} = V_0 \cap P = \set{0}$. This proves that $\mvemph{B}^T$ is injective and that $\mvemph{B}$ is surjective. Next, we have
\begin{equation*}
\ker{\mvemph{B}} = \Set{\mvemph{c} \in \R^N}{\mvemph{B}\mvemph{c}=\mvemph{0}} = \Set{\mvemph{c}}{\forall \mvemph{d} \in \R^{N_{\min}}: \skalar[2]{\mvemph{B}\mvemph{c}}{\mvemph{d}} = 0} = \Set{\mvemph{c}}{\forall \mvemph{d} \in \R^{N_{\min}}: \skalar[2]{\mvemph{c}}{E_N \Pi \mvemph{d}} = 0} = \mvemph{C}.
\end{equation*}

Now, for all $\mvemph{c} \in \R^N$, there holds $\mvemph{A}\mvemph{c} = \sum_{n=1}^{N} \mvemph{c}_n E_N \phi_n = E_N(\sum_{n=1}^{N} \mvemph{c}_n \phi_n) = E_N \Phi \mvemph{c}$. In particular, for all $\mvemph{c} \in \mvemph{C}$, 
\begin{equation*}
\skalar[2]{\mvemph{A}\mvemph{c}}{\mvemph{c}} = \skalar[2]{E_N \Phi \mvemph{c}}{\mvemph{c}} \stackrel{\cref{Coord_mappings_Phi_Pi_Props}}{=} \bilinear[a]{\Phi\mvemph{c}}{\Phi\mvemph{c}} = \seminorm[a]{\Phi\mvemph{c}}^2 \stackrel{\cref{Coord_mappings_Phi_Pi_Props}}{\cgeq} \hMin{}^{2k-d} \norm[2]{\mvemph{c}}^2.
\end{equation*}

Finally, consider coefficient vectors $\mvemph{c} \in \R^N$ and $\mvemph{d} \in \R^{N_{\min}}$ such that $(\begin{smallmatrix} \mvemph{A} & \mvemph{B}^T \\ \mvemph{B} & \mvemph{0} \end{smallmatrix}) (\begin{smallmatrix} \mvemph{c} \\ \mvemph{d} \end{smallmatrix}) = (\begin{smallmatrix} \mvemph{0} \\ \mvemph{0} \end{smallmatrix})$. The second row tells us that $\mvemph{B}\mvemph{c} = \mvemph{0}$, so that $\mvemph{c} \in \ker{\mvemph{B}} = \mvemph{C}$. Then, we multiply the first row with $\mvemph{c}$ and get $0 = \skalar{\mvemph{A}\mvemph{c}+\mvemph{B}^T\mvemph{d}}{\mvemph{c}} = \skalar{\mvemph{A}\mvemph{c}}{\mvemph{c}} + \skalar{\mvemph{d}}{\mvemph{B}\mvemph{c}} = \skalar{\mvemph{A}\mvemph{c}}{\mvemph{c}} \cgeq \hMin{}^{2k-d} \norm[2]{\mvemph{c}}^2$, which implies $\mvemph{c} = \mvemph{0}$. Now the first row simplifies to $\mvemph{B}^T\mvemph{d}=\mvemph{0}$, from which we obtain $\mvemph{d}=\mvemph{0}$.
\end{proof}

At this point we have all the ingredients to show that the solution $u \in V$ of \cref{Interpol_problem_2} can be represented by means of the ansatz functions $\phi_n$ and $\pi_{\alpha}$ from \cref{Ansatz_functions}.

\begin{lemma} \label{LSE_cd_to_u}
Let $f \in V$ and set $\mvemph{f} := E_N f \in \R^N$. Denote by $\mvemph{c} \in \R^N$ and $\mvemph{d} \in \R^{\N_{\min}}$ the unique solution of the following LSE:
\begin{equation*}
\(\begin{matrix} \mvemph{A} & \mvemph{B}^T \\ \mvemph{B} & \mvemph{0} \end{matrix}\) \(\begin{matrix} \mvemph{c} \\ \mvemph{d} \end{matrix}\) = \(\begin{matrix} \mvemph{f} \\ \mvemph{0} \end{matrix}\).
\end{equation*}
Then, the function
\begin{equation*}
u := \Phi\mvemph{c} + \Pi\mvemph{d} = \sum_{n=1}^{N} \mvemph{c}_n \phi_n + \sum_{\abs{\alpha}<k_{\min}} \mvemph{d}_{\alpha} \pi_{\alpha} \in V
\end{equation*}
coincides with the unique solution of \cref{Interpol_problem_2}.
\end{lemma}

\begin{proof}
The second row of the LSE tells us that $\mvemph{B}\mvemph{c}=\mvemph{0}$, so that $\mvemph{c} \in \ker{\mvemph{B}} = \mvemph{C}$, by \cref{LSE_invertible}. Owing to \cref{Native_space_Props} and \cref{Coord_mappings_Phi_Pi_Props}, the function $u := \Phi\mvemph{c} + \Pi\mvemph{d}$ then lies in $V$. Using \cref{LSE_invertible} again, the first row of the LSE yields $E_N u = E_N\Phi\mvemph{c} + E_N\Pi\mvemph{d} = \mvemph{A}\mvemph{c} + \mvemph{B}^T\mvemph{d} = \mvemph{f} = E_N f$. Furthermore, for all $v \in V_0$, there holds the following orthogonality:
\begin{equation*}
\bilinear[a]{u}{v} = \bilinear[a]{\Phi\mvemph{c}+\Pi\mvemph{d}}{v} \stackrel{\cref{Native_space_Props}}{=} \bilinear[a]{\Phi\mvemph{c}}{v} \stackrel{\cref{Coord_mappings_Phi_Pi_Props}}{=} \skalar[2]{\mvemph{c}}{E_N v} \stackrel{\cref{Hom_native_space}}{=} \skalar[2]{\mvemph{c}}{\mvemph{0}} = 0.
\end{equation*}
Therefore, the function $u$ is a solution of \cref{Interpol_problem_2} with respect to the data $f$. By uniqueness, $u$ is \emph{the} solution. This concludes the proof.
\end{proof}

\subsection{The representation formula for the inverse matrix} \label{SSec:Rep_formula}

Earlier, in \cref{LSE_invertible}, we developed the fact that the interpolation matrix $(\begin{smallmatrix} \mvemph{A} & \mvemph{B}^T \\ \mvemph{B} & \mvemph{0} \end{smallmatrix})$ from \cref{LSE_Matrices} is invertible. Analogous to \cref{Main_result}, let us write its inverse in the form $(\begin{smallmatrix} \mvemph{A} & \mvemph{B}^T \\ \mvemph{B} & \mvemph{0} \end{smallmatrix})^{-1} = (\begin{smallmatrix} \mvemph{S_{11}} & \mvemph{S_{12}} \\ \mvemph{S_{21}} & \mvemph{S_{22}} \end{smallmatrix})$ with matrices $\mvemph{S_{11}} \in \R^{N \times N}$, $\mvemph{S_{21}} \in \R^{N_{\min} \times N}$, $\mvemph{S_{12}} \in \R^{N \times N_{\min}}$ and $\mvemph{S_{22}} \in \R^{N_{\min} \times N_{\min}}$. In this section, we develop a representation formula for the main block $\mvemph{S_{11}}$.

Recall from \cref{Interpol_problem_2_Solvable} that the mapping $f \mapsto u$ of data $f \in V$ to the solution $u \in V$ of \cref{Interpol_problem_2} is a linear operator. Clearly, this abstract operator must play a prominent role in the sought representation formula.
\begin{definition} \label{Solution_op}
For every $f \in V$, denote by $S_N f \in V$ the unique solution of \cref{Interpol_problem_2}, i.e.,
\begin{equation*}
E_N S_N f = E_N f, \quad \quad \quad \quad \quad \forall v \in V_0: \quad \bilinear[a]{S_N f}{v} = 0.
\end{equation*}

The linear operator $\fDef{S_N}{V}{V}$ is called \emph{solution operator}.
\end{definition}

Next, we need a way to map a given coefficient vector $\mvemph{f} \in \R^N$ to a function $f \in V$ with $E_N f = \mvemph{f}$. Since $\CkO{\infty}{\R^d} \subseteq V$ (cf. \cref{Native_space_Props}), such a function $f$ can easily be constructed from a family of \emph{dual functions} $\lambda_1,\dots,\lambda_N \in \CkO{\infty}{\R^d}$ with $\lambda_m(x_n) = \kronecker{mn}$ (Kronecker-$\delta$). In fact, the function $f := \sum_{n=1}^{N} \mvemph{f}_n \lambda_n \in \CkO{\infty}{\R^d} \subseteq V$ does the job. It is not surprising that the definition of the dual functions depends on the separation distance $\hMin{}$ from \cref{Interpol_points}:
\begin{definition} \label{Dual_functions}
Let $\lambda \in \CkO{\infty}{\R}$ be a function with $\supp{\lambda} \subseteq \Ball[2]{0}{1}$ and $\lambda(0) = 1$. For every $m \in \set{1,\dots,N}$, we define the function $\lambda_m(x) := \lambda((x-x_m)/\hMin{})$. We refer to $\set{\lambda_1,\dots,\lambda_N}$ as the set of \emph{dual functions}.
\end{definition}

The basic properties of the dual functions $\lambda_m$ are summarized in the next lemma. Due to its simplicity, the proof is omitted. We remind the reader of \cref{Bubbles}, where the bubbles $\Omega_m \subseteq \R^d$ were introduced.
\begin{lemma} \label{Dual_functions_Props}
For all $m,n \in \set{1,\dots,N}$, there holds $\lambda_m \in \CkO{\infty}{\R^d} \subseteq V$, $\supp{\lambda_m} \subseteq \Bubbles{m}$ and $\lambda_m(x_n) = \kronecker{mn}$. Furthermore, we have the stability bound $\seminorm[a]{\lambda_m} \cleq \hMin{}^{d/2-k}$.
\end{lemma}


Now that the dual functions $\lambda_m \in \CkO{\infty}{\R^d}$ are properly defined, let us introduce a name for the mapping $\mvemph{f} \mapsto \sum_{n=1}^{N} \mvemph{f}_n \lambda_n$ from before.
\begin{definition} \label{Coord_mapping_Lambda}
We define the operators
\begin{equation*}
\fDefB[\Lambda]{\R^N}{V}{\mvemph{f}}{\sum_{n=1}^{N} \mvemph{f}_n \lambda_n}, \quad \quad \quad \fDefB[\Lambda^T]{V}{\R^N}{v}{(\bilinear[a]{v}{\lambda_n})_{n=1}^{N}}.
\end{equation*}
\end{definition}

Recall from \cref{Native_space_Props} that $V$ is \emph{not} necessarily a Hilbert space, so that $\Lambda^T$ cannot be the proper Hilbert space transpose of $\Lambda$. However, as discussed below, the defining equation for transposed operators is still satisfied.

In the subsequent lemma, we use $\supp{\mvemph{f}} := \Set{n \in \set{1,\dots,N}}{\mvemph{f}_n \neq 0}$ to denote the \emph{support} of a vector $\mvemph{f} \in \R^N$. Remember that $\Bubbles{I} = \bigcup_{n \in I} \Bubbles{n} \subseteq \R^d$ denotes a union of bubbles (cf. \cref{Bubbles}). Furthermore, we make use of the local seminorms $\seminorm[a,\Omega]{\cdot}$, $\Omega \subseteq \R^d$, from \cref{Native_space}.

\begin{lemma} \label{Coord_mapping_Lambda_Props}
The operators $\Lambda$ and $\Lambda^T$ are transposed in the following sense:
\begin{equation*}
\forall v \in V: \forall \mvemph{f} \in \R^N: \quad \quad \bilinear[a]{v}{\Lambda\mvemph{f}} = \skalar[2]{\Lambda^T v}{\mvemph{f}}.
\end{equation*}

Both $\Lambda$ and $\Lambda^T$ preserve locality: For all $\mvemph{f} \in \R^N$, $v \in V$ and $I \subseteq \set{1,\dots,N}$, there holds
\begin{equation*}
\supp{\Lambda \mvemph{f}} \subseteq \Bubbles{\supp{\mvemph{f}}}, \quad \quad \quad \seminorm[a]{\Lambda \mvemph{f}} \cleq \hMin{}^{d/2-k} \norm[2]{\mvemph{f}}, \quad \quad \quad \norm[\lp{2}{I}]{\Lambda^T v} \cleq \hMin{}^{d/2-k} \seminorm[a,\Bubbles{I}]{v}.
\end{equation*}

Finally, for all $\mvemph{f} \in \R^N$, there holds $E_N \Lambda \mvemph{f} = \mvemph{f}$. In particular, we have $v-\Lambda E_N v \in V_0$ for all $v \in V$.
\end{lemma}

\begin{proof}
Let $v \in V$ and $\mvemph{f} \in \R^N$. The transposition formula is straightforward: $\bilinear[a]{v}{\Lambda\mvemph{f}} = \bilinear[a]{v}{\sum_{n=1}^{N} \mvemph{f}_n \lambda_n} = \sum_{n=1}^{N} \bilinear[a]{v}{\lambda_n} \mvemph{f}_n = \skalar[2]{\Lambda^T v}{\mvemph{f}}$. Next, in order to determine the support of $\Lambda\mvemph{f}$, we remember from \cref{Dual_functions_Props} that $\supp{\lambda_n} \subseteq \Bubbles{n}$. Abbreviating $I := \supp{\mvemph{f}}$, we find that $\supp{\Lambda\mvemph{f}} = \supp{\sum_{n \in I} \mvemph{f}_n \lambda_n} \subseteq \bigcup_{n \in I} \supp{\lambda_n} \subseteq \Bubbles{I}$.

In order to see the upper bound for $\seminorm[a]{\Lambda\mvemph{f}}$, we argue that the disjointness of the bubbles $\Bubbles{n}$ implies the orthogonality $\bilinear[a]{\lambda_n}{\lambda_m} = \sum_{l,\alpha} (\sigma_l l!/\alpha!) \skalar[\Lp{2}{\Bubbles{n} \cap \Bubbles{m}}]{\DN{\alpha}{\lambda_n}}{\DN{\alpha}{\lambda_m}} = 0$ for all $m \neq n$. Therefore, with \cref{Dual_functions_Props}, we obtain
\begin{equation*}
\seminorm[a]{\Lambda\mvemph{f}}^2 = \bilinear[a]{\Lambda \mvemph{f}}{\Lambda \mvemph{f}} = \sum_{n,m=1}^{N} \mvemph{f}_n \mvemph{f}_m \bilinear[a]{\lambda_n}{\lambda_m} = \sum_{n=1}^{N} \mvemph{f}_n^2 \seminorm[a]{\lambda_n}^2 \cleq \hMin{}^{d-2k} \norm[2]{\mvemph{f}}^2.
\end{equation*}

Next, consider an arbitrary index set $I \subseteq \set{1,\dots,N}$. To show the upper bound for $\norm[\lp{2}{I}]{\Lambda^T v}$, we use a localized version of the Cauchy-Schwarz inequality from \cref{Native_space_Props} and exploit the disjoint bubbles once again:
\begin{equation*}
\norm[\lp{2}{I}]{\Lambda^T v}^2 = \sum_{n \in I} \bilinear[a]{v}{\lambda_n}^2 \leq \sum_{n \in I} \seminorm[a,\Bubbles{n}]{v}^2 \seminorm[a]{\lambda_n}^2 \cleq \hMin{}^{d-2k} \sum_{n \in I} \seminorm[a,\Bubbles{n}]{v}^2 = \hMin{}^{d-2k} \seminorm[a,\Bubbles{I}]{v}^2.
\end{equation*}

Finally, for all $\mvemph{f} \in \R^N$, the Kronecker-$\delta$-property from \cref{Dual_functions_Props} gives us the identity $E_N \Lambda \mvemph{f} = (\sum_m \mvemph{f}_m \lambda_m(x_n))_{n=1}^{N} = (\mvemph{f}_n)_{n=1}^{N} = \mvemph{f}$. In particular, for all $v \in V$, this implies $E_N(v-\Lambda E_N v) = E_N v - E_N \Lambda E_N v = E_N v - E_N v = \mvemph{0}$, i.e., $v-\Lambda E_N v \in V_0$.
\end{proof}

In \cref{LSE_cd_to_u} we demonstrated how to derive the solution $u \in V$ of \cref{Interpol_problem_2} from the solution $(\mvemph{c},\mvemph{d}) \in \R^N \times \R^{N_{\min}}$ of the LSE in \cref{SSec:LSE}. In the next lemma, we go in the opposite direction, i.e., we construct the coefficient vectors $\mvemph{c}$ and $\mvemph{d}$ from a given solution $u$. Once we have this result, the representation formula for the inverse interpolation matrix is merely a byproduct. 
\begin{lemma} \label{LSE_u_to_cd}
Let $f \in V$ and denote by $u \in V$ the unique solution of \cref{Interpol_problem_2}. Set $\mvemph{f} := E_N f \in \R^N$. Then, the coefficient vectors
\begin{equation*}
\mvemph{c} := \Lambda^T u \in \mvemph{C}, \quad \quad \quad \mvemph{d} := \Pi^{-1}(\identity-\Phi\Lambda^T)u \in \R^{N_{\min}}
\end{equation*}
solve the LSE in \cref{LSE_cd_to_u}.
\end{lemma}

\begin{proof}
First, let us show that indeed $\mvemph{c} \in \mvemph{C}$: For every $q \in P$, there holds $q-\Lambda E_N q \in V_0$ (cf. \cref{Coord_mapping_Lambda_Props}) and thus
\begin{equation*}
\skalar[2]{\mvemph{c}}{E_N q} = \skalar[2]{\Lambda^T u}{E_N q} \stackrel{\cref{Coord_mapping_Lambda_Props}}{=} \bilinear[a]{u}{\Lambda E_N q} \stackrel{\cref{Interpol_problem_2}}{=} \bilinear[a]{u}{q} \stackrel{\cref{Native_space_Props}}{=} 0.
\end{equation*}

Now \cref{Coord_mappings_Phi_Pi_Props} tells us that $\Phi\mvemph{c} \in V$, so that $p := u-\Phi\mvemph{c} \in V$. Using $p-\Lambda E_N p \in V_0$ (cf. \cref{Coord_mapping_Lambda_Props}, again) and the orthogonality of $u$, we get
\begin{equation*}
\seminorm[a]{p}^2 = \bilinear[a]{p}{p} = \bilinear[a]{u}{p} - \bilinear[a]{\Phi\Lambda^T u}{p} \stackrel{\cref{Coord_mappings_Phi_Pi_Props}}{=} \bilinear[a]{u}{p} - \skalar[2]{\Lambda^T u}{E_N p} \stackrel{\cref{Coord_mapping_Lambda_Props}}{=} \bilinear[a]{u}{p-\Lambda E_N p} \stackrel{\cref{Interpol_problem_2}}{=} 0.
\end{equation*}
Due to \cref{Native_space_Props}, this implies $p \in P$, so that $\mvemph{d} := \Pi^{-1} p = \Pi^{-1}(\identity-\Phi\Lambda^T)u \in \R^{N_{\min}}$ is well-defined.

As for the first row of the LSE, we compute
\begin{equation*}
\mvemph{A}\mvemph{c} + \mvemph{B}^T\mvemph{d} \stackrel{\cref{LSE_invertible}}{=} E_N(\Phi\mvemph{c}+\Pi\mvemph{d}) \stackrel{\text{Def.}\mvemph{d}}{=} E_N(\Phi\mvemph{c}+p) \stackrel{\text{Def.}p}{=} E_N u \stackrel{\cref{Interpol_problem_2}}{=} E_N f = \mvemph{f}.
\end{equation*}
Finally, using \cref{LSE_invertible} again, we have $\mvemph{c} \in \mvemph{C} = \ker{\mvemph{B}}$. In other words, $\mvemph{B}\mvemph{c} = \mvemph{0}$, which is precisely the second row of the LSE.
\end{proof}

We close this section with the promised representation formula. The formula establishes a relationship between the interpolation matrix $(\begin{smallmatrix} \mvemph{A} & \mvemph{B}^T \\ \mvemph{B} & \mvemph{0} \end{smallmatrix})$ from \cref{LSE_Matrices}, the coordinate mappings $\fDef{\Phi}{\R^N}{\Ck{0}{\R^d}}$ and $\fDef{\Pi}{\R^{N_{\min}}}{P}$ from \cref{Coord_mappings_Phi_Pi}, the solution operator $\fDef{S_N}{V}{V}$ from \cref{Solution_op} and the operators $\fDef{\Lambda}{\R^N}{V}$ and $\fDef{\Lambda^T}{V}{\R^N}$ from \cref{Coord_mapping_Lambda}.
\begin{corollary} \label{Rep_formula}
For all $\mvemph{f} \in \R^N$, there holds the identity
\begin{equation*}
\(\begin{matrix} \mvemph{A} & \mvemph{B}^T \\ \mvemph{B} & \mvemph{0} \end{matrix}\) \(\begin{matrix} \Lambda^T S_N \Lambda \mvemph{f} \\ \Pi^{-1}(\identity-\Phi\Lambda^T) S_N \Lambda \mvemph{f} \end{matrix}\) = \(\begin{matrix} \mvemph{f} \\ \mvemph{0} \end{matrix}\).
\end{equation*}
In particular, the main block $\mvemph{S_{11}} \in \R^{N \times N}$ from the representation $(\begin{smallmatrix} \mvemph{A} & \mvemph{B}^T \\ \mvemph{B} & \mvemph{0} \end{smallmatrix})^{-1} = (\begin{smallmatrix} \mvemph{S_{11}} & \mvemph{S_{12}} \\ \mvemph{S_{21}} & \mvemph{S_{22}} \end{smallmatrix})$ satisfies the relation
\begin{equation*}
\mvemph{S_{11}} \mvemph{f} = \Lambda^T S_N \Lambda \mvemph{f} \in \mvemph{C}.
\end{equation*}
\end{corollary}

\begin{proof}
Let $\mvemph{f} \in \R^N$ and set $f := \Lambda\mvemph{f} \in V$. According to \cref{Solution_op}, the function $u := S_N f \in V$ is the unique solution of \cref{Interpol_problem_2}. Now \cref{LSE_u_to_cd} tells us that the coefficient vectors $\mvemph{c} := \Lambda^T u = \Lambda^T S_N \Lambda \mvemph{f} \in \mvemph{C}$ and $\mvemph{d} := \Pi^{-1}(\identity-\Phi\Lambda^T)u = \Pi^{-1}(\identity-\Phi\Lambda^T)S_N \Lambda\mvemph{f} \in \R^{N_{\min}}$ solve the LSE from \cref{LSE_cd_to_u}, where the right-hand side is given by $\mvemph{\tilde{f}} := E_N f \in \R^N$. However, we know from \cref{Coord_mapping_Lambda_Props} that $\mvemph{\tilde{f}} = E_N f = E_N \Lambda \mvemph{f} = \mvemph{f}$. This proves the asserted identity.
\end{proof}

\subsection{Reduction from matrix-level to function-level} \label{SSec:Mat_lvl_to_fct_lvl}

In our main result, \cref{Main_result}, we claimed that the main diagonal block $\mvemph{S_{11}}$ of the inverse interpolation matrix can be approximated well by $\mathcal{H}$-matrices. The goal of the subsequent lemma is to reduce this ``matrix-level'' question to an analogous question on the ``function-level''. The majority of the work has already been done in achieving \cref{Rep_formula}, where the action of the matrix $\mvemph{S_{11}}$ on any given vector $\mvemph{f} \in \R^N$ was characterized by the abstract solution operator $\fDef{S_N}{V}{V}$ from \cref{Solution_op}. Furthermore, thanks to the asserted inequality (cf. \cref{Block_partition})
\begin{equation*}
\norm[2]{\mvemph{M}} \leq C \ln(N) \max_{(I,J) \in \BPart} \norm[2]{\restrictN{\mvemph{M}}{I \times J}} \quad \quad \forall \mvemph{M} \in \R^{N \times N},
\end{equation*}
it suffices to derive an error bound for each matrix block $\restrictN{\mvemph{S_{11}}}{I \times J}$, where $(I,J)$ is any given cluster tuple from the sparse hierarchical block partition $\BPart$.


\begin{lemma} \label{Mat_lvl_to_fct_lvl}
Let $\mvemph{S_{11}} \in \R^{N \times N}$ be the main diagonal block of the inverse interpolation matrix from \cref{LSE_Matrices}. Let $I,J \subseteq \set{1,\dots,N}$ and $W \subseteq V$ be a finite-dimensional subspace. Then, there exist an integer $r \leq \dimN{W}$ and matrices $\mvemph{X} \in \R^{I \times r}$ and $\mvemph{Y} \in \R^{J \times r}$, such that there holds the following error bound:
\begin{equation*}
\norm[2]{\restrictN{\mvemph{S_{11}}}{I \times J} - \mvemph{X} \mvemph{Y}^T} \cleq \hMin{}^{d-2k} \* \sup_{\substack{f \in V: \\ \supp{f} \subseteq \Bubbles{J}}} \inf_{w \in W} \frac{\seminorm[a,\Bubbles{I}]{S_N f - w}}{\seminorm[a]{f}}.
\end{equation*}
\end{lemma}

\begin{proof}
Follows from \cref{Coord_mapping_Lambda_Props} and \cref{Rep_formula}. We omit a detailed proof and refer the reader to \cite[Lemma 3.13]{Angleitner_H_matrices_FEM}, where a similar bound was derived in terms of the $L^2$-norm.
\end{proof}

\cref{Mat_lvl_to_fct_lvl} can be interpreted as follows: Given boxes $B,D \subseteq \R^d$ (cf. \cref{Box}), a free parameter $L \in \N$ and a function $f \in V$ with $\supp{f} \subseteq D$, can we construct a subspace $V_{B,D,L} \subseteq V$ such that the dimension $\dim{V_{B,D,L}}$ and the best-approximation error $\inf_{w \in V_{B,D,L}} \seminorm[a,B]{S_N f - w}$ are both ``small'' at the same time? More precisely, is it possible to build $V_{B,D,L}$ in a way such that $\dim{V_{B,D,L}} \cleq L^{c_1}$ and $\inf_{w \in V_{B,D,L}} \seminorm[a,B]{S_N f - w} \cleq 2^{-c_2 L} \seminorm[a]{f}$ for some constants $c_1,c_2>0$?

Over the course of the subsequent sections, we will show that the answer to this question is affirmative (up to a factor $L^k/\hMin{}^k$ in the error bound), if the domains $B,D$ satisfy the \emph{admissibility condition} $\diam{B} \leq \CAdm \dist{B}{D}$ from \cref{Block_partition}.

Let us quickly give an overview of the construction of the space $V_{B,D,L}$:
\begin{enumerate}
\item Due to the admissibility condition, we can safely ``inflate'' the box $B$ in $L$ increments of identical size, before it touches the box $D$. This generates a family of $L$ concentric boxes between $B$ and $D$, i.e., $B \subseteq B_1 \subseteq \dots \subseteq B_L \subseteq \R^d\backslash D$.

\item Starting at the outermost layer $B_L$, we set $u_L := S_N f \in V$. Exploiting the facts $\supp{f} \subseteq D$ and $B_L \cap D = \emptyset$, one can then prove that $u_L$ lies in a certain subspace $\VHarm{B_L} \subseteq V$ associated with the box $B_L$. The properties of this subspace allow us to construct a function $u_{L-1} \in V$ with $\Landau{L^d}$ degrees of freedom that is a good approximation to $u_L$ on the box $B_{L-1}$. More precisely, we can show that
\begin{equation*}
\sum_{l=0}^{k} (\delta/\CSSCO)^l \seminorm[\Hk{l}{B_{L-1}}]{u_L-u_{L-1}} \leq \frac{1}{2} \sum_{l=0}^{k} (\delta/\CSSCO)^l \seminorm[\Hk{l}{B_L}]{u_L},
\end{equation*}
where $\delta>0$ is a parameter proportional to $L^{-1}$ and where $\CSSCO>0$ is a certain constant.

\item Our construction of the approximant $u_{L-1}$ will guarantee that $u_{L-1} \in \VHarm{B_{L-1}}$, and, since these subspaces are nested, we find that the error $u_L - u_{L-1} \in \VHarm{B_{L-1}}$ as well. Again, the properties of this subspace allow us to construct a function $u_{L-2} \in \VHarm{B_{L-2}}$ with $\Landau{L^d}$ degrees of freedom that approximates $u_L-u_{L-1}$ well on the box $B_{L-2}$. Proceeding inwards from the largest box $B_L$ to the smallest box $B_1$, this procedure generates a finite sequence of functions $u_{L-1},u_{L-2},\dots,u_1,u_0 \in V$ such that
\begin{equation*}
\sum_{l=0}^{k} (\delta/\CSSCO)^l \seminorm[\Hk{l}{B}]{S_N f-(u_{L-1}+u_{L-2}+\dots+u_1+u_0)} \leq 2^{-L} \sum_{l=0}^{k} (\delta/\CSSCO)^l \seminorm[\Hk{l}{B_L}]{S_N f}.
\end{equation*}
From here, it is then not difficult to get an error estimate in the native space seminorm $\seminorm[a,B]{\cdot}$.

\end{enumerate}

\subsection{The cut-off operator} \label{SSec:Cut_off_op}

Let us now define precisely what we mean by an \emph{inflated box}:
\begin{definition} \label{Box_infl}
Let $B = \bigtimes_{i=1}^{d} (a_i,b_i)$ with $a_i \leq b_i$ be a box as in \cref{Box}. For every $\delta \geq 0$, we introduce the \emph{inflated box}
\begin{equation*}
\inflateN{B}{\delta} := \bigtimes_{i=1}^{d} (-\delta+a_i,b_i+\delta) \subseteq \R^d.
\end{equation*}
\end{definition}

Note that $\inflateN{B}{\delta}$ is again a box. In particular, we can iterate $\inflate{\inflateN{B}{\delta}}{\delta} = \inflateN{B}{2\delta}$, $\inflate{\inflate{\inflateN{B}{\delta}}{\delta}}{\delta} = \inflateN{B}{3\delta}$, et cetera.

In order for our construction of the space $V_{B,D,L}$ to work, we need a way to ``restrict'' the support of a given function $v \in V$ to a box $B \subseteq \R^d$, without destroying its global smoothness. This can be achieved by multiplying $v$ with a smooth \emph{cut-off function}:
\begin{lemma} \label{Cut_off_fct}
Let $B \subseteq \R^d$ be a box and $\delta>0$. Then, there exists a smooth \emph{cut-off function} $\CutoffFc{B}{\delta}$ with the following properties:
\begin{equation*}
\CutoffFc{B}{\delta} \in \CkO{\infty}{\R^d}, \quad \quad \supp{\CutoffFc{B}{\delta}} \subseteq \inflateN{B}{\delta}, \quad \quad \restrictN{\CutoffFc{B}{\delta}}{B} \equiv 1, \quad \quad 0 \leq \CutoffFc{B}{\delta} \leq 1, \quad \quad \forall l \in \N_0: \seminorm[\Wkp{l}{\infty}{\R^d}]{\CutoffFc{B}{\delta}} \cleq \delta^{-l}.
\end{equation*}
\end{lemma}

\begin{proof}
Write $B = \bigtimes_{i=1}^{d} (a_i,b_i)$ and pick a univariate function $g \in \Ck{\infty}{\R}$ with $0 \leq g \leq 1$, $\restrictN{g}{(-\infty,0]} \equiv 1$ and $\restrictN{g}{[1/2,\infty)} \equiv 0$. Then, the function $\CutoffFc{B}{\delta}(x) := \prod_{i=1}^{d} g((a_i-x_i)/\delta)g((x_i-b_i)/\delta)$, $x \in \R^d$, is a valid choice.
%
\end{proof}

It is convenient to introduce a name for the cut-off process described earlier:
\begin{definition} \label{Cut_off_op}
Let $B \subseteq \R^d$ be a box, $\delta>0$ and $\CutoffFc{B}{\delta} \in \CkO{\infty}{\R^d}$ be the smooth cut-off function from \cref{Cut_off_fct}. We define the corresponding \emph{cut-off operator}
\begin{equation*}
\fDefB[\CutoffOp{B}{\delta}]{V}{\Hk{k}{\R^d}}{v}{\CutoffFc{B}{\delta} v}.
\end{equation*}
\end{definition}

In the following, we summarize the most important facts about the cut-off operator.
\begin{lemma} \label{Cut_off_op_Props}
Let $B \subseteq \R^d$ be a box and $\delta>0$. For all $v \in V$, the linear operator $\CutoffOp{B}{\delta}$ has the cut-off property $\supp{\CutoffOp{B}{\delta} v} \subseteq \inflateN{B}{\delta}$ and the local projection property $\restrict{\CutoffOp{B}{\delta} v}{B} = \restrictN{v}{B}$. Furthermore, 
there holds the stability bound
\begin{equation*}
\sum_{l=0}^{k} \delta^l \seminorm[\Hk{l}{\R^d}]{\CutoffOp{B}{\delta} v} \leq C(d,k) \sum_{l=0}^{k} \delta^l \seminorm[\Hk{l}{\inflateN{B}{\delta}}]{v}.
\end{equation*}
\end{lemma}

\begin{proof}
Denote by $\CutoffFc{B}{\delta} \in \CkO{\infty}{\R^d}$ the smooth cut-off function from \cref{Cut_off_fct} and let $v \in V$. The relations $\supp{\CutoffOp{B}{\delta} v} \subseteq \inflateN{B}{\delta}$ and $\restrict{\CutoffOp{B}{\delta} v}{B} = \restrictN{v}{B}$ readily follow from \cref{Cut_off_fct}. As for the stability bound, we have that $v \in V \subseteq \HkLoc{k}{\R^d}$ (cf. \cref{Native_space_Props}). Then, using Leibniz' product rule for derivatives, we obtain
\begin{equation*}
\sum_{l=0}^{k} \delta^l \seminorm[\Hk{l}{\R^d}]{\CutoffOp{B}{\delta} v} = \sum_{l=0}^{k} \delta^l \seminorm[\Hk{l}{\inflateN{B}{\delta}}]{\CutoffFc{B}{\delta} v} \cleq \sum_{l=0}^{k} \delta^l \sum_{i=0}^{l} \seminorm[\Wkp{l-i}{\infty}{\R^d}]{\CutoffFc{B}{\delta}} \seminorm[\Hk{i}{\inflateN{B}{\delta}}]{v} \cleq \sum_{l=0}^{k} \delta^l \sum_{i=0}^{l} \delta^{i-l} \seminorm[\Hk{i}{\inflateN{B}{\delta}}]{v} \cleq \sum_{l=0}^{k} \delta^l \seminorm[\Hk{l}{\inflateN{B}{\delta}}]{v}.
\end{equation*}
\end{proof}

\subsection{The space $\VHarm{B}$} \label{SSec:VHarm}

The key components of our proof are the subspaces $\VHarm{B} \subseteq V$ to be defined now. They exhibit a number of properties that allow us to find good low-dimensional approximants to their members. 
\begin{definition} \label{Space_VHarm}
Let $B \subseteq \R^d$ open. We say that a function $u \in V$ is \emph{homogeneous and harmonic on $B$}, if it satisfies the following conditions:
\begin{enumerate}
\item For all $n \in \set{1,\dots,N}$ with $x_n \in B$, there holds $u(x_n) = 0$.
\item For all $v \in V_0$ with $\supp{v} \subseteq B$, there holds $\bilinear[a]{u}{v} = 0$.
\end{enumerate}
The subspace of functions $u \in V$ which are homogeneous and harmonic functions on $B$ is denoted by $\VHarm{B} \subseteq V$.

\end{definition}

Loosely speaking, the space $\VHarm{B}$ consists of all functions $u \in V$ that vanish on $\set{x_1,\dots,x_N} \cap B$ and satisfy $\DN{2k}{u} = 0$ on $B\backslash\set{x_1,\dots,x_N}$ in a weak sense (cf. \cref{Diff_op}, \cref{Bil_form_Diff_op}). Furthermore, note that $\VHarm{B}$ is an infinite-dimensional space, in general.

The next lemma establishes the fact that these spaces are nested and interact nicely with the solution operator $\fDef{S_N}{V}{V}$ from \cref{Solution_op} and the cut-off operator $\fDef{\CutoffOp{B}{\delta}}{V}{\Hk{k}{\R^d}}$ from \cref{Cut_off_op}.

\begin{lemma} \label{Space_VHarm_Props}
\leavevmode
\begin{enumerate}
\item For all $B \subseteq B^+ \subseteq \R^d$, there holds $\VHarm{B^+} \subseteq \VHarm{B}$.
\item For all $B,D \subseteq \R^d$ with $B \cap D = \emptyset$ and all $f \in V$ with $\supp{f} \subseteq D$, there holds $S_N f \in \VHarm{B}$.
\item For all $B \subseteq \R^d$, $\delta>0$ and $u \in \VHarm{B}$, there holds $\CutoffOp{B}{\delta} u \in \VHarm{B}$.
\end{enumerate}
\end{lemma}

\begin{proof}
Item $(1)$ is obvious. In order to see $(2)$, consider subsets $B,D$ and functions $f$ as above. Then, for all $n \in \set{1,\dots,N}$ with $x_n \in B$, \cref{Solution_op} and the assumptions on $B,D$ tell us that $(S_N f)(x_n) = f(x_n) = 0$. Furthermore, for all $v \in V_0$ (even with arbitrary support), $\bilinear[a]{S_N f}{v} = 0$. This proves $S_N f \in \VHarm{B}$. Finally, let us prove item $(3)$: Let $B \subseteq \R^d$, $\delta>0$ and $u \in \VHarm{B}$. For all $n \in \set{1,\dots,N}$ with $x_n \in B$, we have $(\CutoffOp{B}{\delta} u)(x_n) = u(x_n) = 0$, owing to the fact that $\restrict{\CutoffOp{B}{\delta} u}{B} = \restrictN{u}{B}$ (cf. \cref{Cut_off_op_Props}). On the other hand, for all $v \in V_0$ with $\supp{v} \subseteq B$, we have
\begin{equation*}
\bilinear[a]{\CutoffOp{B}{\delta}u}{v} \stackrel{\cref{Native_space}}{=} \sum_{l=k_{\min}}^{k} \sigma_l \sum_{\abs{\alpha}=l} \frac{l!}{\alpha!} \skalar[\Lp{2}{B}]{\D{\alpha}{\CutoffOp{B}{\delta}u}}{\DN{\alpha}{v}} \stackrel{\cref{Cut_off_op_Props}}{=} \sum_{l=k_{\min}}^{k} \sigma_l \sum_{\abs{\alpha}=l} \frac{l!}{\alpha!} \skalar[\Lp{2}{B}]{\DN{\alpha}{u}}{\DN{\alpha}{v}} = \bilinear[a]{u}{v} = 0.
\end{equation*}
\end{proof}

Note that the membership $S_N f \in \VHarm{B}$ only uses the interpolation and orthogonality conditions from \cref{Solution_op}. The explicit representation $S_N f = \sum_{n=1}^{N} \mvemph{c}_n \phi_n + \sum_{\abs{\alpha}<k_{\min}} \mvemph{d}_{\alpha} \pi_{\alpha}$ from \cref{LSE_cd_to_u} is completely irrelevant in this context. In fact, \cref{Mat_lvl_to_fct_lvl} already contained the last (implicit) occurrence of the fundamental solution $\phi$ for the remainder of this paper.

While \cref{Space_VHarm_Props} covers the basic, quick-to-prove aspects of the spaces $\VHarm{B}$, we still need two more ingredients. Most importantly, we need a \emph{Caccioppoli-type inequality} for functions $u \in \VHarm{\inflateN{B}{\delta}}$, i.e., we want to bound the $k$-th derivatives of $u$ on the box $B$ by its lower-order derivatives on the slightly larger box $\inflateN{B}{\delta}$. 

\begin{lemma} \label{Space_VHarm_Cacc}
Let $B \subseteq \R^d$ be a box and $\delta>0$ with $\delta \cleq 1$. Then, there holds the \emph{Caccioppoli inequality}:
\begin{equation*}
\forall u \in \VHarm{\inflateN{B}{\delta}}: \quad \quad \delta^k \seminorm[\Hk{k}{B}]{u} \leq C(d,k) \sum_{l=0}^{k-1} \delta^l \seminorm[\Hk{l}{\inflateN{B}{\delta}}]{u}.
\end{equation*}
\end{lemma}

\begin{proof}
Let us abbreviate $\kappa := \CutoffFc{B}{\delta} \in \CkO{\infty}{\R^d}$ for the smooth cut-off function from \cref{Cut_off_fct}. Recall that $\seminorm[\Wkp{l}{\infty}{\R^d}]{\kappa} \cleq \delta^{-l}$ and $\seminorm[\Wkp{l}{\infty}{\R^d}]{\kappa^2} \cleq \sum_{i=0}^{l} \seminorm[\Wkp{i}{\infty}{\R^d}]{\kappa} \seminorm[\Wkp{l-i}{\infty}{\R^d}]{\kappa} \cleq \delta^{-l}$. Furthermore, $\kappa(x_n) = 0$ for all $n \in \set{1,\dots,N}$ satisfying $x_n \notin \inflateN{B}{\delta}$.

Now, let $u \in \VHarm{\inflateN{B}{\delta}}$. From \cref{Space_VHarm} we know that $u(x_n) = 0$ for all $n \in \set{1,\dots,N}$ with $x_n \in \inflateN{B}{\delta}$. In particular, the product $v := \kappa^2 u \in V$ satisfies $v(x_n) = 0$ for \emph{all} $n \in \set{1,\dots,N}$. In other words, $v \in V_0$. Furthermore, using \cref{Cut_off_op_Props}, we have $\supp{v} \subseteq \supp{\kappa} \subseteq \inflateN{B}{\delta}$. This proves that $v$ is an admissible test function for \cref{Space_VHarm}, i.e., $\bilinear[a]{u}{v} = 0$. Plugging in \cref{Native_space}, we get the identity
\begin{equation*}
0 = \bilinear[a]{u}{\kappa^2 u} = \sum_{l=k_{\min}}^{k} \sigma_l \sum_{\abs{\alpha}=l} \frac{l!}{\alpha!} \skalar[\Lp{2}{\R^d}]{\DN{\alpha}{u}}{\D{\alpha}{\kappa^2 u}} = \sum_{l=k_{\min}}^{k} \sigma_l \sum_{\abs{\alpha}=l} \frac{l!}{\alpha!} \sum_{\beta \leq \alpha} \binom{\alpha}{\beta} \skalar[\Lp{2}{\R^d}]{\DN{\alpha}{u}}{\D{\alpha-\beta}{\kappa^2} \DN{\beta}{u}}.
\end{equation*}

We transfer the summands with $\beta<\alpha$ to the other side of the equality and obtain the following expression:
\begin{eqnarray*}
\sum_{l=k_{\min}}^{k} \sigma_l \sum_{\abs{\alpha}=l} \frac{l!}{\alpha!} \norm[\Lp{2}{\R^d}]{\kappa \DN{\alpha}{u}}^2 &=& -\sum_{l=k_{\min}}^{k} \sigma_l \sum_{\abs{\alpha}=l} \frac{l!}{\alpha!} \sum_{\beta<\alpha} \binom{\alpha}{\beta} \skalar[\Lp{2}{\R^d}]{\DN{\alpha}{u}}{\D{\alpha-\beta}{\kappa^2} \DN{\beta}{u}} \\
&\cleq& \sum_{l=k_{\min}}^{k} \sigma_l \sum_{\abs{\alpha}=l} \frac{l!}{\alpha!} \[ \sum_{i=1}^{d} \abs{\skalar[\Lp{2}{\R^d}]{\DN{\alpha}{u}}{\partial{i}{\kappa^2} \DN{\alpha-e_i}{u}}} \\
&& \hspace{15em} + \sum_{\substack{\beta < \alpha, \\ \abs{\beta} \leq \abs{\alpha}-2}} \abs{\skalar[\Lp{2}{\R^d}]{\DN{\alpha}{u}}{\D{\alpha-\beta}{\kappa^2} \DN{\beta}{u}}} \].
\end{eqnarray*}

For the summands in the first sum, we use Young's inequality (with variable $\epsilon>0$):
\begin{eqnarray*}
\abs{\skalar[\Lp{2}{\R^d}]{\DN{\alpha}{u}}{\partial{i}{\kappa^2} \DN{\alpha-e_i}{u}}} &=& 2\abs{\skalar[\Lp{2}{\inflateN{B}{\delta}}]{\kappa \DN{\alpha}{u}}{(\partialN{i}{\kappa}) \DN{\alpha-e_i}{u}}} \\
&\cleq& \norm[\Lp{2}{\R^d}]{\kappa \DN{\alpha}{u}} \norm[\Lp{\infty}{\R^d}]{\partialN{i}{\kappa}} \norm[\Lp{2}{\inflateN{B}{\delta}}]{\DN{\alpha-e_i}{u}} \\
&\cleq& \norm[\Lp{2}{\R^d}]{\kappa \DN{\alpha}{u}} \delta^{-1} \seminorm[\Hk{\abs{\alpha}-1}{\inflateN{B}{\delta}}]{u} \\
&\cleq& \epsilon \norm[\Lp{2}{\R^d}]{\kappa \DN{\alpha}{u}}^2 + \epsilon^{-1} \delta^{-2} \seminorm[\Hk{\abs{\alpha}-1}{\inflateN{B}{\delta}}]{u}^2.
\end{eqnarray*}
Note that, by choosing $\epsilon$ sufficiently small, we can absorb the $\Landau{\epsilon}$-term in the left-hand side of the overall inequality.

For the summands in the second sum, we can pick an index $i \in \set{1,\dots,d}$ with $\alpha_i \geq 1$ (in the case $\alpha=0$, the sum is empty anyways). Then, we perform partial integration with respect to the $i$-th coordinate:
\begin{eqnarray*}
\abs{\skalar[\Lp{2}{\R^d}]{\DN{\alpha}{u}}{\D{\alpha-\beta}{\kappa^2} \DN{\beta}{u}}} &=& \abs{\skalar[\Lp{2}{\inflateN{B}{\delta}}]{\DN{\alpha-e_i}{u}}{\D{\alpha-\beta+e_i}{\kappa^2} \DN{\beta}{u} + \D{\alpha-\beta}{\kappa^2} \DN{\beta+e_i}{u}}} \\
&\leq& \norm[\Lp{2}{\inflateN{B}{\delta}}]{\DN{\alpha-e_i}{u}} (\norm[\Lp{\infty}{\R^d}]{\D{\alpha-\beta+e_i}{\kappa^2}} \norm[\Lp{2}{\inflateN{B}{\delta}}]{\DN{\beta}{u}} + \norm[\Lp{\infty}{\R^d}]{\D{\alpha-\beta}{\kappa^2}} \norm[\Lp{2}{\inflateN{B}{\delta}}]{\DN{\beta+e_i}{u}}) \\
&\cleq& \seminorm[\Hk{\abs{\alpha}-1}{\inflateN{B}{\delta}}]{u} (\delta^{-\abs{\alpha}+\abs{\beta}-1} \seminorm[\Hk{\abs{\beta}}{\inflateN{B}{\delta}}]{u} + \delta^{-\abs{\alpha}+\abs{\beta}} \seminorm[\Hk{\abs{\beta}+1}{\inflateN{B}{\delta}}]{u}) \\
&=& \delta^{-2\abs{\alpha}} (\delta^{\abs{\alpha}-1} \seminorm[\Hk{\abs{\alpha}-1}{\inflateN{B}{\delta}}]{u}) (\delta^{\abs{\beta}} \seminorm[\Hk{\abs{\beta}}{\inflateN{B}{\delta}}]{u} + \delta^{\abs{\beta}+1} \seminorm[\Hk{\abs{\beta}+1}{\inflateN{B}{\delta}}]{u}) \\
&\cleq& \delta^{-2\abs{\alpha}} \sum_{i=0}^{\abs{\alpha}-1} \delta^{2i} \seminorm[\Hk{i}{\inflateN{B}{\delta}}]{u}^2.
\end{eqnarray*}

Finally, we put everything together (exploiting $\kappa \equiv 1$ on $B$):
\begin{equation*}
\delta^{2k} \seminorm[\Hk{k}{B}]{u}^2 \stackrel{\sigma_k>0}{\cleq} \delta^{2k} \sum_{l=k_{\min}}^{k} \sigma_l \sum_{\abs{\alpha}=l} \frac{l!}{\alpha!} \norm[\Lp{2}{\R^d}]{\kappa \DN{\alpha}{u}}^2 \cleq \sum_{l=k_{\min}}^{k} \sigma_l \sum_{\abs{\alpha}=l} \frac{l!}{\alpha!} \underbrace{\delta^{2(k-\abs{\alpha})}}_{\cleq 1} \sum_{i=0}^{\abs{\alpha}-1} \delta^{2i} \seminorm[\Hk{i}{\inflateN{B}{\delta}}]{u}^2 \cleq \sum_{l=0}^{k-1} \delta^{2l} \seminorm[\Hk{l}{\inflateN{B}{\delta}}]{u}^2.
\end{equation*}
This concludes the proof.
\end{proof}

We end the discussion of the spaces $\VHarm{B}$ with the following observation: If $k_{\min}=0$, we remember from \cref{SSec:Interpol_problem_1} that $(V,\bilinear[a]{\cdot}{\cdot})$ is a Hilbert space. One can then show that $\VHarm{B} \subseteq V$ is a closed subspace with respect to the norm $\seminorm[a]{\cdot}$, so that the orthogonal projection from $V$ to $\VHarm{B}$ is well-defined. However, in the general case $k_{\min} \geq 0$, this line of reasoning is not possible anymore and we need to find a different projection $\fDef{P_{B,H}}{V}{\VHarm{B}}$ that is in some sense stable. The idea here is to replace $\bilinear[a]{\cdot}{\cdot}$ by a \emph{strictly} positive definite inner product $\bilinear[b]{\cdot}{\cdot}$ and then use the corresponding orthogonal projection. The new inner product is weighted with a free parameter $H>0$ that will be important for the definition of the low-rank approximation operator $\fDef{\Pi_H}{\Hk{k}{\R^d}}{\Hk{k}{\R^d} \cap \Ck{\infty}{\R^d}}$ in \cref{Low_rank_approx_op} below.

\begin{lemma} \label{Space_VHarm_OP}
Let $B \subseteq \R^d$ be a box with $\meas{B}>0$ and let $H>0$. There exists a linear operator
\begin{equation*}
\fDef{P_{B,H}}{V}{\VHarm{B}}
\end{equation*}
with the following properties:
\begin{enumerate}
\item \emph{Projection:} For all $v \in \VHarm{B}$, there holds $P_{B,H} v = v$.
\item \emph{Stability:} For all $v \in V$, there holds the stability bound
\begin{equation*}
\sum_{l=k_{\min}}^{k} H^l \seminorm[\Hk{l}{\R^d}]{P_{B,H} v} + \sum_{l=0}^{k} H^l \seminorm[\Hk{l}{B}]{P_{B,H} v} \leq C(d,k) \( \sum_{l=k_{\min}}^{k} H^l \seminorm[\Hk{l}{\R^d}]{v} + \sum_{l=0}^{k} H^l \seminorm[\Hk{l}{B}]{v} \).
\end{equation*}
\end{enumerate}
\end{lemma}

\begin{proof}
We equip the native space $V$ from \cref{Native_space} with the following bilinear form:
\begin{equation*}
\forall u,v \in V: \quad \quad \bilinear[b]{u}{v} := \sum_{l=k_{\min}}^{k} H^{2l} \sum_{\abs{\alpha}=l} \skalar[\Lp{2}{\R^d}]{\DN{\alpha}{u}}{\DN{\alpha}{v}} + \sum_{l=0}^{k} H^{2l} \sum_{\abs{\alpha}=l} \skalar[\Lp{2}{B}]{\DN{\alpha}{u}}{\DN{\alpha}{v}}.
\end{equation*}
We remind the reader of \cref{Native_space_Props}, where the inclusion $V \subseteq \HkLoc{k}{\R^d}$ was derived. In particular, the local quantities $\seminorm[\Hk{l}{B}]{u}$ are finite for \emph{all} $l \in \set{0,\dots,k}$, so that $\bilinear[b]{\cdot}{\cdot}$ is indeed well-defined. Note that the assumption $\meas{B}>0$ guarantees the \emph{strict} positive definiteness of $\bilinear[b]{\cdot}{\cdot}$ on all of $V$. The proof of completeness of $(V,\bilinear[b]{\cdot}{\cdot})$ is very similar to the one of \cref{Hom_native_space_Props} and will therefore be omitted. Finally, using the continuous Sobolev embedding $\Hk{k}{B} \subseteq \Ck{0}{B}$ and the Cauchy-Schwarz inequality for $\bilinear[a]{\cdot}{\cdot}$ (cf. \cref{Native_space_Props}), one can show that $\VHarm{B}$ is a closed subspace of $V$ with respect to $\bilinear[b]{\cdot}{\cdot}$. Consequently, the $\bilinear[b]{\cdot}{\cdot}$-orthogonal projection $\fDef{P_{B,H}}{V}{\VHarm{B}}$ is well-defined. The asserted stability bound follows immediately from the fact that $\norm[b]{P_{B,H} v} \leq \norm[b]{v}$ for all $v \in V$. This finishes the proof.
\end{proof}

\subsection{The low-rank approximation operator} \label{SSec:Low_rank_approx_op}

We remind the reader of \cref{SSec:Mat_lvl_to_fct_lvl}, where we argued that we need to construct a certain subspace $V_{B,D,L} \subseteq V$ of low dimension. More precisely, our goal was to achieve an algebraic dimension bound of the form $\dim{V_{B,D,L}} \cleq L^{c_1}$, where $c_1>0$ is some constant. For this purpose, we use an approximation operator of moderate rank with good local approximation properties. Our construction is a ``partition of unity'' method (see, e.g., \cite{Melenk_PUFEM} and \cite{Melenk_PUFEM_2}) on a perfect tensor-product grid with meshsize $H>0$ that spans all of $\R^d$. A pleasing side effect of this method is the fact that the rank of this operator varies ''smoothly`` with respect to the parameter $H$. Once again, we make use of the axes-parallel boxes $B \subseteq \R^d$ and their inflated relatives $\inflateN{B}{\delta} \subseteq \R^d$, $\delta>0$ for the proof (cf. \cref{Box} and \cref{Box_infl}).

\begin{lemma} \label{Low_rank_approx_op}
Let $H>0$ be a free parameter. Then, there exists a linear operator
\begin{equation*}
\fDef{\Pi_H}{\Hk{k}{\R^d}}{\Hk{k}{\R^d} \cap \Ck{\infty}{\R^d}}
\end{equation*}
with the following properties:
\begin{enumerate}
\item \emph{Local rank:} For every box $B \subseteq \R^d$, there holds the dimension bound
\begin{equation*}
\dimN{\Set{\Pi_H v}{v \in \Hk{k}{\R^d} \,\, \text{with} \,\, \supp{v} \subseteq B}} \leq C(d,k)(1+\diam{B}/H)^d.
\end{equation*}

\item \emph{Stability/error bound:} For all $v \in \Hk{k}{\R^d}$, there hold the following stability and error estimates:
\begin{equation*}
\begin{array}{rcl}
\sum_{l=0}^{k} H^l \seminorm[\Hk{l}{\R^d}]{\Pi_H v} &\leq& C(d,k) \sum_{l=0}^{k} H^l \seminorm[\Hk{l}{\R^d}]{v}, \\
\sum_{l=0}^{k} H^l \seminorm[\Hk{l}{\R^d}]{v-\Pi_H v} &\leq& C(d,k) H^k \seminorm[\Hk{k}{\R^d}]{v}.
\end{array}
\end{equation*}
\end{enumerate}
\end{lemma}

\begin{proof}
Let $\mu \in \CkO{\infty}{\R}$ be a ''mollifier`` with $\supp{\mu} \subseteq [-1/4,1/4]$, $\mu \geq 0$ and $\I{\R}{\mu(x)}{x} = 1$. Consider the mollified characteristic function of the interval $[0,1) \subseteq \R$, i.e., $\hat{g}_1(x) := (\mu*\charFunc{[0,1)})(x) = \I{\R}{\mu(y)\charFunc{[0,1)}(x-y)}{y}$. There holds $\hat{g}_1 \in \CkO{\infty}{\R}$, $\supp{\hat{g}_1} \subseteq [-1/4,5/4]$, $\hat{g}_1 \geq 0$ and $\sum_{m \in \Z} \hat{g}_1(x+m) = \I{\R}{\mu(y)\sum_{m \in \Z} \charFunc{y-x+[0,1)}(m)}{y} = \I{\R}{\mu(y)}{y} = 1$ for all $x \in \R$. (Note that, for every given $x$, the sum contains at most two non-zero summands.) Set $\hat{\Omega} := [-1/4,5/4]^d \subseteq \R^d$ (''reference patch``) and $\hat{g}(x) := \prod_{i=1}^{d} \hat{g}_1(x_i)$ for all $x \in \R^d$ (''reference bump function``). There holds $\hat{g} \in \CkO{\infty}{\R^d}$, $\supp{\hat{g}} \subseteq \hat{\Omega}$, $\hat{g} \geq 0$ and $\sum_{m \in \Z^d} \hat{g}(x+m) = 1$ for all $x \in \R^d$ (at most $2^d$ non-zero summands). Furthermore, denote by $\fDef{\hat{\Pi}}{\Hk{k}{\hat{\Omega}}}{\Pp{k-1}{\hat{\Omega}}}$ the orthogonal projection with respect to $\skalar[\Hk{k}{\hat{\Omega}}]{\cdot}{\cdot}$.

Next, let $H>0$ be a given parameter (''meshwidth``). For every $m \in \Z^d$, consider the affine transformations $\fDef{F_m,F_m^{-1}}{\R^d}{\R^d}$, given by $F_m(\hat{x}) := H(\hat{x}+m)$ and $F_m^{-1}(x) := x/H-m$. The $m$-th patch and $m$-th bump function are defined by $\Omega_m := F_m(\hat{\Omega}) \subseteq \R^d$ and $g_m(x) := \hat{g}(F_m^{-1}(x))$ for all $x \in \R^d$. There holds $g_m \in \CkO{\infty}{\R^d}$, $\supp{g_m} \subseteq \Omega_m$, $g_m \geq 0$ and $\sum_{m \in \Z^d} g_m(x) = 1$ for all $x \in \R^d$ (at most $2^d$ non-zero summands). Furthermore, for all $l \in \N_0$, we have the relation $\seminorm[\Wkp{l}{\infty}{\R^d}]{g_m} = H^{-l}\seminorm[\Wkp{l}{\infty}{\R^d}]{\hat{g}} \ceq H^{-l}$.

Now we have everything we need to define the asserted operator:
\begin{equation*}
\fDefB[\Pi_H]{\Hk{k}{\R^d}}{\Hk{k}{\R^d} \cap \Ck{\infty}{\R^d}}{v}{\sum_{m \in \Z^d} (\hat{\Pi}(v \circ F_m) \circ F_m^{-1})g_m}.
\end{equation*}
Note that we implicitly restricted $v \circ F_m$ from $\R^d$ to $\hat{\Omega}$ and extended the polynomial $\hat{\Pi}(v \circ F_m) \circ F_m^{-1}$ from $\Omega_m$ to $\R^d$. Furthermore, for every compact $K \subseteq \R^d$, the sum in $\restrict{\Pi_H v}{K}$ contains only finitely many non-zero summands. In particular, since $\hat{\Pi}(v \circ F_m) \circ F_m^{-1} \in \Pp{k-1}{\R^d} \subseteq \Ck{\infty}{\R^d}$ and $g_m \in \CkO{\infty}{\R^d}$, we have $\Pi_H v \in \Ck{\infty}{\R^d}$, indeed. The fact that $\Pi_H v \in \Hk{k}{\R^d}$ follows from the stability bound below.

As for the dimension bound, let $B \subseteq \R^d$ be a box and denote by $ms(B) := \Set{m \in \Z^d}{B \cap \Omega_m \neq \emptyset}$ the set of indices of the adjacent patches. We will need an upper bound for the cardinality of $ms(B)$ in terms of $\diam{B}$ and $H$. To this end, we use appropriate sub- and supersets of $\bigcup_{m \in ms(B)} \Omega_m$ and exploit the monotonicity of the $d$-dimensional Lebesgue measure, denoted by $\meas{\cdot}$. On one hand, since every patch $\Omega_m$ is itself a box of side length $3H/2$, it is not surprising that $\bigcup_{m \in ms(B)} \Omega_m \subseteq \inflateN{B}{3H/2}$, where $\inflateN{B}{3H/2}$ denotes the inflated box in the sense of \cref{Box_infl}. On the other hand, for every $m \in ms(B)$, consider the $m$-th subpatch $\omega_m := F_m([0,1]^d) \subseteq \Omega_m$. Clearly, these subpatches are pairwise disjoint and fulfill $\bigcup_{m \in ms(B)} \Omega_m \supseteq \bigcup_{m \in ms(B)} \omega_m$. Combining both inclusions, we get the desired bound for $\cardN{ms(B)}$:
\begin{equation*}
\cardN{ms(B)} H^d = \sum_{m \in ms(B)} H^d = \sum_{m \in ms(B)} \meas{\omega_m} = \measB{\bigcup_{m \in ms(B)} \omega_m} \leq \measB{\bigcup_{m \in ms(B)} \Omega_m} \leq \meas{\inflateN{B}{3H/2}} \leq C(d)(\diam{B}+H)^d.
\end{equation*}

Now, for every $v \in \Hk{k}{\R^d}$ with $\supp{v} \subseteq B$, the support properties of the bump functions $g_m$ guarantee that the sum in $\Pi_H v$ only ranges over $ms(B)$, rather than all of $\Z^d$. Therefore,
\begin{equation*}
\begin{array}{rcl}
\dimN{\Set{\Pi_H v}{v \in \Hk{k}{\R^d}, \supp{v} \subseteq B}} &=& \dimN{\Set{\sum_{m \in ms(B)} (\hat{\Pi}(v \circ F_m) \circ F_m^{-1})g_m}{v \in \Hk{k}{\R^d}, \supp{v} \subseteq B}} \\
&\leq& \dimN{\Set{\sum_{m \in ms(B)} v_m g_m}{v_m \in \Pp{k-1}{\R^d}}} \\
&\leq& \cardN{ms(B)} \cdot \dim{\Pp{k-1}{\R^d}} \\
&\leq& C(d,k)(1+\diam{B}/H)^d.
\end{array}
\end{equation*}

Now on to the stability and error estimates. The derivation is based on the following facts:
\begin{enumerate}
\item \emph{Partition of unity:} For every $v \in \Hk{k}{\R^d}$ there holds $v = v \cdot 1 = \sum_{m \in \Z^d} v g_m$.
\item \emph{Scaling argument:} For all $v \in \Hk{k}{\Omega_m}$, there holds $v \circ F_m \in \Hk{k}{\hat{\Omega}}$ with $H^l \seminorm[\Hk{l}{\Omega_m}]{v} \ceq H^{d/2} \seminorm[\Hk{l}{\hat{\Omega}}]{v \circ F_m}$.
\item \emph{Bramble-Hilbert:} For all $v \in \Hk{k}{\hat{\Omega}}$, there holds $\norm[\Hk{k}{\hat{\Omega}}]{v-\hat{\Pi}v} \leq C(d,k,\hat{\Omega}) \seminorm[\Hk{k}{\hat{\Omega}}]{v}$.
\end{enumerate}

Let $v \in \Hk{k}{\R^d}$ and $n \in \Z^d$. Again, we denote by $ms(n) := \Set{m \in \Z^d}{\norm[\infty]{m-n} \leq 1}$ the indices of the patches touching $\Omega_n$. Using $(1)$, $(2)$ and $(3)$, we can establish a \emph{local} error bound first:
\begin{equation*}
\begin{array}{rcl}
\sum_{l=0}^{k} H^l \seminorm[\Hk{l}{\Omega_n}]{v-\Pi_H v} &\stackrel{(1)}{=}& \sum_{l=0}^{k} H^l \seminorm[\Hk{l}{\Omega_n}]{\sum_{m \in ms(n)} (v-\hat{\Pi}(v \circ F_m) \circ F_m^{-1})g_m} \\
&\leq& \sum_{l=0}^{k} H^l \sum_{m \in ms(n)} \seminorm[\Hk{l}{\Omega_n \cap \Omega_m}]{(v-\hat{\Pi}(v \circ F_m) \circ F_m^{-1})g_m} \\
&\cleq& \sum_{l=0}^{k} H^l \sum_{m \in ms(n)} \sum_{j=0}^{l} \seminorm[\Hk{j}{\Omega_m}]{v-\hat{\Pi}(v \circ F_m) \circ F_m^{-1}} \seminorm[\Wkp{l-j}{\infty}{\R^d}]{g_m} \\
&\ceq& \sum_{l=0}^{k} H^l \sum_{m \in ms(n)} \sum_{j=0}^{l} \seminorm[\Hk{j}{\Omega_m}]{v-\hat{\Pi}(v \circ F_m) \circ F_m^{-1}} H^{j-l} \\
&\cleq& \sum_{m \in ms(n)} \sum_{j=0}^{k} H^j \seminorm[\Hk{j}{\Omega_m}]{v-\hat{\Pi}(v \circ F_m) \circ F_m^{-1}} \\
&\stackrel{(2)}{\ceq}& \sum_{m \in ms(n)} H^{d/2} \norm[\Hk{k}{\hat{\Omega}}]{v \circ F_m - \hat{\Pi}(v \circ F_m)} \\
&\stackrel{(3)}{\cleq}& \sum_{m \in ms(n)} H^{d/2} \seminorm[\Hk{k}{\hat{\Omega}}]{v \circ F_m} \\
&\stackrel{(2)}{\cleq}& H^k \sum_{m \in ms(n)} \seminorm[\Hk{k}{\Omega_m}]{v}.
\end{array}
\end{equation*}

To get the desired \emph{global} error bound, we exploit the covering property $\R^d \subseteq \bigcup_{n \in \Z^d} \Omega_n$ and sum up the local error contributions from above. Additionally, we use $\cardN{ms(n)} = 3^d$ and the fact that $\cardN{\Set{n \in \Z^d}{x \in \Omega_n}} \leq 2^d$ for all $x \in \R^d$: 
\begin{equation*}
\begin{array}{rclcl}
\sum_{l=0}^{k} H^{2l} \seminorm[\Hk{l}{\R^d}]{v-\Pi_H v}^2 &\leq& \sum_{n \in \Z^d} \sum_{l=0}^{k} H^{2l} \seminorm[\Hk{l}{\Omega_n}]{v-\Pi_H v}^2 &\cleq& \sum_{n \in \Z^d} H^{2k} \sum_{m \in ms(n)} \seminorm[\Hk{k}{\Omega_m}]{v}^2 \\
&=& 3^d H^{2k} \sum_{n \in \Z^d} \seminorm[\Hk{k}{\Omega_n}]{v}^2 &\leq& 6^d H^{2k} \seminorm[\Hk{k}{\R^d}]{v}^2.
\end{array}
\end{equation*}

Finally, the stability bound can be shown via a simple triangle inequality and the previously established error bound. This concludes the proof.
\end{proof}

\subsection{The single- and multi-step coarsening operators} \label{SSec:Coarse_ops}

This section contains the heart of our proof. Using the subspaces $\VHarm{B} \subseteq V$ from \cref{Space_VHarm}, let us quickly recapitulate the proof outline from \cref{SSec:Mat_lvl_to_fct_lvl}: Given a function $u \in \VHarm{\inflateN{B}{\delta}}$, our goal is to construct a function $\tilde{u} \in \VHarm{B}$ of low ''dimension`` that is a good approximation to $u$ on the box $B \subseteq \R^d$. To this end, we concatenate the cut-off operator $\fDef{\CutoffOp{B}{\delta/2}}{V}{\Hk{k}{\R^d}}$ from \cref{Cut_off_op}, the low-rank approximation operator $\fDef{\Pi_H}{\Hk{k}{\R^d}}{\Hk{k}{\R^d} \cap \Ck{\infty}{\R^d}}$ from \cref{Low_rank_approx_op} and the projection $\fDef{P_{B,H}}{V}{\VHarm{B}}$ from \cref{Space_VHarm_OP}. The cut-off operator guarantees that the right-hand side of the error estimate is a \emph{local} quantity, i.e., $\seminorm[\Hk{l}{\inflateN{B}{\delta}}]{\cdot}$. The low-rank operator $\Pi_H$ is responsible for the reduction of the ''dimension`` of $\tilde{u}$. Finally, the projection $P_{B,H}$ maps the output of $\Pi_H$ back into the space $\VHarm{B}$. 

\begin{theorem} \label{Coarse_op_single}
Let $B \subseteq \R^d$ be a box with $\meas{B}>0$ and $\delta>0$ be a free parameter with $\delta \cleq 1$. Then, there exist a constant $\CSSCO \geq 1$ and a linear \emph{single-step coarsening operator}
\begin{equation*}
\fDef{\CoarseningOp{B}{\delta}}{\VHarm{\inflateN{B}{\delta}}}{\VHarm{B}}
\end{equation*}
with the following properties:
\begin{enumerate}
\item \emph{Rank bound:} The rank is bounded by
\begin{equation*}
\rank{\CoarseningOp{B}{\delta}} \leq C(d,k) (1+\diam{B}/\delta)^d.
\end{equation*}

\item \emph{Approximation error:} For all $u \in \VHarm{\inflateN{B}{\delta}}$, there holds the error bound
\begin{equation*}
\sum_{l=0}^{k} (\delta/\CSSCO)^l \seminorm[\Hk{l}{B}]{u - \CoarseningOp{B}{\delta} u} \leq \frac{1}{2} \sum_{l=0}^{k} (\delta/\CSSCO)^l \seminorm[\Hk{l}{\inflateN{B}{\delta}}]{u}.
\end{equation*}
\end{enumerate}
\end{theorem}

\begin{proof}
Let $B \subseteq \R^d$ and $\delta>0$ as above. The asserted single-step coarsening operator is composed of three operators: First, we need the cut-off operator $\fDef{\CutoffOp{B}{\delta/2}}{V}{\Hk{k}{\R^d}}$ from \cref{Cut_off_op}. Second, let $H>0$ and denote by $\fDef{\Pi_H}{\Hk{k}{\R^d}}{\Hk{k}{\R^d} \cap \Ck{\infty}{\R^d}}$ the operator from \cref{Low_rank_approx_op}. The precise value of $H$ will be chosen during the proof (it will be $H := \delta/\CSSCO$ for some specific constant $\CSSCO \geq 1$). Third, let $\fDef{P_{B,H}}{V}{\VHarm{B}}$ be the projection from \cref{Space_VHarm_OP}. The single-step coarsening operator is then defined as
\begin{equation*}
\fDef{\CoarseningOp{B}{\delta} := P_{B,H} \Pi_H \CutoffOp{B}{\delta/2}}{\VHarm{\inflateN{B}{\delta}}}{\VHarm{B}}.
\end{equation*}
Recall from \cref{Native_space_Props} that $\Hk{k}{\R^d} \subseteq V$, so that the output of $\Pi_H$ is indeed a valid input for $P_{B,H}$.

We start the discussion of $\CoarseningOp{B}{\delta}$ with the error estimate: Let $u \in \VHarm{\inflateN{B}{\delta}}$. On one hand, we know from \cref{Space_VHarm_Props} that $u \in \VHarm{B}$ and that $\CutoffOp{B}{\delta/2} u \in \VHarm{B}$. Since $\fDef{P_{B,H}}{V}{\VHarm{B}}$ is a projection, we know that $P_{B,H} \CutoffOp{B}{\delta/2} u = \CutoffOp{B}{\delta/2} u$. It follows that $\restrictN{u}{B} = \restrict{\CutoffOp{B}{\delta/2} u}{B} = \restrict{P_{B,H} \CutoffOp{B}{\delta/2} u}{B}$, because the cut-off operator $\CutoffOp{B}{\delta/2}$ leaves $u$ untouched on the box $B$ (cf. \cref{Cut_off_op_Props}). On the other hand, we know that $u \in \VHarm{\inflateN{B}{\delta/2}}$ (again by \cref{Space_VHarm_Props}). Since $0<\delta/2 \cleq 1$, we may apply the Caccioppoli inequality from \cref{Space_VHarm_Cacc} to the box $\inflateN{B}{\delta/2}$, the parameter $\delta/2$ and the function $u \in \VHarm{\inflateN{B}{\delta/2}}$:
\begin{equation*}
\begin{array}{rcl}
\sum_{l=0}^{k} H^l \seminorm[\Hk{l}{B}]{u - \CoarseningOp{B}{\delta} u} &=& \sum_{l=0}^{k} H^l \seminorm[\Hk{l}{B}]{P_{B,H} \CutoffOp{B}{\delta/2} u - P_{B,H} \Pi_H \CutoffOp{B}{\delta/2} u} \\
&=& \sum_{l=0}^{k} H^l \seminorm[\Hk{l}{B}]{P_{B,H} (\identity - \Pi_H) \CutoffOp{B}{\delta/2} u} \\
&\stackrel{\cref{Space_VHarm_OP}}{\cleq}& \sum_{l=0}^{k} H^l \seminorm[\Hk{l}{\R^d}]{(\identity - \Pi_H) \CutoffOp{B}{\delta/2} u} \\
&\stackrel{\cref{Low_rank_approx_op}}{\cleq}& H^k \seminorm[\Hk{k}{\R^d}]{\CutoffOp{B}{\delta/2} u} \\
&\stackrel{\cref{Cut_off_op_Props}}{\cleq}& (H/\delta)^k \sum_{l=0}^{k} \delta^l \seminorm[\Hk{l}{\inflateN{B}{\delta/2}}]{u} \\
&\stackrel{\cref{Space_VHarm_Cacc}}{\cleq}& 2^{-1} (H/\delta)^k \sum_{l=0}^{k-1} \delta^l \seminorm[\Hk{l}{\inflateN{B}{\delta}}]{u}.
\end{array}
\end{equation*}

Now, denote by $\CSSCO \geq 1$ the implicit cumulative constant. We choose $H := \delta/\CSSCO>0$. Then,
\begin{equation*}
\sum_{l=0}^{k} H^l \seminorm[\Hk{l}{B}]{u - \CoarseningOp{B}{\delta} u} \leq (\CSSCO/2)(H/\delta)^k \sum_{l=0}^{k-1} \delta^l \seminorm[\Hk{l}{\inflateN{B}{\delta}}]{u} = \frac{1}{2} \sum_{l=0}^{k-1} (1/\CSSCO)^{k-1-l} H^l \seminorm[\Hk{l}{\inflateN{B}{\delta}}]{u} \leq \frac{1}{2} \sum_{l=0}^{k-1} H^l \seminorm[\Hk{l}{\inflateN{B}{\delta}}]{u}.
\end{equation*}

Finally, we turn our attention to the rank bound. For every $u \in \VHarm{\inflateN{B}{\delta}}$, we know from \cref{Cut_off_op_Props} that $\CutoffOp{B}{\delta/2} u \in \Hk{k}{\R^d}$ and that $\supp{\CutoffOp{B}{\delta/2} u} \subseteq \inflateN{B}{\delta/2}$. Using the local rank bound of the operator $\Pi_H$ (cf. \cref{Low_rank_approx_op}), we get
\begin{equation*}
\begin{array}{rclcl}
\rank{\CoarseningOp{B}{\delta}} &=& \dimN{\Set{P_{B,H} \Pi_H \CutoffOp{B}{\delta/2} u}{u \in \VHarm{\inflateN{B}{\delta}}}} &\leq& \dimN{\Set{\Pi_H v}{v \in \Hk{k}{\R^d} \,\, \text{with} \,\, \supp{v} \subseteq \inflateN{B}{\delta/2}}} \\
&\stackrel{\cref{Low_rank_approx_op}}{\leq}& C(d,k)(1+\diam{\inflateN{B}{\delta/2}}/H)^d &\stackrel{\delta \ceq H}{\leq}& C(d,k)(1+\diam{B}/H)^d.
\end{array}
\end{equation*}
This finishes the proof.
\end{proof}

A closer inspection of the previous proof tells us that the sum on the right-hand side of the approximation error actually ranges from $0$ to $k-1$, rather than $k$. However, we won't exploit this fact any further.

Next, let us combine $L \in \N$ single-step coarsening operators into one \emph{multi-step coarsening operator}:
\begin{theorem} \label{Coarse_op_multi}
Let $B \subseteq \R^d$ be a box with $\meas{B}>0$ and $\delta>0$ be a free parameter with $\delta \cleq 1$. Furthermore, let $L \in \N$. Then, there exists a linear \emph{multi-step coarsening operator}
\begin{equation*}
\fDef{\CoarseningOp{B}{\delta,L}}{\VHarm{\inflateN{B}{\delta L}}}{\VHarm{B}}
\end{equation*}
with the following properties:
\begin{enumerate}
\item \emph{Rank bound:} The rank is bounded by
\begin{equation*}
\rank{\CoarseningOp{B}{\delta,L}} \leq C(d,k) (L+\diam{B}/\delta)^{d+1}.
\end{equation*}

\item \emph{Approximation error:} For all $u \in \VHarm{\inflateN{B}{\delta L}}$, there holds the error bound
\begin{equation*}
\norm[\Hk{k}{B}]{u - \CoarseningOp{B}{\delta,L} u} \leq C(d,k) \delta^{-k} 2^{-L} \norm[\Hk{k}{\inflateN{B}{\delta L}}]{u}.
\end{equation*}
\end{enumerate}
\end{theorem}

\begin{proof}
The construction is identical to the one in \cite[Theorem 3.32]{Angleitner_H_matrices_FEM}, and consists only of $k$-times iterative application of the single-step approximation result on nested boxes.
The additional factor $\delta^{-k}$ in the error bound stems from the norm equivalence $\delta^k \norm[\Hk{k}{B}]{\cdot} \cleq \sum_{l=0}^{k} (\delta/\CSSCO)^l \seminorm[\Hk{l}{B}]{\cdot} \cleq \norm[\Hk{k}{B}]{\cdot}$, which makes use of the relations $\delta^k \cleq \delta^l \cleq 1$.
\end{proof}

Looking at the right-hand side of the error bound in \cref{Mat_lvl_to_fct_lvl}, it would be more natural to look for an error bound in the local ''energy``-seminorm $\seminorm[a,B]{\cdot}$ from \cref{Native_space}, i.e., an upper bound for $\seminorm[a,B]{u-\CoarseningOp{B}{\delta,L} u}$ in terms of $\seminorm[a,\inflateN{B}{\delta L}]{u}$. In fact, if we \emph{could} prove $\seminorm[a,B]{u-\CoarseningOp{B}{\delta} u} \leq 2^{-1} \seminorm[a,\inflateN{B}{\delta}]{u}$ in \cref{Coarse_op_single}, then we would obtain $\seminorm[a,B]{u - \CoarseningOp{B}{\delta,L} u} \leq 2^{-L} \seminorm[a,\inflateN{B}{\delta L}]{u}$ without the extra factor $\delta^{-k}$. However, in the proof of \cref{Coarse_op_single} we ''had to`` use the estimate $\delta^k \seminorm[\Hk{k}{\R^d}]{\CutoffOp{B}{\delta/2} u} \cleq \sum_{l=0}^{k} \delta^l \seminorm[\Hk{l}{\inflateN{B}{\delta/2}}]{u}$ from \cref{Cut_off_op_Props}. Here, Leibniz' product rule introduced lower-order derivatives of $u$, irrespective of the degree $k_{\min} \in \set{0,\dots,k}$ and the coefficients $\sigma_l \geq 0$ from \cref{Native_space}. Since the right-hand side of the overall inequality now contains all $l \in \set{0,\dots,k}$ orders of derivatives, we might as well work with a \emph{full} $H^k$-norm (rather than $\seminorm[a,B]{\cdot}$) right from the beginning. It seems that this minor inconvenience cannot be avoided with our method of proof.

\subsection{Putting everything together} \label{SSec:Put_together}

We are finally in the position to prove our main result, \cref{Main_result}. In \cref{SSec:Mat_lvl_to_fct_lvl}, we reduced the original problem of approximating the inverse of the interpolation matrix from \cref{LSE_Matrices} to the problem of finding a low-dimensional subspace $V_{B,D,L} \subseteq V$ with certain approximation properties. Here, it was assumed that the boxes $B,D \subseteq \R^d$ satisfy some admissibility condition. As we shall prove next, the range of the multi-step coarsening operator from \cref{Coarse_op_multi} is a valid choice for $V_{B,D,L}$. 

\begin{theorem} \label{Space_VBDL}
Let $B,D \subseteq \R^d$ be two boxes with $B \cap \set{x_1,\dots,x_N} \neq \emptyset$, $D \cap \set{x_1,\dots,x_N} \neq \emptyset$ and $\hMin{} \leq \diam{B} \leq \CAdm \dist{B}{D}$. Furthermore, let $L \in \N$. Then, there exists a subspace
\begin{equation*}
V_{B,D,L} \subseteq V
\end{equation*}
with the following properties:
\begin{enumerate}
\item \emph{Dimension bound:} There holds the dimension bound
\begin{equation*}
\dimN{V_{B,D,L}} \leq C(d,k,\CAdm) L^{d+1}.
\end{equation*}

\item \emph{Approximation property:} For every $f \in V$ with $\supp{f} \subseteq D$, there holds the error bound
\begin{equation*}
\inf_{v \in V_{B,D,L}} \norm[\Hk{k}{B}]{S_N f - v} \leq C(d,k,\CAdm,(\sigma_l),(\xi_{\alpha})_{\alpha}) (L/\hMin{})^k 2^{-L} \seminorm[a]{f}.
\end{equation*}
\end{enumerate}
\end{theorem}

\begin{proof}
Let $B,D \subseteq \R^d$ and $L \in \N$ as above. Set $\delta := \diam{B}/(2\sqrt{d}\CAdm L) > 0$ and denote by $\fDef{\CoarseningOp{B}{\delta,L}}{\VHarm{\inflateN{B}{\delta L}}}{\VHarm{B}}$ the multi-step coarsening operator from \cref{Coarse_op_multi}. We choose the space
\begin{equation*}
V_{B,D,L} := \ran{\CoarseningOp{B}{\delta,L}} \subseteq \VHarm{B} \subseteq V.
\end{equation*}

Using \cref{Coarse_op_multi} and the definition of $\delta$, we can bound the dimension as follows:
\begin{equation*}
\dimN{V_{B,D,L}} = \rank{\CoarseningOp{B}{\delta,L}} \leq C(d,k) (L+\diam{B}/\delta)^{d+1} \leq C(d,k,\CAdm) L^{d+1}.
\end{equation*}

Finally, let $f \in V$ with $\supp{f} \subseteq D$. In order to show that the error bound from \cref{Coarse_op_multi} is applicable to $S_N f \in V$, we first need to establish the fact that $S_N f \in \VHarm{\inflateN{B}{\delta L}}$. According to \cref{Space_VHarm_Props}, it suffices to prove that the sets $\inflateN{B}{\delta L}$ and $D$ are disjoint. To that end, we choose a point $z \in \closureN{\inflateN{B}{\delta L}}$ with $\dist{\inflateN{B}{\delta L}}{D} = \dist{z}{D}$. Then, $\dist{B}{D} \leq \dist{B}{z} + \dist{z}{D} \leq \sqrt{d} \delta L + \dist{\inflateN{B}{\delta L}}{D}$. Combined with the definition of $\delta$ and the admissibility condition, this yields
\begin{equation*}
\dist{\inflateN{B}{\delta L}}{D} \geq \dist{B}{D} - \sqrt{d} \delta L = \dist{B}{D} - \diam{B}/(2\CAdm) \geq \diam{B}/(2\CAdm) \geq \hMin{} > 0.
\end{equation*}

Now \cref{Space_VHarm_Props} implies $S_N f \in \VHarm{\inflateN{B}{\delta L}}$, so that $\CoarseningOp{B}{\delta,L}(S_N f) \in V_{B,D,L}$. Hence, the error bound from \cref{Coarse_op_multi} is applicable to the function $S_N f$. Then, exploiting $\supp{f} \subseteq D$ and $\delta \cgeq \diam{B}/L \geq \hMin{}/L$, we can estimate
\begin{equation*}
\inf_{v \in V_{B,D,L}} \norm[\Hk{k}{B}]{S_N f - v} \leq \norm[\Hk{k}{B}]{S_N f - \CoarseningOp{B}{\delta,L}(S_N f)} \cleq \delta^{-k} 2^{-L} \norm[\Hk{k}{\inflateN{B}{\delta L}}]{S_N f} \cleq (L/\hMin{})^k 2^{-L} \norm[\Hk{k}{\inflateN{B}{\delta L}}]{f - S_N f}.
\end{equation*}

To bound the remaining local $H^k$-norm, we pick a suitable superset $\Omega \supseteq \inflateN{B}{\delta L}$ and use the Poincar\'e-type inequality from \cref{Poincare_inequality}. To that end, recall from \cref{Interpol_points} that there exists a constant $C>0$, such that $\norm{x_n} \leq C$ for \emph{all} values of $N \geq N_{\min}$ and $n \in \set{1,\dots,N}$. Exploiting the admissibility condition and the fact that both $B$ and $D$ contain at least one of the interpolation points $x_1,\dots,x_N$, it is not difficult to see that $\diam{B} \leq 2\CAdm C$. We obtain the bound
\begin{equation*}
\sup_{x \in \inflateN{B}{\delta L}} \norm{x} \leq \sup_{x \in B} \norm{x} + \sqrt{d}\delta L \leq \max_{n \in \set{1,\dots,N}} \norm{x_n} + \diam{B} + \sqrt{d}\delta L \leq 2(1+\CAdm)C,
\end{equation*}
which tells us that $\inflateN{B}{\delta L} \subseteq \Ball{0}{2(1+\CAdm)C} =: \Omega$. Note that $\Omega$ is independent of $N$ and contains the unisolvent points $\Set{\xi_{\alpha}}{\abs{\alpha}<k_{\min}} \subseteq \R^d$ from \cref{Interpol_points}, which are independent of $N$ as well. Similar to the proof of \cref{Hom_native_space_Props}, we intend to apply the Poincar\'e-type inequality from \cref{Poincare_inequality} to the bounded Lipschitz domain $\Omega$, the normed vector space $Z = \lp{2}{\Set{\alpha}{\abs{\alpha} < k_{\min}}}$ and the operator $\fDef{\iota_Z}{\Hk{k}{\Omega}}{Z}$, $\iota_Z g := (g(\xi_{\alpha}))_{\abs{\alpha} < k_{\min}}$, which is a relative of the evaluation operator $E_N$ from \cref{Eval_op}. The continuity of $\iota_Z$ follows from the Sobolev embedding $\Hk{k}{\Omega} \subseteq \Ck{0}{\Omega}$ and the implication $(\iota_Z p = \mvemph{0} \Rightarrow p=0)$ for all $p \in \Pp{k_{\min}-1}{\Omega}$ can be argued by unisolvency. Then, by \cref{Poincare_inequality} and \cref{Native_space}, we have the following bound:
\begin{equation*}
\forall g \in \Hk{k}{\Omega} \,\, \text{with} \,\, \iota_Z g = \mvemph{0}: \quad \quad \norm[\Hk{k}{\Omega}]{g} \leq C(d,k,\Omega,(\xi_{\alpha})_{\alpha}) \seminorm[\Hk{k}{\Omega}]{g} \leq C(d,k,\Omega,(\xi_{\alpha})_{\alpha},\sigma_k) \seminorm[a,\Omega]{g}.
\end{equation*}

Now, according to \cref{Solution_op}, the function $f-S_N f$ vanishes on \emph{all} interpolation points $\set{x_1,\dots,x_N}$. Owing to \cref{Interpol_points}, we know that $\Set{\xi_{\alpha}}{\abs{\alpha}<k_{\min}} \subseteq \set{x_1,\dots,x_N}$, which implies $\iota_Z(\restrict{f-S_N f}{\Omega}) = \mvemph{0}$. Therefore, we may apply the aforementioned Poincar\'e inequality to the function $\restrict{f-S_N f}{\Omega}$. Finally, using the a priori estimate $\seminorm[a]{S_N f} \leq \seminorm[a]{f}$ from \cref{Interpol_problem_2_Solvable}, we get the following overall bound:
\begin{equation*}
\inf_{v \in V_{B,D,L}} \norm[\Hk{k}{B}]{S_N f - v} \cleq (L/\hMin{})^k 2^{-L} \norm[\Hk{k}{\Omega}]{f - S_N f} \cleq (L/\hMin{})^k 2^{-L} \seminorm[a,\Omega]{f - S_N f} \cleq (L/\hMin{})^k 2^{-L} \seminorm[a]{f}.
\end{equation*}
This concludes the proof.
\end{proof}

Finally, we have everything we need for the proof of \cref{Main_result}.

\begin{proof}
Let $\mvemph{S_{11}} \in \R^{N \times N}$ be the matrix from \cref{Main_result} and $r \in \N$ a given block rank bound. We define the asserted $\mathcal{H}$-matrix approximant $\mvemph{M} \in \R^{N \times N}$ in a block-wise fashion:

First, consider an admissible block $(I,J) \in \BPartAdm$. From \cref{Block_partition} we know that there exist boxes $B_I,B_J \subseteq \R^d$ with $\Bubbles{I} \subseteq B_I$, $\Bubbles{J} \subseteq B_J$ and $\diam{B_I} \leq \CAdm \dist{B_I}{B_J}$. In particular, $B_I$ and $B_J$ both contain at least one interpolation point and there holds $\diam{B_I} \geq \diam{\Bubbles{I}} \geq \hMin{}$, by \cref{Bubbles}. This means that \cref{Space_VBDL} is applicable to $B_I$ and $B_J$. Now, denote by $C>0$ the constant from the dimension bound in \cref{Space_VBDL}. We set $\CExp := 1/(2C^{1/(d+1)}) > 0$ and $L := \lfloor (r/C)^{1/(d+1)} \rfloor \in \N$. Then, \cref{Space_VBDL} provides a subspace $V_{I,J,r} \subseteq V$. We apply \cref{Mat_lvl_to_fct_lvl} to this subspace and get an integer $\tilde{r} \leq \dimN{V_{I,J,r}}$ and matrices $\mvemph{X}_{I,J,r} \in \R^{I \times \tilde{r}}$ and $\mvemph{Y}_{I,J,r} \in \R^{J \times \tilde{r}}$. We set
\begin{equation*}
\restrictN{\mvemph{M}}{I \times J} := \mvemph{X}_{I,J,r} (\mvemph{Y}_{I,J,r})^T.
\end{equation*}

Second, for every small block $(I,J) \in \BPartSmall$, we make the trivial choice
\begin{equation*}
\restrictN{\mvemph{M}}{I \times J} := \restrictN{\mvemph{S_{11}}}{I \times J}.
\end{equation*}

By \cref{H_matrices}, we have $\mvemph{M} \in \HMatrices{\BPart}{\tilde{r}}$ with a block rank bound
\begin{equation*}
\tilde{r} \leq \dimN{V_{I,J,r}} \stackrel{\text{Def.} \, C}{\leq} C L^{d+1} \stackrel{\text{Def.} \, L}{\leq} r.
\end{equation*}

For the error we get
\begin{eqnarray*}
\norm[2]{\mvemph{S_{11}} - \mvemph{M}} &\stackrel{\cref{Block_partition}}{\leq}& \ln(N) \max_{(I,J) \in \BPartAdm} \norm[2]{\restrictN{\mvemph{S_{11}}}{I \times J} - \mvemph{X}_{I,J,r} (\mvemph{Y}_{I,J,r})^T} \\
&\stackrel{\cref{Mat_lvl_to_fct_lvl}}{\cleq}& \ln(N) \hMin{}^{d-2k} \max_{(I,J) \in \BPartAdm} \sup_{\substack{f \in V: \\ \supp{f} \subseteq B_J}} \inf_{w \in V_{I,J,r}} \frac{\seminorm[a,B_I]{S_N f - w}}{\seminorm[a]{f}} \\
&\stackrel{\cref{Space_VBDL}}{\cleq}& \ln(N) \hMin{}^{d-3k} L^k 2^{-L} \\
&\stackrel{\text{Def.} \, L}{\cleq}& \ln(N) \hMin{}^{d-3k} r^{k/(d+1)} \exp(-\ln(2)(r/C)^{1/(d+1)}) \\
&\stackrel{\text{Def.} \, \CExp}{\cleq}& \ln(N) \hMin{}^{d-3k} \exp(-\CExp r^{1/(d+1)}) \\
&\stackrel{\cref{Interpol_points}}{\cleq}& \ln(N) N^{\CCard(3k-d)/d} \exp(-\CExp r^{1/(d+1)}).
\end{eqnarray*}
\end{proof}

\subsection{The fundamental solution (proofs)} \label{SSec:Fundamental_solutions_proofs}

In this subsection we provide the proofs of \cref{Fundamental_solution_Ex_1} and \cref{Fundamental_solution_Ex_2}.

\begin{proof}[Proof of \cref{Fundamental_solution_Ex_1}]
Recall that $d \geq 1$, $k \in \N$ with $k>d/2$, $b \in (0,\infty)$ and that $\phi$ was given as an integral expression:
\begin{equation*}
\forall x \in \R^d: \quad \quad \phi(x) := \frac{(4\pi)^{-d/2}}{\Gamma(k)} \I[\infty]{0}{t^{k-d/2-1} e^{-b^2 t} e^{-\norm{x}^2/(4t)}}{t}.
\end{equation*}

The trivial bound $e^{-\norm{x}^2/(4t)} \leq 1$, the substitution $t = s/b^2$ and the assumption $k>d/2$ show that the integral is indeed well-defined with a uniform upper bound $\abs{\phi(x)} \leq C(d,k) \Gamma(k-d/2) b^{d-2k}$. In particular, Lebesgue's Dominated Convergence Theorem tells us that $\phi \in \Ck{0}{\R^d}$ as well.

In order to show that $\phi$ is a fundamental solution of $\DN{2k}{} = (b^2-\LaplaceN{})^k$, we use standard Fourier techniques. First, using Fubini's Theorem, a straightforward computation reveals $\phi \in \Lp{1}{\R^d}$.
To compute the Fourier transform $\hat{\phi}$ of $\phi$, recall that the Gau\ss kernel $e^{-\norm{\cdot}^2/2}$ is a fixpoint of the Fourier transform and that there holds the relation $\Fourier{e^{-\norm{\cdot}^2/(4t)}}(y) = (2t)^{d/2} e^{-t\norm{y}^2}$ for all $t>0$ and $y \in \R^d$. Using the substitution $t=s/(b^2+\norm{y}^2)$, we obtain the following expression:
\begin{equation*}
\FourierHatN{\phi}(y) = \frac{(4\pi)^{-d/2}}{\Gamma(k)}\I[\infty]{0}{t^{k-d/2-1} e^{-b^2 t} \Fourier{e^{-\norm{\cdot}^2/(4t)}}(y)}{t} = \frac{(2\pi)^{-d/2}}{\Gamma(k)} \I[\infty]{0}{t^{k-1} e^{-(b^2+\norm{y}^2)t}}{t} = \frac{(2\pi)^{-d/2}}{(b^2+\norm{y}^2)^k}.
\end{equation*}

Now, for all $x_0 \in \R^d$ and $v \in \CkO{\infty}{\R^d}$, we have $\FourierInv{\DN{2k}{v}}(y) = \FourierInv{(b^2-\LaplaceN{})^k v}(y) = (b^2+\norm{y}^2)^k \FourierInvHatN{v}(y)$, where $\FourierInvN{}$ and $\FourierInvHatN{\cdot}$ denote the inverse Fourier transform. We know that $(b^2+\norm{\cdot}^2)^k \FourierInvHatN{v} \in \Schwartz \subseteq \Lp{1}{\R^d}$, where $\Schwartz \subseteq \Ck{\infty}{\R^d}$ is the usual \emph{Schwartz class} of rapidly decreasing functions. Then, using the unit imaginary number $\complexI \in \C$, the usual Fourier computation rules and the duality formula $\I{\R^d}{u\FourierHatN{v}}{x} = \I{\R^d}{\FourierHatN{u}v}{x}$ for $u,v \in \Lp{1}{\R^d}$, we obtain
\begin{equation*}
\begin{array}{rclcl}
\I{\R^d}{\phi(x-x_0)(\DN{2k}{v})(x)}{x} &=& \I{\R^d}{\phi(x-x_0)\Fourier{(b^2+\norm{\cdot}^2)^k \FourierInvHatN{v}}(x)}{x} &=& \I{\R^d}{\FourierHatN{\phi}(y) e^{-\complexI\skalar{x_0}{y}} (b^2+\norm{y}^2)^k \FourierInvHatN{v}(y)}{y} \\
&=& (2\pi)^{-d/2} \I{\R^d}{\FourierInvHatN{v}(y) e^{-\complexI\skalar{x_0}{y}}}{y} &=& v(x_0).
\end{array}
\end{equation*}

Finally, let us prove the asserted conformity. Since $k_{\min}=0$, we have $V = \Hk{k}{\R^d} = \Set{v \in \Lp{2}{\R^d}}{(1+\norm{\cdot}^2)^{k/2} \FourierHatN{v} \in \Lp{2}{\R^d}}$, a well-known characterization of Sobolev spaces (e.g., \cite[Section 5.8.4.]{Evans_PDEs}). In fact, the function $\phi$ itself lies in $V$, because $(1+\norm{\cdot}^2)^{k/2} \FourierHatN{\phi} \ceq (1+\norm{\cdot}^k)^{-1} \in \Lp{2}{\R^d}$. This directly implies $\sum_{n=1}^{N} \mvemph{c}_n \phi(\cdot-x_n) \in V$ for all $\mvemph{c} \in \R^N$, which concludes the proof.
\end{proof}

\begin{proof}[Proof of \cref{Fundamental_solution_Ex_2}]
The continuity of the thin-plate spline $\phi$ is obvious and the fact that $\phi$ is a fundamental solution of $(-\LaplaceN{})^k$ was shown in \cite[Theorem 10.36]{Wendland_Scattered}. It remains to show that the function $v := \sum_{n=1}^{N} \mvemph{c}_n \phi(\cdot-x_n)$ lies in the native space $V$, whenever $\mvemph{c} \in \mvemph{C}$. To do so, we use Fourier techniques: According to \cite[Page 161]{Wendland_Scattered}, the function $\phi$ has a \emph{generalized Fourier transform}, given by $\FourierHatN{\phi}(y) := (2\pi)^{-d/2} \norm{y}^{-2k}$. This means that $\I{\R^d}{\phi\FourierHatN{w}}{x} = \I{\R^d}{\FourierHatN{\phi}w}{y}$ for all \emph{homogeneous Schwartz functions} $w \in \SchwartzO{2k}$, where $\SchwartzO{2k} := \Set{w \in \Schwartz}{\forall \abs{\beta}<2k: (\DN{\beta}{w})(0) = 0}$.

Now, consider the auxiliary function $c := \sum_{n=1}^{N} \mvemph{c}_n e^{-\complexI\skalar{x_n}{\cdot}} \in \Ck{\infty}{\R^d}$. For all $\abs{\beta} < k$, there holds $(\DN{\beta}{c})(0) = \sum_{n=1}^{N} \mvemph{c}_n (-\complexI x_n)^{\beta} = (-\complexI)^{\beta} \skalar[2]{\mvemph{c}}{E_N(\cdot)^{\beta}} = 0$, because $\mvemph{c} \in \mvemph{C}$ and $(\cdot)^{\beta} \in \Pp{k-1}{\R^d} = P$. Using Taylor's Theorem, it follows that $\abs{c(y)} \cleq \min\set{\norm{y}^k,1}$ for all $y \in \R^d$. Then, for all $\abs{\alpha}=k$, it is not difficult to see that $(\complexI\,\cdot)^{\alpha} c \FourierHatN{\phi} \in \Lp{2}{\R^d}$, so that $\FourierInv{(\complexI\,\cdot)^{\alpha} c \FourierHatN{\phi}} \in \Lp{2}{\R^d}$ is well-defined (and also real-valued). Furthermore, for all $w \in \CkO{\infty}{\R^d}$, we have $(-\complexI\,\cdot)^{\alpha} c \FourierInvHatN{w} \in \SchwartzO{2k}$. We use the standard Fourier computation rules to derive the following identity:
\begin{equation*}
\I{\R^d}{v(\DN{\alpha}{w})}{x} = \I{\R^d}{\phi\Fourier{(-\complexI\,\cdot)^{\alpha} c \FourierInvHatN{w}}}{x} = \I{\R^d}{\FourierHatN{\phi} (-\complexI\,\cdot)^{\alpha} c \FourierInvHatN{w}}{y} = (-1)^{\alpha} \I{\R^d}{\FourierInv{(\complexI\,\cdot)^{\alpha} c \FourierHatN{\phi}} w}{y}.
\end{equation*}

This proves that the function $v$ has an $\alpha$-th derivative, given by $\DN{\alpha}{v} = \FourierInv{(\complexI\,\cdot)^{\alpha} c \FourierHatN{\phi}} \in \Lp{2}{\R^d}$. Since this is true for all $\abs{\alpha}=k$, we conclude that $v \in \Set{u \in \LpLoc{1}{\R^d}}{\forall \abs{\alpha}=k: \exists \DN{\alpha}{u} \in \Lp{2}{\R^d}} = V$.
\end{proof}

\section{Numerical examples} \label{Sec:Numerical_examples}

We finish this paper with a few numerical examples that were performed in \texttt{MATLAB} and \texttt{H2Lib}, \cite{H2Lib}.

\begin{figure}[H]
\begin{center}
\includegraphics[width=0.28\textwidth, trim=0cm 0cm 0cm 0cm]{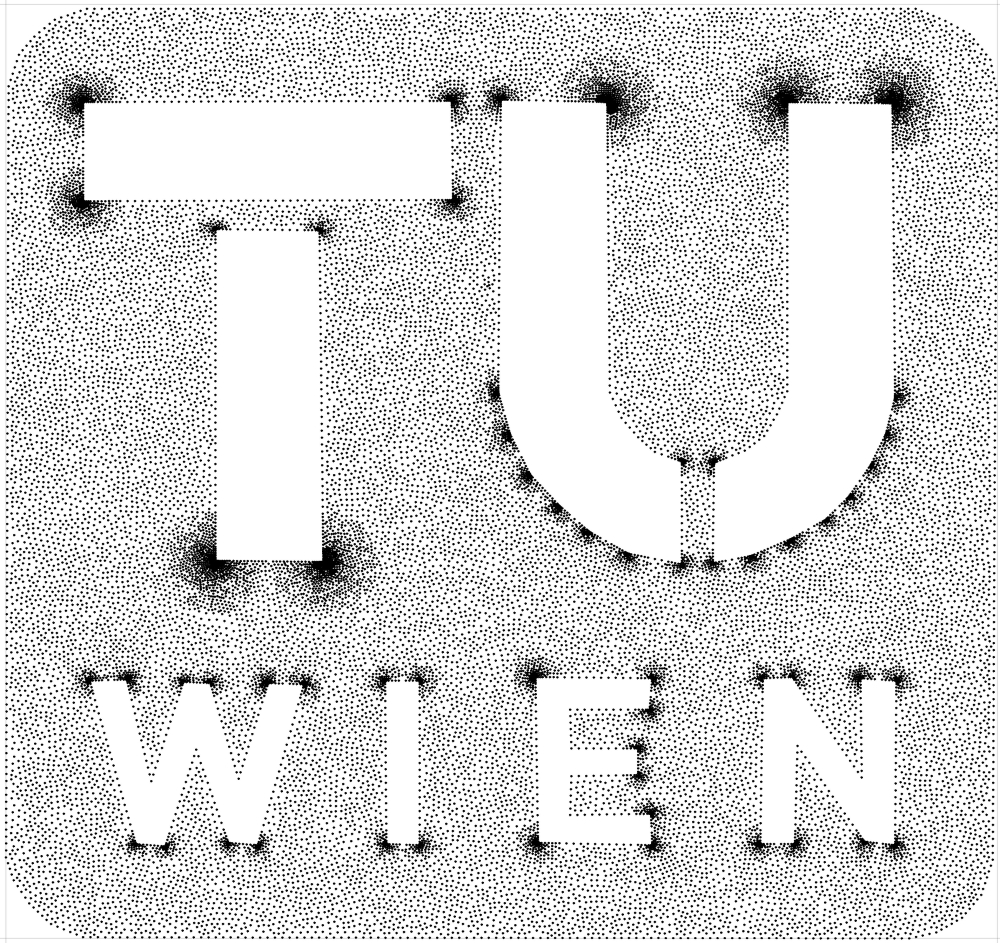} \quad \quad
\includegraphics[width=0.3\textwidth, trim=0cm 0cm 0cm 0cm]{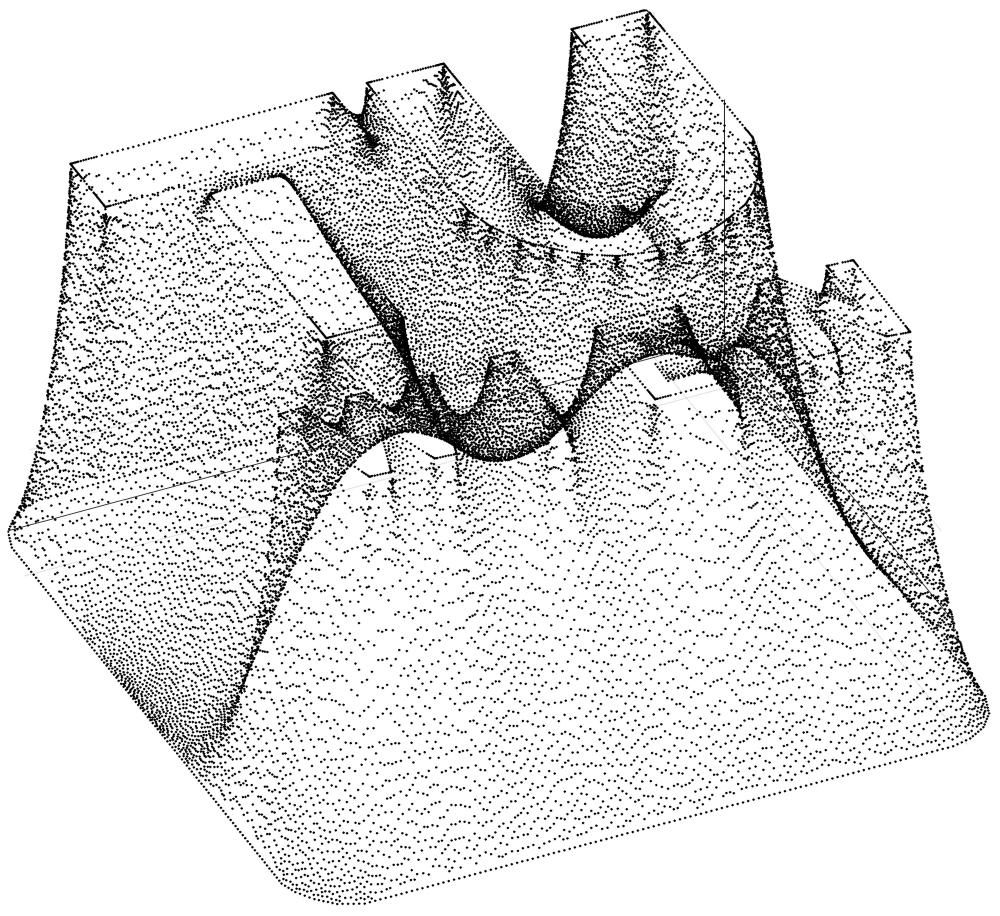} \quad
\includegraphics[width=0.3\textwidth, trim=0cm 0cm 0cm 0cm]{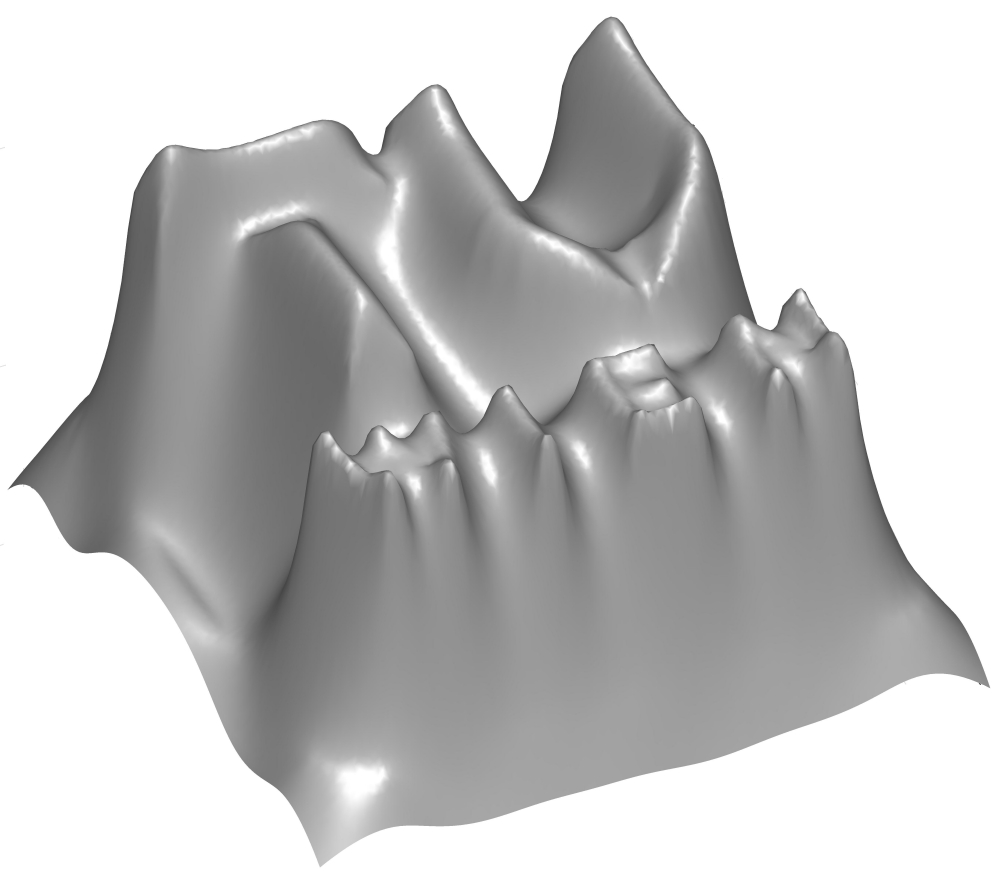}
\caption{Interpolation of smooth data on a non-uniform point distribution.}
\label{Figure_TU_Logo}
\end{center}
\end{figure}

In \cref{Figure_TU_Logo}, we used roughly $N \approx 30.000$ interpolation points in $d=2$ space dimensions. The domain of interest was the exterior of the \emph{TU Wien} logo, i.e., the complement of the letters in the unit square $[0,1] \times [0,1] \subseteq \R^2$. The interpolation points were uniform in the open and algebraically graded towards reentrant corners with a grading exponent $\CCard = 2$ (cf. \cref{Interpol_points}). As for the basis function, we used the thin-plate spline $\phi(x) = \norm{x}^2 \ln\norm{x}$ from \cref{Fundamental_solution_Ex_2}, i.e., $k=k_{\min}=2$. The left image shows the positions $x_n$ of the interpolation points and the one in the middle depicts the pairs $(x_n,f(x_n))$, where the data function $f \in V$ is a smooth indicator function of the letters. On the right-hand side, the solution $u \in V$ of the interpolation problem, \cref{Interpol_problem_1}, is rendered.

\begin{figure}[H]
\begin{center}
\includegraphics[width=0.3\textwidth, trim=0cm 0cm 0cm 0cm]{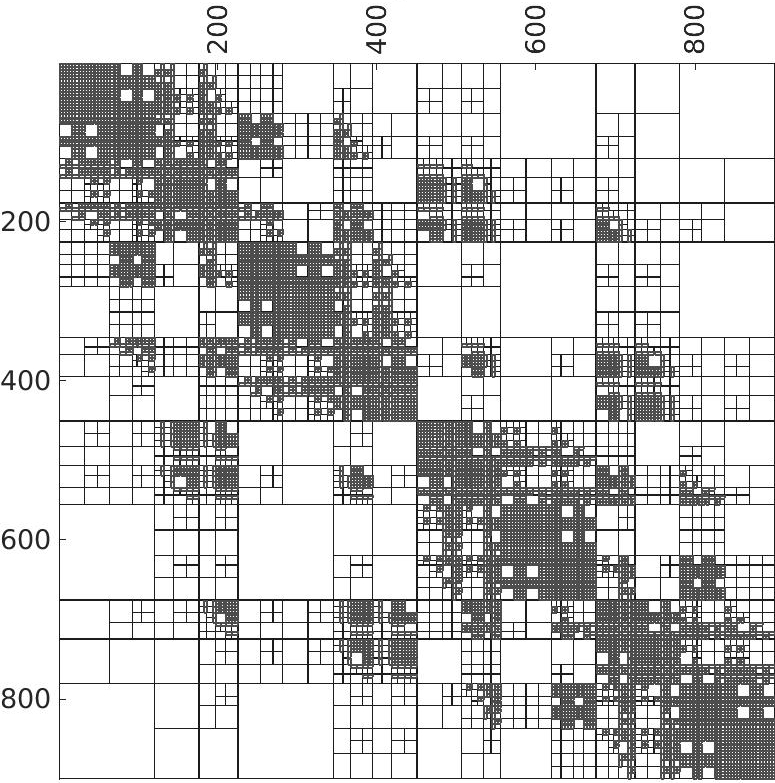} \quad \quad
\includegraphics[width=0.55\textwidth, trim=0cm 0cm 0cm 0cm]{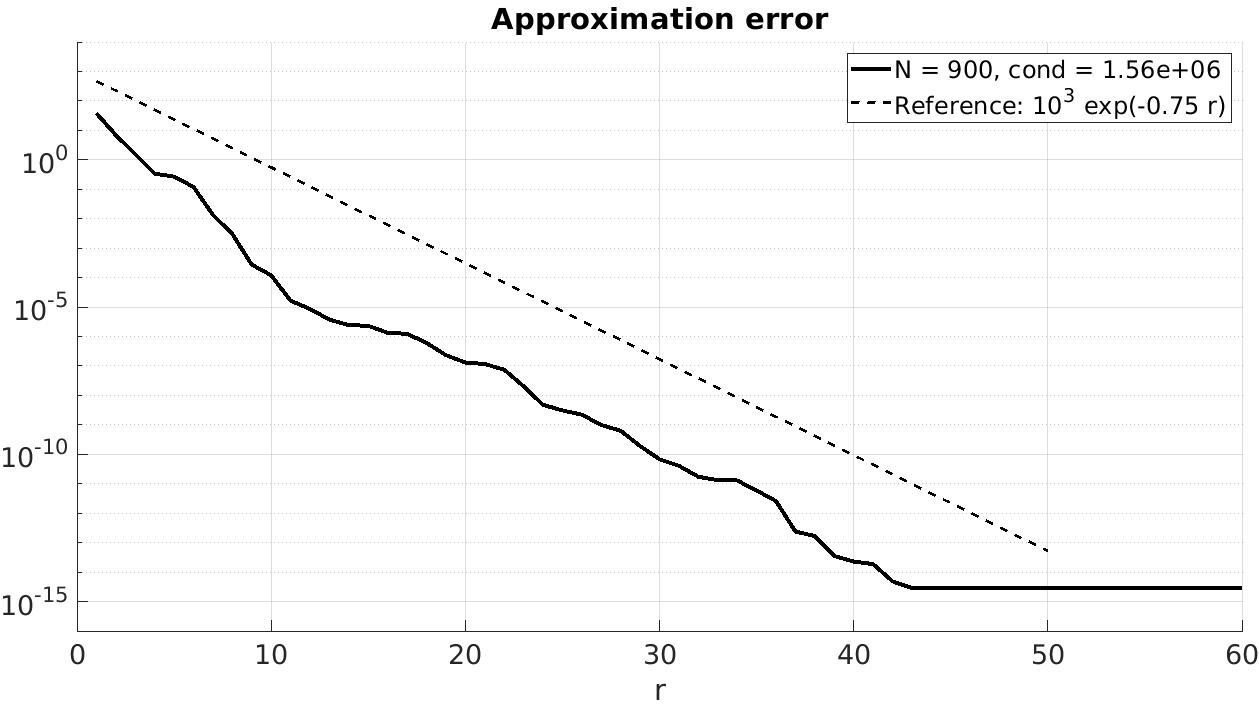}
\caption{A ''typical`` hierarchical block partition and a ''typical`` error plot in $2$D.}
\label{Figure_TPS_2D}
\end{center}
\end{figure}

\cref{Figure_TPS_2D} shows the results of a problem in space dimension $d=2$. The $N=900$ interpolation points $x_n$ produced a uniform $30 \times 30$ grid in the unit square $[0,1] \times [0,1] \subseteq \R^2$, i.e., $\CCard = 1$. Once again the thin plate-spline $\phi(x) = \norm{x}^2 \ln\norm{x}$ with $k=k_{\min}=2$ was employed. In the left image, we can see a ''typical`` sparse hierarchical block partition $\BPart$ in the sense of \cref{Block_partition}. The somewhat ''fractal`` pattern of \emph{small} and \emph{admissible} cluster blocks arises from the fact that we ordered the interpolation points in a row-wise fashion, i.e., $x_1 = (0,0/29)$, $x_{31} = (0,1/29)$, $x_{61} = (0,2/29)$, et cetera.

The right-hand image is empirical evidence that the error bound in our main theorem, \cref{Main_result}, is correct. To generate this plot, the main block $\mvemph{S_{11}} \in \R^{N \times N}$ of the inverse $(\begin{smallmatrix} \mvemph{A} & \mvemph{B}^T \\ \mvemph{B} & \mvemph{0} \end{smallmatrix})^{-1}$ was computed exactly using \texttt{MATLAB}'s built-in inversion routine $\texttt{inv(\dots)}$. Next, for each admissible block $\restrictN{\mvemph{S_{11}}}{I \times J}$, $(I,J) \in \BPartAdm$, we used $\texttt{svds(\dots)}$ to compute the corresponding singular values $\sigma_r(\restrictN{\mvemph{S_{11}}}{I \times J})$, $r \in \set{1,2,\dots,60}$. Truncating the blockwise SVDs at any given rank $r \in \N$, we then assembled the $\mathcal{H}$-matrix $\mvemph{M}_r \in \HMatrices{\BPart}{r}$ (cf. \cref{H_matrices}), the best rank-$r$-approximation to $\mvemph{S_{11}}$. As discussed in our previous work \cite[Section 4]{Angleitner_H_matrices_FEM}, this approach leads to the \emph{computable} error bound
\begin{equation*}
\norm[2]{\mvemph{S_{11}}-\mvemph{M}_r} \cleq \depth{\Tree{N}} \max_{(I,J) \in \BPart} \sigma_{r+1}(\restrictN{\mvemph{S_{11}}}{I \times J}),
\end{equation*}
where $\depth{\Tree{N}}$ denotes the cluster tree depth previously mentioned in \cref{SSec:Hierarchical_matrices}. The semi-logarithmic error plot depicts the computable error bound along with a dashed reference line. The apparent similarity suggests a relation of the form $\norm[2]{\mvemph{S_{11}}-\mvemph{M}_r} \cleq C(N) \exp(-\CExp r)$, which is even better than our theoretical prediction $C(N) \exp(-\CExp r^{1/3})$.

On a side note, we mention that the standard $16$-digit precision arithmetic in \texttt{MATLAB} was not enough to generate a conclusive error plot. As is well-established in the literature (see, e.g., \cite[Chapter 12]{Wendland_Scattered}), the condition number of the interpolation matrix $(\begin{smallmatrix} \mvemph{A} & \mvemph{B}^T \\ \mvemph{B} & \mvemph{0} \end{smallmatrix})$ scales very poorly with respect to the separation distance $\hMin{}$ introduced in \cref{Interpol_points}. To overcome this fundamental problem, we used \texttt{MATLAB}'s \emph{variable-precision arithmetic} \texttt{vpa(\dots)} with $32$ digits. This brute-force approach allowed us to carry out the explicit matrix inversion with sufficient accuracy. 

\begin{figure}[H]
\begin{center}
\includegraphics[width=0.3\textwidth, trim=0cm 0cm 0cm 0cm]{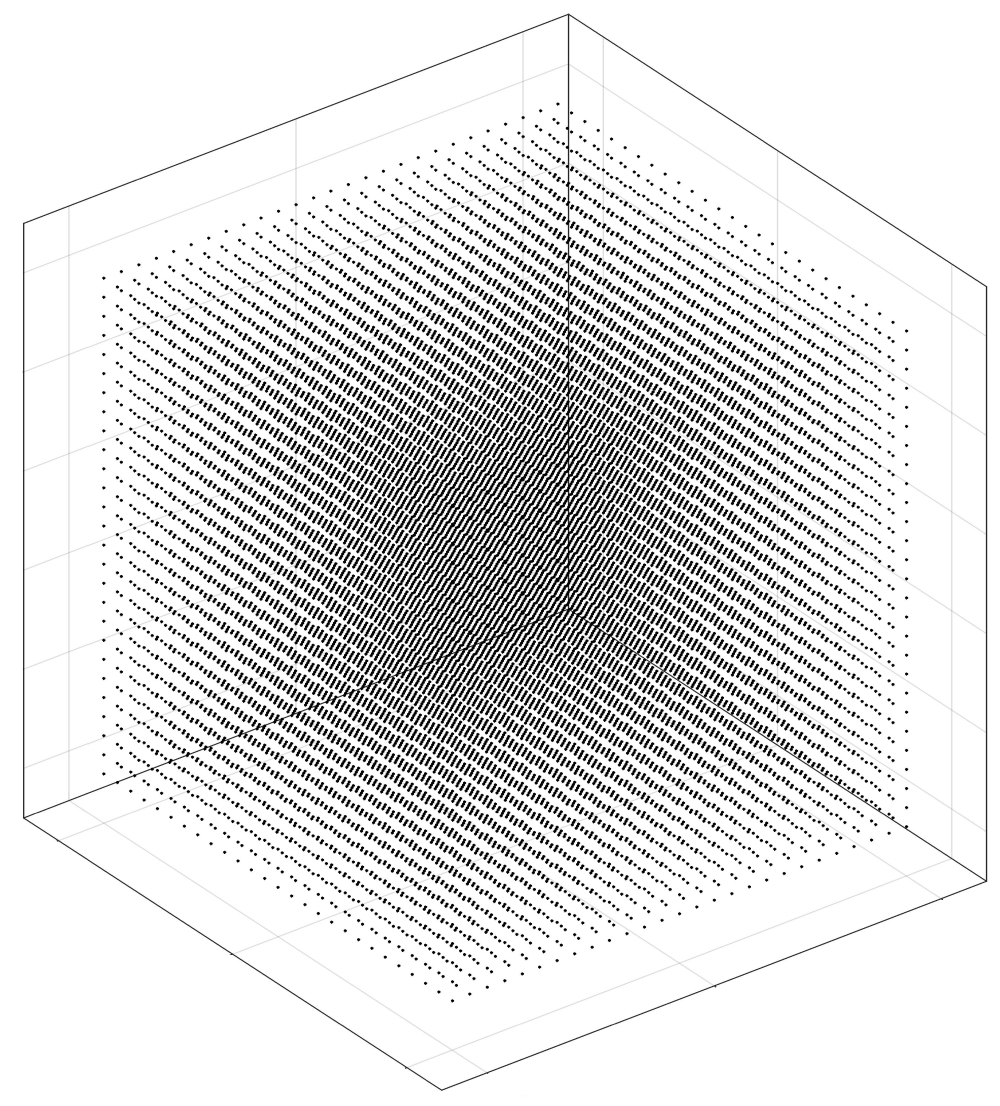} \quad \quad
\includegraphics[width=0.55\textwidth, trim=0cm 0cm 0cm 0cm]{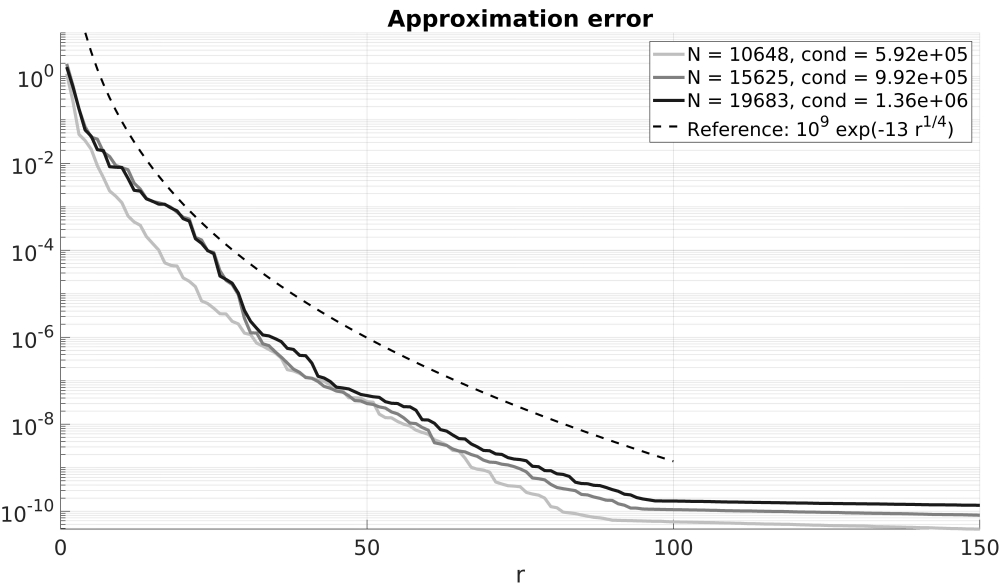}
\caption{A comparison of different problem sizes $N$ for a uniform 3D grid.}
\label{Figure_Exp_3D_uniform}
\end{center}
\end{figure}

The next example, \cref{Figure_Exp_3D_uniform}, covers the case $d=3$ and a uniform point distribution in the unit cube $[0,1] \times [0,1] \times [0,1] \subseteq \R^3$, visualized in the left image. This time, we set $k=2$ and $k_{\min}=0$ and used the Bessel potential $\phi(x) = e^{-\norm{x}}$ from \cref{Fundamental_solution_Ex_1} as the basis function. The error plot shows a comparison between $N \approx 10.000$, $N \approx 15.000$ and $N \approx 20.000$ interpolation points, as well as a reference curve of the form $r \mapsto C \exp(-\CExp r^{1/4})$. In accordance with \cref{Main_result}, the empirical decay rate seems to be independent of the problem size $N$.

\begin{figure}[H]
\begin{center}
\includegraphics[width=0.3\textwidth, trim=0cm 0cm 0cm 0cm]{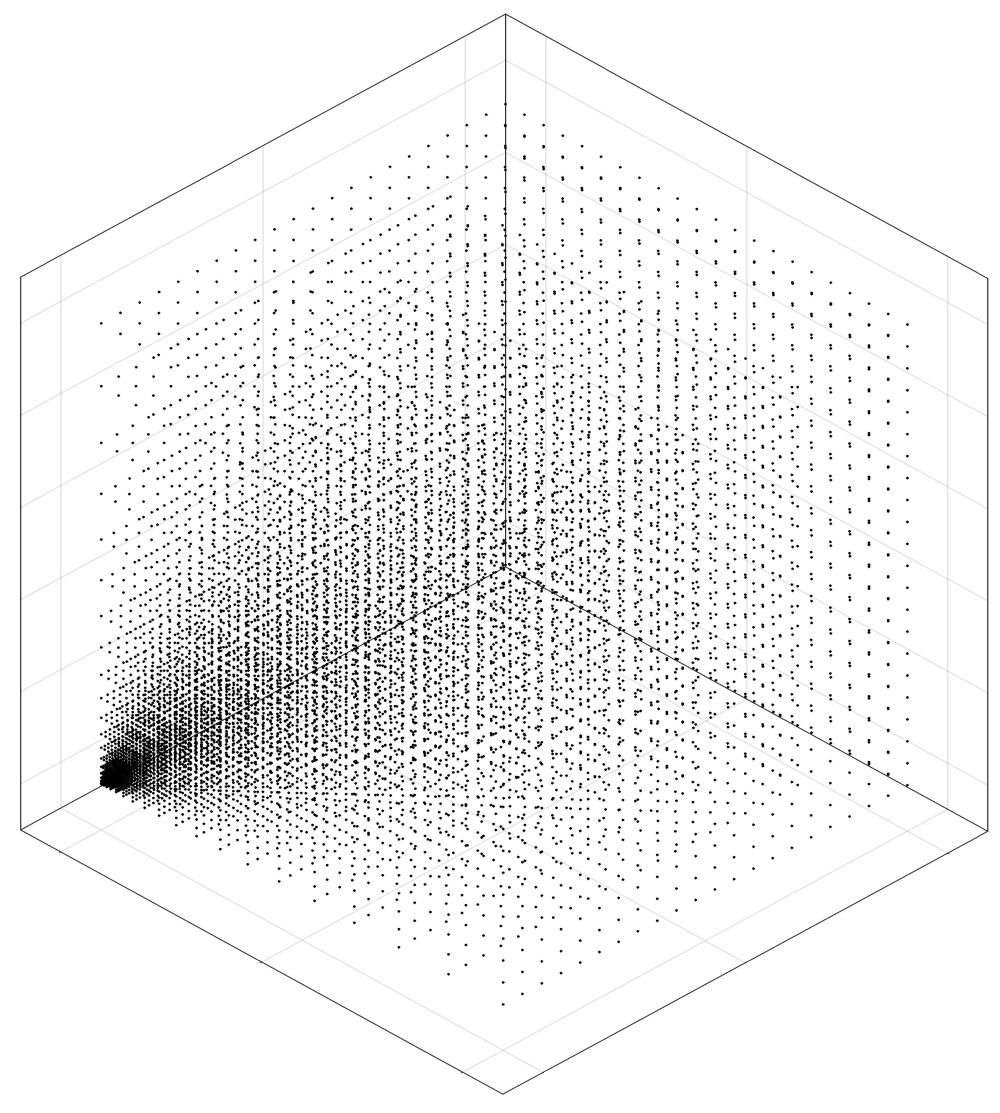} \quad \quad
\includegraphics[width=0.55\textwidth, trim=0cm 0cm 0cm 0cm]{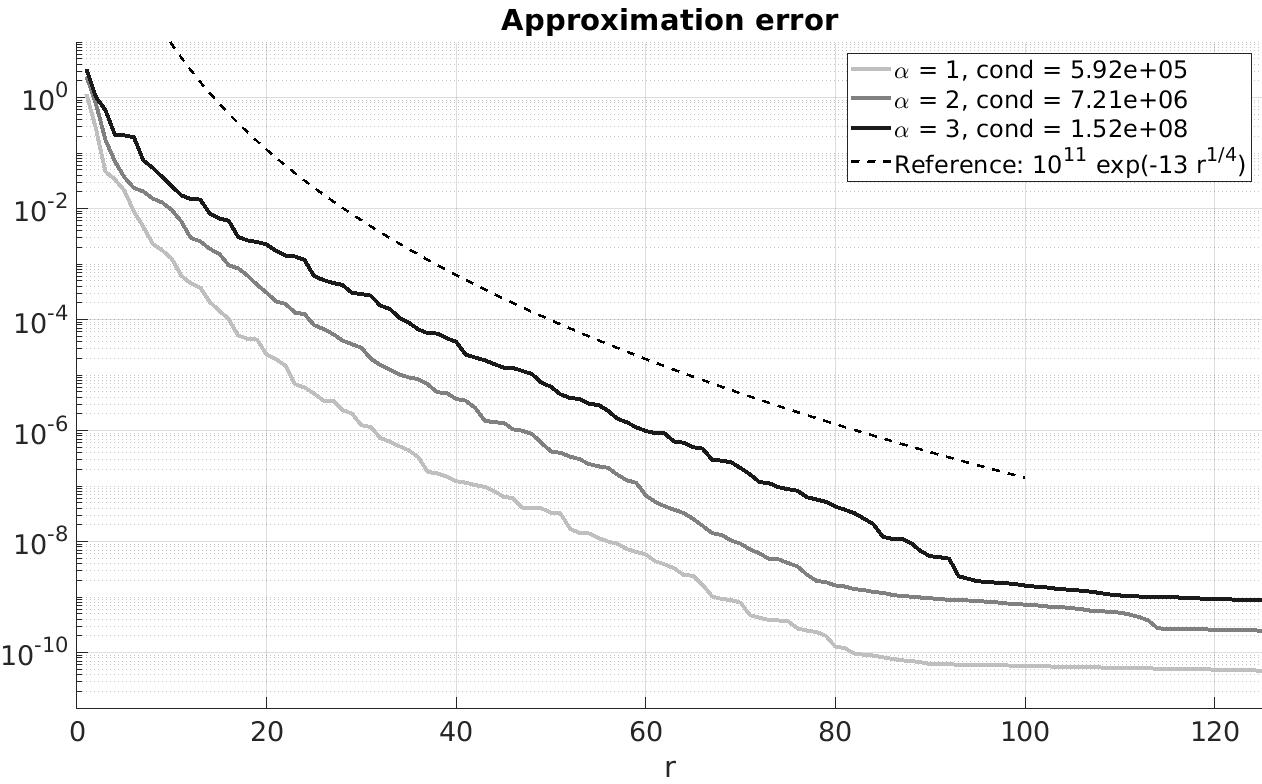}
\caption{Experimenting with an algebraically graded grid in 3D.}
\label{Figure_Exp_3D_graded}
\end{center}
\end{figure}

In \cref{Figure_Exp_3D_graded}, we investigated the influence of the grading exponent $\CCard$ from \cref{Interpol_points} on the error decay rate in $d=3$ space dimensions. We set $k=2$ and $k_{\min}=0$ and used $\phi(x) = e^{-\norm{x}}$ as the basis function. The error plot compares the cases $\CCard \in \set{1,2,3}$, where $\CCard = 1$ is a uniform grid and $\CCard=3$ is ``strongly`` graded towards the origin $0 \in \R^3$. The problem size $N \approx 10.000$ was held constant throughout all three runs. The plot suggests that the constant $\CExp$ from the error bound $\exp(-\CExp r^{1/4})$ in \cref{Main_result} is independent of the grading parameter $\CCard \in \set{1,2,3}$.


Finally, we perform experiments using the $\mathcal{H}$-arithmetic to approximate the inverse system matrix using the library \texttt{H2Lib}. As mentioned in the introduction, previous numerical results have established that the $\mathcal{H}$-matrix arithmetic is a viable tool for solving RBF interpolation problems, \cite{MR3779519,MR4190812,MR3952680}.
In the following, we compute a Cholesky-decomposition in the $\mathcal{H}$-matrix format by approximating the system matrix by an $\mathcal{H}$-matrix with good accuracy and then perform the Cholesky-algorithm using $\mathcal{H}$-arithmetic as described, e.g., in \cite{Bebendorf07}. Afterwards, we project the matrix to a prescribed rank $r$. 
Finally, an approximate inverse can be obtained by inverting the Cholesky factors.

However, a complication arises in the \emph{conditionally positive definite} case of polyharmonic splines as the main block of the system matrix is not invertible so that 
the Cholesky-decomposition cannot be computed. 
In fact, the system matrix has saddle point structure. Like Gaussian elimination, $\mathcal{H}$-matrix inversion of such matrices would require pivoting. 
To avoid pivoting it was advocated in \cite{BL12,MR3952680} to consider instead the augmented Lagrangian $\mvemph{A} + \gamma \mvemph{B}^T \mvemph{B}$ ($\gamma > 0$), which is SPD and therefore amenable to  an $\mathcal{H}$-matrix inversion.

\begin{figure}[H]
\begin{center}
\includegraphics[width=0.45\textwidth, trim=0cm 0cm 0cm 0cm]{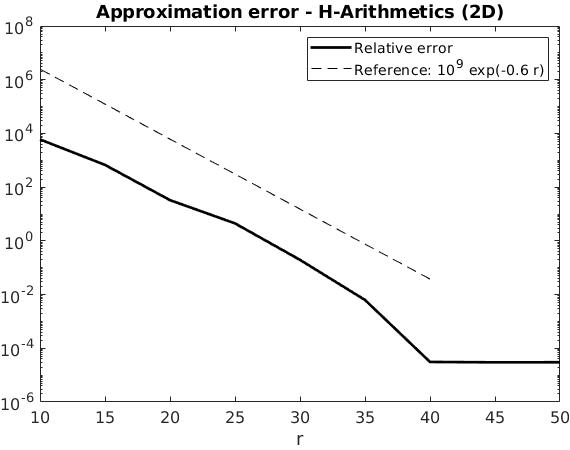} \quad \quad
\includegraphics[width=0.45\textwidth, trim=0cm 0cm 0cm 0cm]{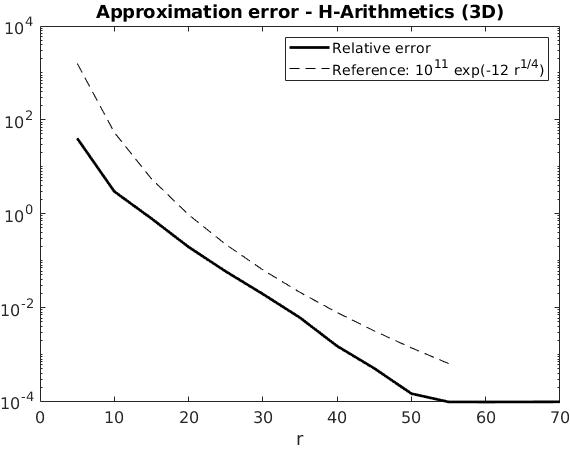}
\caption{Experiment using $\mathcal{H}$-arithmetics to approximate the inverse system matrix, left: 2D thin-plate splines, right: 3D Bessel potential.}
\label{Figure_H_arithmetic}
\end{center}
\end{figure}


In \cref{Figure_H_arithmetic}, we compute an approximate inverse using $\mathcal{H}$-arithmetics as described above for the case of thin-plate splines in space dimension $d=2$ (uniform point distribution in unit square, $N=10000$ $k=k_{\min}=2$, augmented Lagrangian approach with $\gamma=1$) and the Bessel potential in space dimension $d=3$ (uniform point distribution in unit cube, $N=4096$ $k=2$, $k_{\min}=0$).
In contrast to the previous examples performed in \texttt{MATLAB}, we do not compute the exact inverse matrix and perform an SVD to compute an upper bound of the absolute error, but rather use the error measure $\norm[2]{\mvemph{I}-(\mvemph{L}_{\mathcal{H}}\mvemph{U}_{\mathcal{H}})^{-1}\mvemph{A}}$, which is an upper bound for the \emph{relative error}.

Once again, we observe exponential convergence as predicted by our main result \cref{Main_result}. However, we mention that the error flattens out earlier than machine precision due to our method of computation, as the approximation of the initial stiffness matrix using interpolation determines the achievable accuracy.  

\bigskip 

{\bf Acknowledgements.}  NA was funded by the Austrian Science Fund (FWF) Project P 28367 and JMM was
supported by the Austrian Science Fund (FWF) by the special research program Taming complexity in
PDE systems (Grant SFB F65).
\bibliographystyle{amsalpha}
\bibliography{./References}

\end{document}